\documentclass[11pt]{amsart}
\usepackage{amsmath,amsfonts,amsbsy,amsgen,amscd,mathrsfs,amssymb,amsthm}
\usepackage{enumerate,mathtools}
\usepackage{microtype}
\usepackage{tabularx}
\usepackage{array}
\newcolumntype{P}[1]{>{\centering\arraybackslash}p{#1}}
\usepackage{makecell}

\usepackage{stmaryrd}

\usepackage{fullpage}
\usepackage{url}
\usepackage{mathrsfs}
\usepackage{float}
\usepackage{bbm}
\usepackage{quiver}
\usepackage[colorlinks=true,allcolors=blue]{hyperref}

\usepackage[
backend=biber,
style=alphabetic,
maxalphanames=5,
maxnames=5
]{biblatex}
\makeatletter
\renewcommand*{\eqref}[1]{%
  \hyperref[{#1}]{\textup{\tagform@{\ref*{#1}}}}%
}
\makeatother

\usepackage{tikz}
\usetikzlibrary{shadings}

\numberwithin{equation}{section}

\newtheorem{theorem}{Theorem}[section]
\newtheorem{proposition}[theorem]{Proposition}
\newtheorem{question}[theorem]{Question}

\newtheorem{lemma}[theorem]{Lemma}
\theoremstyle{definition}
\newtheorem{definition}[theorem]{Definition}

\newtheorem{example}[theorem]{Example}
\newtheorem{remark}[theorem]{Remark}

\newtheorem*{claim*}{Claim}

\newcommand{\A}{\mathbb{A}}
\newcommand{\C}{\mathbb{C}}
\newcommand{\D}{\mathbb{D}}

\newcommand{\F}{\mathbb{F}}

\newcommand{\Q}{\mathbb{Q}}
\newcommand{\R}{\mathbb{R}}

\newcommand{\Z}{\mathbb{Z}}

\DeclareMathOperator{\CHur}{CHur}

\DeclareMathOperator{\disc}{disc}

\DeclareMathOperator{\End}{End}

\DeclareMathOperator{\Ext}{Ext}

\DeclareMathOperator{\Frob}{Frob}

\DeclareMathOperator{\Gr}{Gr}
\DeclareMathOperator{\Hom}{Hom}
\DeclareMathOperator{\Hur}{Hur}

\DeclareMathOperator{\id}{id}

\DeclareMathOperator{\MHM}{MHM}

\DeclareMathOperator{\Perv}{Perv}

\DeclareMathOperator{\Shv}{Shv}
\DeclareMathOperator{\Spec}{Spec}

\DeclareMathOperator{\Sym}{Sym}
\DeclareMathOperator{\tot}{tot}
\DeclareMathOperator{\Tor}{Tor}
\DeclareMathOperator{\tr}{tr}

\newcommand{\GL}{\operatorname{GL}}

\usepackage[shortlabels]{enumitem}

\usepackage[OT2,T1]{fontenc}
\DeclareSymbolFont{cyrletters}{OT2}{wncyr}{m}{n}
\DeclareMathSymbol{\Sha}{\mathalpha}{cyrletters}{"58}

\usepackage{array, booktabs}

\title{Weight filtration of Hurwitz spaces and quantum shuffle algebras}

\author{Zhao Yu Ma}
\address{Princeton University, Princeton, NJ, USA}
\email{zm5336@princeton.edu}
\begin{document}
\begin{abstract}
We prove an equivalence between filtrations of primitive bialgebras and filtrations of factorizable perverse sheaves, generalizing the results obtained by Kapranov-Schechtman. Under this equivalence, we find that the word length filtration of quantum shuffle algebras as defined in Ellenberg-Tran-Westerland corresponds to the codimension filtration of factorizable perverse sheaves. Furthermore, we find that the geometric weight filtration of factorizable perverse sheaves corresponds to a filtration on quantum shuffle algebras which has not been previously defined in the literature, and we call this the algebraic weight filtration. To apply this to Hurwitz spaces, we prove a comparison theorem between the weight filtrations for Hurwitz spaces over $\F_p$ and $\C$, generalizing the comparison theorem of Ellenberg-Venkatesh-Westerland. This allows us to determine the cohomological weights for Hurwitz spaces explicitly using the algebraic weight filtration of the corresponding quantum shuffle algebra. As a consequence, we find that most weights of Hurwitz spaces are smaller than expected from cohomological degree, and we prove explicit nontrivial upper bounds for weights in some cases, such as when $G=S_3$ and $c$ is the conjugacy class of transpositions.
\end{abstract}

\maketitle

\hypersetup{linktocpage=true}
\tableofcontents
\section{Introduction}
\subsection{Motivation and background}
The geometric approach to arithmetic statistics over function fields by analyzing the cohomology of Hurwitz spaces has led to many breakthroughs. The method was first introduced in \cite{EVW16,ETW17}, and later \cite{landesman2025homologicalstabilityhurwitzspaces,landesman2025cohenlenstramomentsfunctionfields,landesman2025stablehomologyhurwitzmodules} used homological stability methods to prove many conjectures in different settings. Other relevant work on Hurwitz spaces include \cite{Bianchi_Miller_2024,himes2025representationstabilityorderedhurwitz} which prove polynomial stability and representation stability, \cite{Türkelli_2015,Wood_2021,Seguin_2025} on the number of connected components, \cite{Deopurkar_Patel_2015} on the vanishing of the rational Picard group for degrees $d\le 5$, and \cite{Zheng_2023} on stable cohomology for $d=3$.

A rough outline of the strategy is to construct a family of smooth Hurwitz spaces $X_n$ indexed by $n=\dim(X_n)$ which parametrizes the arithmetic objects of interest. We are interested in understanding the asymptotic counts $\#X_n(\F_p)$ as $n\rightarrow \infty$. By the Grothendieck-Lefschetz trace formula, this count can be expressed as an alternating sum
$$\#X_n(\F_p)=\sum_{i=0}^{2n}(-1)^i \tr\left(\Frob_q, H^{2n-i}_c((X_{n})_{\overline \F_p},\overline\Q_\ell)\right)$$
reducing questions of point counts to understanding the trace of Frobenius on the étale cohomology of $X_n$. Since $X_n$ is smooth, we have by Poincaré Duality that $H^{2n-i}_c(X_n,\overline\Q_\ell)\cong H^i(X_n,\overline\Q_\ell)^\vee (-n)$ where we took the dual and applied a Tate twist. Deligne's theory of weights says that the absolute value of the eigenvalues of Frobenius on $H^{2n-i}_c(X_{n},\overline\Q_\ell)$ is not more than $q^{(2n-i)/2}$ for any embedding $\overline\Q_\ell\hookrightarrow \C$. The dominant term usually comes from $H^{2n}_c(X_n,\overline\Q_\ell)$ which has size on the order of $q^n$. To bound the contribution from the other terms, it suffices to show a homological stability result that $H^{i}(X_n,\overline\Q_\ell)$ vanishes for a range of $1\le i\le \epsilon n$ for some constant $\epsilon$, as well as a bound for the Betti numbers $\dim H^{i}(X_n,\overline\Q_\ell)$ for $i>\epsilon n$. 

Using comparison theorems, one can show that this cohomology is the singular cohomology of Hurwitz space $X_n(\C)$ which we view as a complex manifold, reducing this to a problem in algebraic topology. By looking at the cell decomposition of the topological Hurwitz space, \cite{ETW17} observes that the associated chain complex actually comes from a quantum shuffle algebra $\mathfrak A$ graded in $n$, where the singular cohomology is exactly the Ext-cohomology of $\mathfrak A$. This perspective is helpful as it gives a concrete way to analyze the cohomology using bar-complexes and allows us to build upon ongoing work in quantum algebras. 

This geometric approach has proven to be very powerful in understanding the \textit{dominant} terms of very general arithmetic statistics problems over function fields, while comparatively arithmetic methods based on the geometry of numbers or Shintani zeta functions are limited only to very specific cases, such as counting fields of degree $\le 5$. However, geometric methods currently fall short of arithmetic methods when it comes to understanding \textit{secondary terms} of these asymptotic counts. Nevertheless, the hope is that a refinement of this geometric approach could lead to proofs of smaller order terms in very general cases. The current problem is that we have a limited understanding of the cohomology groups after the stable range and how Frobenius acts of them. 

This paper aims to be a starting point for understanding on how Frobenius acts on these middle cohomology groups by investigating the \textit{weights} of these eigenvalues. Deligne's theory of weights actually says something stronger -- for each eigenvalue of Frobenius on $H^{2n-i}_c(X_n,\overline\Q_\ell)$ there is an integer $w\le 2n-i$ called the weight such that the eigenvalue has absolute value $q^{w/2}$ for any embedding $\overline\Q_\ell\hookrightarrow \C$. One could then ask the following basic question:
\begin{question}
For Hurwitz spaces $X_n$, are most of the weights $w$ of $H^{2n-i}_c(X_n,\overline\Q_\ell)$ what we expect from cohomology degree ($w=2n-i$), or are most of them smaller than what we expect ($w<2n-i$)?
\end{question}
We are not aware of any prior work on the weights of cohomology of Hurwitz spaces. However, there have been some results towards the analogous problem of understanding the cohomological weights of $\mathcal M_g$, for example in \cite{Chan_Galatius_Payne_2021} and \cite{Bergström_Faber_Payne_2024}.

To answer the question above, we give two theorems as evidence that weights are often much smaller than what we expect from cohomological degree, namely Theorem \ref{thm: finite nichols algebra} and \ref{thm: concentration of weight in algebra}. The first tells us that for a specific class of Hurwitz spaces, there is some linear range of cohomology with weight strictly smaller than expected. The second applies to all Hurwitz spaces in general, and roughly says that the weight, in the sense of algebras, is concentrated below what one might expect. We leave the discussion of Theorem \ref{thm: concentration of weight in algebra} to the later part of the introduction and discuss Theorem \ref{thm: finite nichols algebra} first. 

Recall from the discussion above that a family of Hurwitz spaces $X_n$ has a corresponding quantum shuffle algebra $\mathfrak A$ with the same cohomology (see also Theorem \ref{thm: KS20} and the discussion after). 
\begin{theorem}\label{thm: finite nichols algebra}
Let $X_n$ be a family of Hurwitz spaces with corresponding quantum shuffle algebra $\mathfrak A$ such that its Nichols subalgebra $\mathfrak B$ (generated by degree $1$ elements) is finite-dimensional and has finitely generated cohomology. Then, there is some $c<1$ and $\epsilon>0$ such that for all $i>cn$, any weight $w$ appearing in $H^{2n-i}_c(X_n,\overline\Q_\ell)$ satisfies
$$w\le (2n-i)-\epsilon(i-cn).$$
\end{theorem}
We remark that the constants $\epsilon$ and $c$ are explicit, and they are given as in Lemma \ref{lem: B finite dim}, \ref{lem: Ext B finite gen} with $\epsilon = \delta/(1-c)$. Many simple arithmetic situations satisfy the condition of Theorem \ref{thm: finite nichols algebra}. For concreteness, we explain this in the case of counting cubic fields as there are already good results from arithmetic methods that could give hints to the homological interpretations of secondary or smaller order terms. 
\subsection{Counting cubic fields}\label{sec: cubic fields}
Secondary terms in cubic extensions over $\Q$ were first discovered in \cite{TT13} with the error term improved in \cite{BTT23}. 
Let $N_3^\pm(X)$ denote the number of isomorphism classes of cubic fields $F$ over $\Q$ satisfying $0<\pm \disc(F)<X$, then it was proven that
$$N_3^\pm(X) = C_1^\pm X+C_2^\pm X^{5/6}+O(X^{2/3+\epsilon})$$
where 
$$C_1^\pm  = \frac{1}{12\zeta(3)}\begin{cases}
1 & \text{ if }+\\
3 & \text{ if }-
\end{cases}\ , \quad C_2^\pm = \frac{4\zeta(1/3)}{5\Gamma(2/3)^3\zeta(5/3)} \begin{cases}
1 & \text{ if }+\\
\sqrt 3 & \text{ if }-
\end{cases}\ .$$

In the case of function fields, finite separable degree $3$ extensions $K/\F_p(t)$ correspond to connected degree $3$ branched covers $f\colon C\rightarrow \mathbb P^1$. By Riemann-Hurwitz, the degree of the branch divisor is even, so the discriminant is of the form $q^{2N}$. If we write $N=3m+k$ for $k\in \{0,1,2\}$, from the work of \cite{Zhao13,Kural25,Ahlquist25}, the number of degree $3$ extensions with discriminant $q^{2N}$ is
$$N_3(2N)=C_1q^{2N}-C_2^kq^{5m}+O(q^{4N/3+\epsilon})$$
where
$$C_1=(1+q^{-3})(1+q^{-1}), \quad C_2^k=(q^{-1}-q^{-3})\begin{cases}q+1 & \text{ if }k=0\\
q^3+q^2 & \text{ if }k=1\\
q^4 & \text{ if }k=2
\end{cases}.$$

We can parametrize these degree $3$ covers with an algebraic stack $Y_{2N}$ with connected components which come from various Hurwitz spaces. For simplicity, we will discuss only one of these Hurwitz spaces. Some connected components of $Y_{2N}$ appear as connected components of the quotient stack $X_{2N}/G$ (see Section \ref{sec: alg Hurwitz spaces} for more details), where here we let $X_n\coloneqq \Hur^{G,c}_{n}$ for $G=S_3$ and $c$ the conjugacy class of transpositions. We state Theorem \ref{thm: finite nichols algebra} for $X_n$ with explicit constants.
\begin{theorem}\label{thm: S3 transposition}
Let $X_n$ be defined as above. For $i>2\lfloor \frac n 6\rfloor$ the weights $w$ appearing in $H^{2n-i}_c(X_n,\overline\Q_\ell)$ are strictly smaller than $2n-i$. We also have the inequality $$w\le (2n-i)-\frac{3(i-n/3)}{10}.$$
\end{theorem}
Let us first discuss this in the context of the error term $O(q^{4N/3+\epsilon})$. Consider the connected components of $Y_{2N}$ that come from $X_{2N}/G$, and eigenvalues on from the cohomology of these connected components. A priori, without our theorem, an eigenvalue on $H^{4N-i}_c(X_{2N}/G,\overline\Q_\ell)$ with $i\ge\frac{4N}{3}$ has weight $w\le \frac{8N}{3}$ so the absolute value of the eigenvalue is at most $q^{4N/3}$ which has the same order of the error term. However, with Theorem \ref{thm: S3 transposition}, our bound on the weight of such an eigenvalue improves to $w\le \frac{37N}{15}$ so the absolute value of the eigenvalue is at most $q^{37N/30}$ which is a factor of $q^{N/10}$ smaller than the error term. This suggests that weight considerations could be useful for strengthening error bounds in homological methods.

We can also discuss this in the context of the secondary terms, where the improvement is less obvious. Theorem \ref{thm: S3 transposition} tells us that an eigenvalue coming from $H^{4N-i}_c(X_{2N}/G,\overline\Q_\ell)$ which has weight as expected can only come from cohomological degree $i\le 2\lfloor \frac N 3\rfloor$ so the absolute value of such an eigenvalue is at least $q^{2N-\lfloor N/3\rfloor}$. We compare this to the order of magnitudes of the secondary term in the following table.
\begin{table}[H]
\centering
\begin{tabular}{>{\centering\arraybackslash}m{7cm}
                >{\centering\arraybackslash}m{2.2cm}
                >{\centering\arraybackslash}m{2.2cm}
                >{\centering\arraybackslash}m{2.2cm}}
\toprule
\textbf{Discriminant $q^{2N}$} & $\boldsymbol{q^{6m}}$ & $\boldsymbol{q^{6m+2}}$ & $\boldsymbol{q^{6m+4}}$ \\
\midrule
\midrule
\textbf{Order of secondary term}
 & $q^{5m}$ & $q^{5m+2}$ & $q^{5m+3}$ \\
 \midrule
\textbf{Smallest eigenvalue with weight as expected from $X_{2N}$}
 & $q^{5m}$ & $q^{5m+2}$ & $q^{5m+4}$ \\
\bottomrule
\end{tabular}
\end{table}

This means that in the case when the discriminant is of the form $q^{6m+4}$, if an eigenvalue of Frobenius from such a connected component has the same order as the secondary term, then its weight must be strictly smaller than expected from cohomological degree.
\subsection{Equivalence of filtrations} Our main strategy is to use the bridge between Hurwitz spaces and quantum shuffle algebras to transport weight information from topology to algebra. This was first developed in \cite{ETW17}, but the connection developed there is only on the level of chain complexes and does not allow us to carry weight information from one side to another. Instead, we build on \cite{KS20} which proves a more general equivalence with more structure. For this paper, let the coefficient field $k$ be an algebraically closed field of characteristic zero.
\begin{theorem}[{\cite[Theorem 3.3.1]{KS20}}]\label{thm: KS20}
Let $\mathcal V$ be a braided monoidal abelian category over a field $k$ with $\otimes$ bi-exact. Then, there is an equivalence between primitive bialgebras and factorizable perverse shaves over $\mathcal V$
$$\text{PB}(\mathcal V)\xleftrightarrow{\ \sim\ } \text{FPS}(\mathcal V).$$
Furthermore, if $A\in \text{PB}(\mathcal V)$ corresponds to $\mathcal F\in \text{FPS}(\mathcal V)$, then there is an isomorphism of cohomologies $R\Gamma^j\mathcal F_n\cong \Tor^A_{-j,n}(k,k)$.
\end{theorem} 
We specialize to the case where $\mathcal V$ is a braided monoidal abelian category of $k$-vector spaces. Here, the primitive bialgebras $\text{PB}(\mathcal V)$ are a certain class of graded Hopf algebras over $\mathcal V$ equipped with algebra and coalgebra structures, and this includes the tensor algebra $T(V)$, quantum shuffle algebra $\mathfrak A(V)$ and Nichols algebra $\mathfrak B(V)$ for all $V\in \mathcal V$. On the other side of the equivalence, factorizable perverse sheaves $\text{FPS}(\mathcal V)$ are a collection of perverse sheaves $K_n$ on $\Sym^n(\C)$ with coefficients in $\mathcal V$ for each $n$, constructible with respect to the diagonal stratification, where factorizability roughly says that $K_{n+m}$ is compatible with $K_n$ and $K_m$ along the map $\Sym^n(U)\times \Sym^m(V)\rightarrow \Sym^{n+m}(U\sqcup V)$ for disjoint $U,V\subseteq \C$.

We explain how this generalizes the equivalence in \cite{ETW17}. By our definition, Hurwitz spaces $X_n$ are finite étale covers of $\Sym^n_{\neq} (\C)$, so a compatible family of Hurwitz spaces defines a collection of local systems $\mathcal L=\{\mathcal L_n\}_n$ via pushforward of the constant sheaf. The cohomology of these local systems can be evaluated by taking $R\Gamma j_{n*}\mathcal L_n$ where $j_n\colon \Sym^n_{\neq}(\C)\hookrightarrow \Sym^n(\C)$. \cite[Theorem 3.3.3]{KS20} proves that under the equivalence described above, the factorizable perverse sheaf $j_*\mathcal L=\{j_{n*}\mathcal L_n\}_n$ corresponds to a quantum shuffle algebra $\mathfrak A(V)$. The theorem also tells us that the cohomologies are equal which is exactly \cite[Theorem 1.3]{ETW17}.

Using the theory of mixed Hodge modules, if each $\mathcal L_n$ has finite monodromy, then we can associate an increasing weight filtration 
$$\cdots \subseteq W_0 j_{n*}\mathcal L_n\subseteq W_1 j_{n*}\mathcal L_n\subseteq \cdots $$ 
to the perverse sheaves $j_{n*}\mathcal L_n$ where the graded piece $\Gr^W_w j_{n*}\mathcal L_n$ is pure of weight $w$. Here, we normalized our local systems $\mathcal L_n$ so that they are pure of weight $0$. We wish to transport this weight filtration to the algebra side. Motivated by this, we prove the following theorem which is the core of our paper.
\begin{theorem}\label{thm: equivalence of filtrations}
Let $\mathcal V$ be a braided monoidal abelian category of $k$-vector spaces. Then, we have an equivalence
$$
\left\{
\text{Filtrations of }
A \in \text{PB}(\mathcal V)
\right\}
\xleftrightarrow{\ \sim\ }
\left\{
\text{Filtrations of }
\mathcal F \in \text{FPS}(\mathcal V)
\right\}
$$
where $A$ and $\mathcal F$ correspond under Theorem \ref{thm: KS20}. Furthermore, there is an isomorphism of spectral sequences between $R\Gamma^j\mathcal F^{gr}_n \Rightarrow R\Gamma^j\mathcal F_n$ and $\Tor^{A^{gr}}_{-j,n}(k,k)\Rightarrow \Tor^{A}_{-j,n}(k,k)$.
\end{theorem}

Our theorem is not a direct consequence of Theorem \ref{thm: KS20}. One could naively think that we can we can somehow apply the equivalence to each graded part or filtered parts, but the problem is that neither the graded nor filtered parts are themselves primitive bialgebras or factorizable perverse sheaves. 

Instead, our proof is based on the philosophy that filtrations of $k$-vector spaces are equivalent to free graded $k[t]$ modules. We construct a braided monoidal abelian category $\mathcal W$ of finitely generated graded $k[t]$-modules which have objects finite $\Z$-indexed diagrams $\{\cdots \rightarrow V_0\rightarrow V_1\rightarrow \cdots\}$ with $V_i\in \mathcal V$. Then, we prove the following diagram of equivalences.
\begin{equation}\label{eqn: diagram for equivalence proof}
\begin{tikzcd}
	\begin{array}{c} \left\{ B\in \text{PB}(\mathcal W) \text{ free over } k[t]\right\} \end{array} & \begin{array}{c} \left\{ \mathcal G\in \text{FPS}(\mathcal W) \text{ with }\Phi_{\tot}(\mathcal G_n)\text{ free over }k[t] \right\} \end{array} \\
	\begin{array}{c} \left\{ \text{Filtrations of } A \in \text{PB}(\mathcal V)\right\} \end{array} & \begin{array}{c} \left\{ \text{Filtrations of } \mathcal F \in \text{FPS}(\mathcal V) \right\} \end{array}
	\arrow["b"', tail reversed, from=1-1, to=1-2]
	\arrow["c"', tail reversed, from=1-2, to=2-2]
	\arrow["a"', tail reversed, from=2-1, to=1-1]
	\arrow[tail reversed, from=2-1, to=2-2]
\end{tikzcd}
\end{equation}
Here, (a) and (c) are essentially extensions of the philosophy above. While (a) is easy to prove, (c) is significantly more difficult because we are working with perverse sheaves over $k[t]$. These are not well behaved because $\otimes$ is not exact, for example the dual $\mathbb D$ or exterior tensor product $\boxtimes$ of two perverse sheaves may fail to be perverse. For (b), we want to apply Theorem \ref{thm: KS20} to $\mathcal W$ but again the problem is that $\otimes$ is not bi-exact -- in fact, Theorem \ref{thm: KS20} is false without this assumption. Nevertheless, we modify the proof essentially by deriving all tensor products and prove that (b) still holds because of the freeness condition on both sides. Lastly, the isomorphism of spectral sequences will follow from the isomorphism of cohomologies in (b).
\subsection{Comparing filtrations}
We apply Theorem \ref{thm: equivalence of filtrations} to compare filtrations on the geometric side and the algebra side. We restrict to the equivalence between filtrations of pushforward of compatible local systems $j_*\mathcal L\in \text{FPS}(\mathcal V)$ and filtrations of quantum shuffle algebras $\mathfrak A(V)\in \text{PB}(\mathcal V)$ as this is the case most relevant to Hurwitz spaces. This is also dual to the equivalence between $j_!\mathcal L$ and the tensor algebra $T(V)$. We make the following two comparisons.

First, we consider the word length filtration in \cite{ETW17} which is the decreasing filtration denoted by $F^k(\mathfrak A)=(\mathfrak A_{>0})^k$. We find that it is more natural geometrically to consider the shifted word length filtration defined by on each algebraic degree $n$ by $F_c(\mathfrak A_n)=F^{n-c}(\mathfrak A_n)$ which is now an increasing filtration. Under Theorem \ref{thm: equivalence of filtrations}, we find that this corresponds to the codimension filtration on $j_*\mathcal L$, defined such that the graded pieces $\Gr_cj_*\mathcal L$ consists of middle extensions from codimension $c$ strata, in particular we have $\Gr_0j_*\mathcal L=F_0j_*\mathcal L=j_{!*}\mathcal L$. This is not surprising because we expect elements in $A_{\lambda_1}\times \cdots A_{\lambda_r}$ to come from codimension $\lambda_1+\cdots +\lambda_r-r$. We also mention a very slight generalization which is that any filtration on the stratification of $\Sym^n(\C)$ induces a filtration of $j_*\mathcal L$ which corresponds to a weighted word length filtration on $\mathfrak A$.

Next, we discuss the weight filtration. Suppose that $\mathcal L_n$ has finite monodromy, so we can define the geometric weight filtration on $j_{*}\mathcal L$ using mixed Hodge modules. We find that it corresponds to an increasing filtration $W_w\mathfrak A$ on the quantum shuffle algebra which we call the algebraic weight filtration. To the best of our knowledge this has not appeared in the literature before, and we define it via the dual tensor algebra $T=T(V^*)=\bigoplus_{n=0}^\infty (V^*)^{\otimes n}$ as follows.
\begin{definition}\label{defn: alg weight filtration}
The \textit{algebraic weight filtration} $W_wT_n$ on the dual tensor algebra $T$ is defined inductively on $n$, where for $n=0,1$ it is given by the trivial filtration concentrated in weight $0$, i.e. $W_0T_0=T_0$,  $W_0T_1=T_1$ and $W_{-1}T_0=W_{-1}T_1=\{0\}$. For $n\ge 2$, given the filtration on $T_p$ for $p<n$, there is an induced filtration on 
$$S=\bigoplus_{\substack{p+q=n\\0<p,q<n}} T_p\otimes T_q\subseteq (T\otimes T)_n$$
and the filtration on $T_n$ is defined by
$$W_wT_n= \mu (W_wS)+ \left(\mu(W_{w+1}S)\cap P_n\right),$$
where $P_n\subseteq T_n$ are the primitive elements in $T$, i.e. those that satisfy $\Delta_\star(x)=1\otimes x + x\otimes 1$. 

From this, we obtain the \textit{algebraic weight filtration} $W_w\mathfrak A_n$ on the quantum shuffle algebra $\mathfrak A_n$ by $W_w\mathfrak A_n=(W_{-w-1}T_n)^\perp$ so that the graded piece $\Gr_w\mathfrak A_n$ is dual to $\Gr_{-w}T_n$.
\end{definition}
Intuitively, the weight filtration on $W_wT_n$ respects the multiplication $\mu$ on $T$, except it goes down by $1$ on primitive elements. From this description, we can see that the weight zero part $W_0\mathfrak A$ of the quantum shuffle algebra is the Nichols subalgebra $\mathfrak B$ generated by degree $1$ elements. We will show that the algebraic weight filtration is indeed a bialgebra filtration under the assumption that $\mathcal L_n$ has finite monodromy, but we suspect that this may not be true in the general case.

Our method for comparing filtrations across Theorem \ref{thm: equivalence of filtrations} is a spectral sequence argument. The idea is to split $j_{n*}\mathcal L_n$ into two parts, the first part from the origin and the second part an intermediate extension from outside the origin. Using factorizability we can understand the second part from lower degrees. Then, constraints on the differentials of a spectral sequence helps us understand the first part from the second. For example, for the weight filtration we crucially rely on the property that the only nontrivial differentials are those between complexes which differ in weight by $1$.
\subsection{Consequences of the weight filtration} Now that we have related the geometric weight filtration on factorizable perverse sheaves over $\C$ with the algebraic weight filtration on quantum shuffle algebras, we need a comparison theorem for the geometric weight filtration over $\F_p$ and $\C$ in order for our results to be relevant for counting $\F_p$-points via the Grothendieck-Lefschetz trace formula. We specialize to the case of Hurwitz spaces and prove Theorem \ref{thm: comparison theorem} which shows that the cohomologies of the weight filtration over $\F_p$ and $\C$ are isomorphic, generalizing \cite[Proposition 7.7, 7.8]{EVW16} which compares the cohomology of the Hurwitz space over $\F_p$ and $\C$. The proof of Theorem \ref{thm: comparison theorem} uses the normal crossings compactification of Hurwitz space that was constructed in \cite[Appendix B]{ellenberg2025homologicalstabilitygeneralizedhurwitz} by Dori Bejleri and Aaron Landesman. This allows us to construct the weight filtration over $\Z_p$ and use a vanishing cycle argument to finish the proof.

This comparison theorem, along with the equivalence of the geometric and algebraic weight filtrations discussed above, enable us to compute the cohomological weights of Hurwitz spaces explicitly from the algebraic weight filtration on the corresponding quantum shuffle algebra as follows. Let $X_n$ be a family of Hurwitz spaces with corresponding quantum shuffle algebra $\mathfrak A$. We can decompose 
$$H^{2n-i}_c((X_n)_{\overline\F_p},\overline\Q_\ell)=\bigoplus_{w\in \Z} \Gr^w_WH^{2n-i}_c((X_n)_{\overline\F_p},\overline\Q_\ell)$$
into a direct sum of weight spaces where Frobenius acts by $q^{w/2}$ on the $w$-th graded piece. On the other hand, the algebraic weight filtration $W_w\mathfrak A_n$ induces a similar weight decomposition on $\Tor^{\mathfrak A}_{-j,n}(k,k)$ by the spectral sequence $\Tor^{\mathfrak A^{gr}}_{-j,n}(k,k)\Rightarrow \Tor^{\mathfrak A}_{-j,n}(k,k)$ in Theorem \ref{thm: equivalence of filtrations}. Here, $\mathfrak A^{gr}$ is the associated graded of $\mathfrak A$ with respect to the weight filtration, and this is bigraded with respect to the algebra grading $n$ and weight grading $w$. Taking the dual \textit{without} flipping the sign of the weight grading $w$ gives the decomposition of Ext-cohomology $$\Ext_{\mathfrak A}^{-j,n}(k,k)=\bigoplus_{w\in \Z} \Gr^w_W \Ext_{\mathfrak A}^{-j,n}(k,k).$$

Then, we have the following theorem that relates the two decompositions.
\begin{theorem}\label{thm: relate weight decompositions}
Let $p>|G|$ and $\ell>n$ be primes with $p\neq \ell$. The isomorphism of vector spaces $$H^{2n-i}_c((X_n)_{\overline\F_p},\overline\Q_\ell)\cong \Ext_{\mathfrak A}^{n-i,n}(k,k)$$ coming from \cite[Theorem 1.3]{ETW17} induces an isomorphism on the weight subspaces
$$\Gr^{2n-i-w}_WH^{2n-i}_c((X_n)_{\overline\F_p},\overline\Q_\ell)\cong \Gr^w_W \Ext_{\mathfrak A}^{n-i,n}(k,k)$$
\end{theorem}

Note that there is a shift in weights between $2n-i-w$ and $w$ above. This is because the weights on the algebra side come from the weights of perverse sheaves, and because these are objects in the derived category the weight here is shifted by cohomological degree. It is important to keep track of which weights are derived weights (all weights on algebras and perverse sheaves), and which are Frobenius weights.

Our next goal is to understand the distribution of weights in $\mathfrak A_n$ for a general quantum shuffle algebra. Since the weight filtration can be defined algebraically, this is a purely algebraic question even though it has consequences for Hurwitz spaces. We find that as $n\rightarrow \infty$, most of the weights are concentrated around some $cn$, as shown in the following theorem. 
\begin{theorem}\label{thm: concentration of weight in algebra}
Let $W_w\mathfrak A$ be the weight filtration on the quantum shuffle algebra $\mathfrak A(V)$. Then, there exists some constant $0\le c\le 1$ where for any $c^-<c<c^+$, we have
$$\frac{\dim W_{c^-n}\mathfrak A_n}{\dim\mathfrak A_n}\rightarrow 0, \qquad \frac{\dim W_{c^+n}\mathfrak A_n}{\dim\mathfrak A_n}\rightarrow 1$$
as $n\rightarrow \infty$. Furthermore, $c=0$ if and only if $\mathfrak A$ is generated in degree $1$, i.e. $\mathfrak A(V)=\mathfrak B(V)$.
\end{theorem}
One can ask whether an analogous statement of concentration of weights in cohomology is true.
\begin{question}
Does there exist some constant $0\le c\le 1$, not necessarily the same as the one in Theorem \ref{thm: concentration of weight in algebra}, such that for any $c^-<c<c^+$ we have
$$\frac{\dim W_{c^-n}\Ext^n_{\mathfrak A_n}(k,k)}{\dim\Ext^n_{\mathfrak A_n}(k,k)}\rightarrow 0, \qquad \frac{\dim W_{c^+n}\Ext^n_{\mathfrak A_n}(k,k)}{\dim\Ext^n_{\mathfrak A_n}(k,k)}\rightarrow 1$$
as $n\rightarrow \infty$?
\end{question}
This cohomological statement would have the following geometric consequence. Let $X_n$ be a family of Hurwitz spaces which correspond to some quantum shuffle algebra $\mathfrak A(V)$. Then as $n\rightarrow \infty$, almost all weights of Frobenius $w$ on $H^*_c(X_n,\overline\Q_\ell)$ satisfy $w\in [2n-i-c^+n,2n-i-c^-n]$ where $2n-i$ is the cohomological degree. 

The key problem in deducing concentration of weight for cohomology from Theorem \ref{thm: concentration of weight in algebra} is that while we have control over the weights of the bar-complex, large cancellations may happen when taking cohomology for weights in the concentrated range while little to no cancellations occur for weights outside the concentrated range. 

Finally, we discuss the proof of Theorem \ref{thm: finite nichols algebra}. We follow the steps of \cite{ETW17} except now in the context of the weight filtration instead of the word length filtration. Recall that $\mathfrak A^{gr}$ is the associated graded of $\mathfrak A$ with respect to the weight filtration which is bigraded in $n$ and $w$, and that the Nichols algebra $\mathfrak B\subseteq \mathfrak A^{gr}$ is the weight zero part. Define $\mathfrak C\coloneqq \mathfrak A^{gr}\Box_{\mathfrak B}k$ to be the cotensor product. Then, we have the three spectral sequences
\begin{equation}\label{eqn: triple spectral sequences}
\Ext_{\mathfrak B}(k,k)\otimes \Ext_{\mathfrak C}(k,k) \Rightarrow \Ext_{\mathfrak B}(k,\Ext_{\mathfrak C}(k,k))\Rightarrow \Ext_{\mathfrak A^{gr}}(k,k)\Rightarrow \Ext_{\mathfrak A}(k,k)
\end{equation}
which respect the algebra, homological and weight gradings. This gives the weight decomposition of Ext-cohomology which is what we want by Theorem \ref{thm: relate weight decompositions}.

Suppose that $\mathfrak B$ is finite-dimensional and that $\Ext_{\mathfrak B}(k,k)$ is finitely generated. It is conjectured that the first implies the latter, and this was proven in various cases such as for Nichols algebra of diagonal type \cite{PointedHopf} and for the Fomin-Kirillov algebra $\mathcal{FK}_3$ \cite{FK3} which corresponds to the case of cubic fields. The finite-dimensionality of $\mathfrak B$ will tell us that on $\mathfrak C$ the weight grading $w\ge \delta n$ increases linearly in algebraic degree. Meanwhile, the finite generation of $\Ext_{\mathfrak B}(k,k)$ tells us that its homological degree $i\ge cn$ increases linearly in terms of the algebraic degree. Putting these two inequalities together with the chain of spectral sequences gives us Theorem \ref{thm: finite nichols algebra}.
\subsection{Acknowledgments}
I would like to thank my advisor Will Sawin for suggesting this problem, and for providing me much valuable guidance along the way. I would also like to thank Jordan Ellenberg, Aaron Landesman and Craig Westerland for their helpful comments. I am especially grateful to Aaron for providing detailed comments which led to significant improvements.
\section{Preliminaries}
\subsection{Hurwitz spaces} Let $G$ be a group and $c$ be a union of conjugacy classes inside $G$. In this section, we discuss the topological definition of the Hurwitz spaces $\Hur^{G,c}_n$, then explain how they can also be defined algebraically. 
\subsubsection{Topological Hurwitz spaces} We follow \cite[Section 2]{EVW16}, except that we look at configuration space over $\C$ rather than over a disk $D$, which is the same as they are homeomorphic to one another. First, we introduce the diagonal stratification $S_n=\{S_\lambda\}_\lambda$ on $\Sym^n(\C)$ with strata indexed by partitions of $n$ written as $\lambda = (\lambda_1\ge \cdots \ge \lambda_p)$ with $\lambda_1+\cdots +\lambda_p=n$ and 
$$S_\lambda = \left\{\sum_{i=1}^p\lambda_i x_i\in \Sym^n(\C)\ \middle \vert \ x_i \text{ distinct}\right\}.$$ 
For example, we have that the generic stratum $\Sym^n_{\neq}(\C)= S_{(1^n)}$ is the configuration space of $n$ unordered distinct points on $\C$. 

Let $c_n=\{1,2,\ldots, n\}$ be a chosen basepoint, then the fundamental group $\pi_1(\Sym^n_{\neq}(\C),c_n)\cong B_n$ is the braid group on $n$ strands which has the following presentation
$$B_n\coloneqq \langle \sigma_1,\ldots,\sigma_{n-1}\mid \sigma_i\sigma_{i+1}\sigma_i=\sigma_{i+1}\sigma_i\sigma_{i+1}, \sigma_i\sigma_j=\sigma_j\sigma_i \text{ if }|i-j|>1\rangle.$$
In the isomorphism above, the generator $\sigma_i$ corresponds to the half Dehn twist swapping the points $i$ and $i+1$ by moving them counterclockwise around each other. 

There is a $B_n$ action on $c^{n}$ where generators act via
$$\sigma_i(g_1,\ldots, g_n)=(g_1,\ldots, g_{i-1},g_{i+1},g_{i+1}^{-1}g_ig_{i+1},g_{i+2},\cdots,g_n).$$
Topologically, this corresponds to the action of $B_n$ on $\Hom(\pi_1(\C-\{1,\ldots, n\},*),G)$ for some basepoint $*$ and generators $\gamma_1,\ldots, \gamma_n$ around the punctures, where we recover $g_i$ from the image of $\gamma_i$. With this, we can define topological Hurwitz spaces as follows.
\begin{definition}
Let $\Hur^{G,c}_n$ be the finite unramified cover of $\Sym^n_{\neq}(\C)$ with fibers $c^n$ and monodromy given by the $B_n$ action on $c^n$. Equivalently, we have
$$\Hur^{G,c}_n=\widetilde{\Sym^n_{\neq}(\C)}\times_{B_n}c^n$$
where $\widetilde{\Sym^n_{\neq}(\C)}\rightarrow \Sym^n_{\neq}(\C)$ is the universal cover.
\end{definition}
One can prove that $\Hur^{G,c}_n$ is the moduli space of $n$-branched $G$-covers of a disk $D$ that have a marked point on the boundary of the cover, such that the local monodromy around each branch point lies in $c$ and they are not necessarily connected. Here, we needed to use the homeomorphism from $\C$ to $D$ so that we can mark a point on the boundary of the cover, this roughly corresponds to marking a point of the cover over $\C$ that is ``near infinity''. The point of this marking is to rigidify the cover and get rid of the $G$-automorphisms. To see the correspondence between the earlier definition and the moduli space interpretation, we note that the projection to $\Sym^n_{\neq}(\C)$ determines the branch points, and the tuple $(g_1,\ldots, g_n)$ determines the local monodromy of the cover. 

Similarly, we can define the connected Hurwitz space $\CHur^{G,c}_n\subseteq \Hur^{G,c}_n$ to be the union of connected components that parametrize the marked $n$-branched $G$-covers of $D$ which are connected. This is equivalent to requiring covers to have full monodromy, so it is the cover of $\Sym^n_{\neq}(\C)$ with fibers $C\coloneqq\{(g_1,\ldots, g_n)\in c^n \mid G=\langle g_1,\ldots,g_n\rangle\}$ with the same action of $B_n$ which we write as
$$\CHur^{G,c}_n=\widetilde{\Sym^n_{\neq}(\C)}\times_{B_n} C.$$

There is a $G$-action on both $\Hur^{G,c}_n$ and $\CHur^{G,c}_n$ corresponding to moving the marked point on the boundary of the cover, and we can define the quotients $\Hur^{G,c}_n/G$ and $\CHur^{G,c}_n/G$. Now, these are the moduli spaces of unmarked $n$-branched $G$-covers of $D$ (or $\C$, since we no longer require the marking), where the latter also requires these covers to be connected. We write these as
$$\Hur^{G,c}_n/G=\widetilde{\Sym^n_{\neq}(\C)}\times_{B_n} c^n/G,$$
$$\CHur^{G,c}_n/G=\widetilde{\Sym^n_{\neq}(\C)}\times_{B_n} C/G.$$
\subsubsection{Algebraic Hurwitz spaces}\label{sec: alg Hurwitz spaces} First, we define $\Sym^n(\A^1)$ and its diagonal stratification algebraically over $\Z$. We can view $\Sym^n(\A^1)\cong \A^n$ via the isomorphism given by sending $n$ points to the (coefficients of) monic degree $n$ polynomial with these $n$ roots. Given any partition $\lambda$ of $n$, we let $S_{\lambda}$ to be the reduced locally closed stratum that consists of polynomials with roots given by the partition $\lambda$, just like in the case of $\C$. For example, $\Sym^n_{\neq}(\A^1)=S_{(1^n)}$ is the open subscheme of monic polynomials with no repeated roots, so when viewed inside $\A^n$ it is the complement of the vanishing locus of the discriminant.

We want an algebraic theory of Hurwitz spaces over nonzero characteristic for our applications of counting over finite fields, but first we argue that these spaces are algebraic over $\C$. Riemann's existence theorem tells us that finite topological coverings of a complex algebraic variety are equivalent to finite étale covers. Applying this to our case of coverings over $\Sym^n_{\neq}(\A^1_\C)$, we see that $\Hur^{G,c}_n,\CHur^{G,c}_n,\Hur^{G,c}_n/G,\CHur^{G,c}_n/G$ can be defined as algebraic varieties over $\C$, and their $\C$-points agree with the topological definition given above. 

In \cite[Section 2.1]{landesman2025cohenlenstramomentsfunctionfields}, the authors extended this and showed that $\Hur^{G,c}_n$ (and hence $\CHur^{G,c}_n$) can be defined as a scheme over $\Z[\frac{1}{|G|}]$ which is finite étale over $\Sym^n_{\neq}(\A^1)$. Furthermore, they showed that $\Hur^{G,c}_n/G$ (and hence $\CHur^{G,c}_n/G$) are algebraic stacks over $\Z[\frac{1}{|G|}]$. Hence, for all primes $p$ not dividing $|G|$, we can define these Hurwitz spaces over $\F_p$ as the special fiber of the construction.

These spaces over $\Z[\frac{1}{|G|}]$ were defined using a moduli interpretation. Let $T$ be a scheme over $\Spec \Z[\frac{1}{|G|}]$, then the $T$-points of $\Hur^{G,c}_n/G$ are easy to define, they are basically finite $G$-covers over $\mathbb P^1_T$ where the local monodromy of a point over $\A^1_T$ lies in $c$, see Definition 2.1.1 of loc. cit. for more details. This is analogous to the topological moduli space discussed previously. The moduli space of $\Hur^{G,c}_n$ is more difficult, because we need to mark an unramified point in the cover ``near infinity". This is not a problem if $\infty$ is not a branch point, as we can simply mark a point over infinity which is already unramified. On the other hand, if $\infty$ is a branch point with inertia group of order $e$, then Definition 2.1.3 of loc. cit. uses a root stack of order $e$ along $\infty$ of $\mathbb P^1$ as defined in \cite[Definition 2.2.4]{Cadman_2007} to perform this marking. We omit the details, but intuitively, this technique lets us ``mark $1/e$ of a ramification point''.

Finally, we explain the discussion in Section \ref{sec: cubic fields} of the introduction in more detail. The connected components of $Y_{2N}$ that come from $X_{2N}$ are the connected components of $\CHur^{G,c}_n/G$ which don't ramify at infinity (so the discriminant remains $2N$), where here we have $n=2N$, $G=S_3$ and $c$ conjugacy class of transpositions. We take $\CHur^{G,c}_n$ because for counting cubic fields over $\F_p(t)$ we only want connected covers, and then we remove the marked point by taking the $G$-quotient (refer to \cite[Section 10]{landesman2025homologicalstabilityhurwitzspaces} for more details). The cohomology of $\CHur^{G,c}_n/G$ is simply the $G$-invariant cohomology of $\CHur^{G,c}_n$, which appears as a subspace of the cohomology of $\Hur^{G,c}_n$. Hence, the weight bound in Theorem \ref{thm: S3 transposition} for $\Hur^{G,c}_n$ also applies to $\CHur^{G,c}_n/G$.
\subsection{Shuffle algebras}
Here, we define primitive bialgebras in the context of braided monoidal abelian categories, and discuss the three important examples of tensor algebras $T(V)$, quantum shuffle algebras $\mathfrak A(V)$ and Nichols algebras $\mathfrak B(V)$. We follow the exposition in \cite[Section 2]{ETW17} and \cite[Section 2,3]{KS20}.
\subsubsection{Braided monoidal abelian categories} We first discuss the notion of a braided vector space over $k$.
\begin{definition}
A \textit{braided vector space} $(V,R)$ over $k$ consists of a finite-dimensional $k$-vector space $V$ with a braiding $R\colon V\otimes V\xrightarrow{\sim} V\otimes V$ which satisfies the Yang-Baxter equation
$$(R\otimes \id)\circ (\id \otimes R)\circ (R\otimes \id)=(\id \otimes R)\circ (R\otimes \id)\circ (\id \otimes R).$$
\end{definition}
This relation is exactly what is required for $V^{\otimes n}$ to be a representation of $B_n$ where $\sigma_i$ acts on the $i$ and $i+1$-th copies of $V$ by $R$. 
\begin{example}\label{eg: braided vector spaces}
We give several examples of braided vector spaces $(V,R)$ in increasing generality.
\begin{enumerate}[(\alph*)]
    \item Trivial braided vector space. This has braiding given by swapping the terms $R(v\otimes w)=w\otimes v$. 
    \item Braided vector spaces of diagonal type. Let $\{x_i\}_i$ be a basis of $V$, and $(q_{ij})$ be a matrix with entries in $k^\times$. We define the braiding on the basis $x_i\otimes x_j$ by $R(x_i\otimes x_j)=q_{ij}x_j\otimes x_i$.
    \item Braided vector spaces of rack type. A \textit{rack} is a set $S$ with binary operation $(a,b)\mapsto b^a$ satisfying (i) $(c^a)^{b^a}=(c^b)^a$ and (ii) for each $a,b\in S$, there is a unique $c\in S$ with $c^a=b$ which we denote as $c={}^ab$. A rack $2$-cocycle is a function $x\colon S\times S\rightarrow k^\times$ satisfying $x_{ab}x_{a^bc}=x_{ac}x_{a^cb^c}$. Given a rack and a corresponding cocycle, we can define the braided vector space $V(S,x)=kS$ with braiding $R(a\otimes b)=x_{ab}(b\otimes a^b)$. 
 \end{enumerate}
\end{example}
Instead of just considering one particular braided vector space, we want to consider compatible braidings for objects in a category, which motivates the following definition.
\begin{definition}
A \textit{braided monoidal abelian category} $(\mathcal V,\otimes ,1,R)$ over $k$ is a monoidal $k$-linear abelian category $(\mathcal V,\otimes, 1)$, and for any $V,W\in \mathcal V$ there is a braiding $R_{V,W}\colon V\otimes W\xrightarrow\sim W\otimes V$ which satisfies the braiding axioms, see \cite[Definition 3.2.1] {HS20}. 
\end{definition}
We see that if $\mathcal V$ is a braided monoidal abelian category of $k$-vector spaces, then each object can be given a braided vector space structure. The main example relevant to Hurwitz spaces is the category of Yetter-Drinfeld modules $\mathcal {YD}_G$, as we will see at the end of Section \ref{sec: KS equivalence}.
\begin{example}\label{eg: YD G-module}
Let $G$ be a finite group. A Yetter-Drinfeld $G$-module is a finite-dimensional $G$-graded $k$-vector space $V=\bigoplus_{g\in G} V_g$ such that $V_g\cdot h\subseteq V_{h^{-1}gh}$. These are the objects of the category $\mathcal {YD}_G$ with $\otimes$ being the usual graded tensor product and the braiding given by $R(v\otimes w)=w\otimes (v\cdot h)$ for $v\in V_g$ and $w\in V_h$. Examples of Yetter-Drinfeld $G$-module are the braided vector spaces of rack type in Example \ref{eg: braided vector spaces}(c) where the rack $S$ which as a set is a union of conjugacy classes in the group $G$ with $g^h=h^{-1}gh$.
\end{example}
\subsubsection{Tensor algebras, quantum shuffle algebras and Nichols algebras} Let $V\in \mathcal V$ be an object in a braided monoidal abelian category over $k$, with the braiding $R=R_{V,V}\colon V\otimes V\xrightarrow\sim V\otimes V$. For example, $(V,R)$ can be a braided vector space. 

For any permutation $\sigma \in S_n$, there is a Matsumoto lift $\tilde \sigma\in B_n$ given by expressing $\sigma$ as the minimum length (reduced) word in transpositions $s_i=(i,i+1)$ and replacing these with the braid $\sigma_i$. This is well-defined because any two reduced words are related by $s_is_{i+1}s_i=s_{i+1}s_is_{i+1}$ or $s_is_j=s_js_i$ for $|i-j|>1$ which are exactly the braid relations. This allows us to define the braiding $R_\sigma\colon V^{\otimes n}\xrightarrow \sim V^{\otimes n}$ given by applying $R$ to the Matsumoto lift $\tilde \sigma$.
 
Define the set of $(p,q)$-shuffles to be
$$\text{Sh}(p,q)\coloneqq\{\sigma\in S_{p+q} \colon \sigma(1)<\cdots <\sigma(p) \text{ and }\sigma(p+1)<\cdots<\sigma(p+q)\}$$
which are the permutations that preserve the order of the first $p$ elements and last $q$ elements. The inverses of the $(p,q)$-shuffles are exactly the $(p,q)$-unshuffles, which we denote by $\text{USh}(p,q)$.

The tensor algebra $T(V)=\bigoplus_{n=0}^\infty V^{\otimes n}$ has free multiplication $\mu$, i.e. it has components $\mu_{p,q}\colon V^{\otimes p}\otimes V^{\otimes q}\xrightarrow\sim V^{\otimes p+q}$ given by the identity map. The comultiplication $\Delta$ is the unique algebra map that extends $\Delta|_{T^1(V)}=1\otimes \id + \id \otimes 1$, concretely, it is given on components by
$$\Delta_{p,q} (v)= \sum_{\sigma\in \text{USh}(p,q)} R_{\sigma}v$$
where we interpret the RHS as an element of $V^{\otimes p}\otimes V^{\otimes q}$ via the identity map $V^{\otimes p+q}\xrightarrow\sim V^{\otimes p}\otimes V^{\otimes q}$. We call this the unshuffle coproduct and sometimes write $\Delta=\Delta_\star$ for clarity.

Dually, the quantum shuffle algebra (or cotensor algebra) $\mathfrak A(V)=\bigoplus_{n=0}^\infty V^{\otimes n}$ has free comultiplication with components $\Delta_{p,q}\colon V^{\otimes p+q}\xrightarrow\sim V^{\otimes p}\otimes V^{\otimes q}$ given by the identity map, and multiplication $\mu=\mu_\star$ has components $\mu_{p,q}(v,w)$ given by the shuffle product
$$v \star_{p,q}w = \sum_{\sigma\in \text{Sh}(p,q)} R_\sigma (v\otimes w).$$
Suppose that $V\in \mathcal V$ is dualizable and has dual $V^*$. For example, all braided vector spaces $(V,R)$ have duals $(V^*,R^*)$. Then, the tensor algebra $T(V)$ is dual to the quantum shuffle algebra $\mathcal A(V^*)$.

Lastly, the Nichols algebra $\mathfrak  B(V)$ is defined to be the image of the morphism of bialgebras $T(V)\rightarrow \mathfrak A(V)$ which is uniquely determined by sending $T^1(V)=V$ into $\mathfrak A^1(V)=V$ via the identity map. In other words, it is the subalgebra of $\mathfrak A(V)$ generated by degree $1$ elements under the shuffle product. 

In \cite{KS20}, the tensor algebra, quantum shuffle algebra and Nichols algebra are denoted as $T_!(V)$, $T_*(V)$ and $T_{!*}(V)$ respectively, as they correspond to $j_!\mathcal L$, $j_*\mathcal L$ and $j_{!*}\mathcal L$ as in Proposition \ref{prop: equivalence properties}(b).
\subsubsection{Primitive bialgebras}\label{sec: primitive bialgebra} Let $\mathcal V$ be a braided monoidal abelian category over $k$ with $V\in \mathcal V$. The tensor algebra $T(V)$, quantum shuffle algebra $\mathfrak A(V)$ and Nichols algebra $\mathfrak B(V)$ are examples of primitive bialgebras in $\mathcal V$, which we now define.
\begin{definition}
A \textit{primitive bialgebra} in $\mathcal V$ is a graded bialgebra $A=\bigoplus_{n= 0}^\infty A_n$ which is connected and coconnected. The category of primitive bialgebras $\text{PB}(\mathcal V)$ has these objects with morphisms respecting the bialgebra structure.
\end{definition}
In more detail, a primitive bialgebra $A$ has the structure of a connected graded algebra with unit $\eta\colon A_0=k\hookrightarrow A$ and multiplication $\mu\colon A\otimes A\rightarrow A$ which respects the grading, as well as the structure of a coconnected graded coalgebra with counit $\epsilon \colon A\twoheadrightarrow k=A_0$ and comultiplication $\Delta \colon A\rightarrow A\otimes A$ satisfying the compatibility condition that $\Delta$ is a morphism of algebras. Here, the algebra structure of $A\otimes A$ has multiplication given by $(\mu\otimes \mu)\circ (\id\otimes R_{A,A}\otimes \id)\colon(A\otimes A)\otimes (A\otimes A)\rightarrow A\otimes A$. For non-negative integers $p,q$, let $\mu_{p,q}\colon A_p\otimes A_q\rightarrow A_{p+q}$ and $\Delta_{p,q}\colon A_{p+q}\rightarrow A_p\otimes A_q$ be the components of multiplication and comultiplication in $A$. One can also show that primitive bialgebras are graded Hopf algebras \cite[Proposition 2.4.11]{KS20} by defining the antipode in terms of multiplication and comultiplication. 

To any primitive bialgebra we can associate the $n$-th bar-complex and $n$-th cobar-complex given respectively by
$$B_n(A)=\left\{A_1^{\otimes n}\rightarrow \cdots \rightarrow \bigoplus_{\substack{p+q+r=n\\
0<p,q,r<n}} A_p\otimes A_q \otimes A_r\rightarrow \bigoplus_{\substack{p+q=n\\
0<p,q<n}} A_p\otimes A_q \rightarrow A_n\right\},$$
$$B_n^\dagger(A)=\left\{A_n\rightarrow \bigoplus_{\substack{p+q=n\\
0<p,q<n}} A_p\otimes A_q \rightarrow \bigoplus_{\substack{p+q+r=n\\
0<p,q,r<n}} A_p\otimes A_q \otimes A_r\rightarrow \cdots\rightarrow A_1^{\otimes n}\right\},$$
where the grading of the complexes are normalized such that $A_n$ is in degree $-1$ and degree $1$ respectively. The $n$-th bar-complex (resp. cobar-complex) come froms the $n$-th bar-cube (resp. cobar-cube) which is a hypercube diagram consisting of $2^{n-1}$ vertices with entries $A_{\lambda_1}\otimes \cdots \otimes A_{\lambda_m}$ where $\lambda_1+\cdots +\lambda_m=n$ and $\lambda_i>0$, and has edges given by multiplication (resp. comultiplicaiton). We collapse the maps in the bar-cube and cobar-cube while adding a Koszul sign twist to get the bar-complex and and cobar-complex respectively.
\subsection{Perverse sheaves}
We first give general definitions for perverse sheaves, the define the notion of factorizable perverse sheaves over $\Sym(\C)$ and explain Theorem \ref{thm: KS20} in more detail.  
\subsubsection{Definitions}\label{sec: perverse definitions} Perverse sheaves can be defined in various different but related settings, for example:
\begin{enumerate}
\item Analytically. When $X$ is a complex (or real analytic) manifold with coefficients in an abelian category $\mathcal V$ (e.g. $k$-vector spaces), there is $D^b_c(X,\mathcal V)$ and $\Perv(X,\mathcal V)$ as in \cite{Kashiwara_Schapira, KS20}.
\item Algebraically. When $X$ is a scheme over a field $K$ with $\overline\Q_\ell$-coefficients (also $\Z_\ell,\Q_\ell$, etc.) for $\ell\neq \text{char} (K)$, one can construct the derived category $D^b_c(X_{\text{ét}},\overline\Q_\ell)$ and perverse sheaves $\Perv(X_{\text{ét}},\overline\Q_\ell)$ as in \cite{BBD, Kiehl_Weissauer_2010}.
\item Relative version of the algebraic case. When $X\rightarrow S$ is a finitely presented morphism of schemes over $\Z[\frac 1 \ell]$ with $\overline\Q_\ell$-coefficients (also $\Z_\ell,\Q_\ell$, etc.), there is $D^b_c(X,\overline\Q_\ell)$ and $\Perv(X,\overline\Q_\ell)$ as in \cite{Hansen_Scholze_2023}. 
\end{enumerate}

We remark that when $X$ is a finite-type complex algebraic variety, there is an analytification functor $D^b_c(X,\overline\Q_\ell)\rightarrow D^b_c(X^{an},\overline\Q_\ell)$ from the algebraic setting (2) to the analytic setting (1) that is essentially surjective, likewise for $\Z_l,\Q_\ell$ coefficients \cite[Section 6]{BBD}. 

Our paper will make use of all three constructions in some way. This is because we want results on weights of schemes over $\F_p$ in the context of (2) while making use of the machinery developed in \cite{KS20} in the context of (1). We will bridge between these two situations in Section \ref{sec: construction weight filtration} with the help of the relative algebraic setting (3). However, apart from this and Section \ref{sec: weights over F_p}, the rest of our paper will only be concerned with perverse sheaves in the analytic setting. Thus, we introduce perverse sheaves in the analytic setting following \cite[Section 1.1,1.4]{KS20}, and leave the algebraic case to analogy, referring the reader to the references above. 

Let $X$ be a connected complex manifold, with $\dim_\C(X)=n$, and let $S=\{X_{\alpha}\}$ be a complex analytic Whitney stratification of $X$. In this paper we will only use the diagonal stratification of $\Sym^n \C$ introduced previously. Let $\mathcal V$ be an abelian category, for example the category of $k$-vector spaces, or a braided monoidal abelian category over $k$. Let $\Shv(X,\mathcal V)$ be the abelian category of sheaves on $X$ with values in $\mathcal V$. We define sheaves in the usual way as a contravariant functor $\mathcal F$ from the category of open sets in $X$ to $\mathcal V$ satisfying the sheaf axiom where $0\rightarrow \mathcal F(U)\rightarrow \prod_i \mathcal F(U_i)\rightarrow \prod_{i,j} \mathcal F(U_i\cap U_j)$ is exact.

Let $\Shv(X,S,\mathcal V)$ be the abelian subcategory of $\Shv(X,\mathcal V)$ consisting of sheaves $\mathcal F$ which are constructible with respect to $S$, which means that $\mathcal F$ is locally constant on each stratum $X_{\alpha}$. Let $C^b(X,S,\mathcal V)$ be the category of bounded chain complexes of $\Shv(X,S,\mathcal V)$, and from this, we can construct the bounded derived category $D^b(X,S,\mathcal V)$ by inverting all quasi-isomorphisms. Note that for all $K\in D^b(X,S,\mathcal V)$, the cohomology sheaves $\mathcal H^i(K)$ are in $\Shv(X,S,\mathcal V)$. We can define the six derived functors $f_*,f_!,f^*,f^!,\otimes^\mathbf L,\D$ in the derived category.

Let $\Perv(X,S,\mathcal V)\subseteq D^b(X,S,\mathcal V)$ be the abelian category which is the heart of the perverse t-structure, see \cite{Kashiwara_Schapira,Kiehl_Weissauer_2010} for more details. Explicitly, a complex $K\in D^b(X,S,\mathcal V)$ is perverse if for every stratum $i_{\alpha}\colon X_{\alpha}\hookrightarrow X$ the cohomology sheaves satisfy both $\mathcal H^i(i_{\alpha}^*K)=0$ for $i>-\dim_\C(X_{\alpha})$ and $\mathcal H^i(i_{\alpha}^!K)=0$ for $i<-\dim_\C(X_{\alpha})$. One common way to define perverse sheaves in the literature is to use the dual instead of the costalk. We avoid this because this is not the correct notion when $\otimes$ is not bi-exact in $\mathcal V$. 

We denote $D^b_c(X,\mathcal V)=\bigcup_S D^b(X,S,\mathcal V)$ to be the bounded category of constructible derived sheaves and $\Perv(X,\mathcal V)=\bigcup_S \Perv(X,S,\mathcal V)$ to be the subcategory of perverse sheaves, where here we take the union over all possible stratifications.

For a complex $K\in D^b(X,S,\mathcal V)$, we can take its perverse cohomology ${}^p \mathcal H^i(K)\in \Perv(X,S,\mathcal V)$. For a locally closed embedding $j\colon U\hookrightarrow X$, the middle extension for $K\in \Perv(U,\mathcal V)$ is given by $j_{!*}K=\text{Im}({}^p\mathcal H^0(j_!K)\rightarrow {}^p\mathcal H^0(j_*K))\in \Perv(X,\mathcal V)$. This sends simple perverse sheaves to simple perverse sheaves, and conversely any simple perverse sheaf in $\Perv(X,S,\mathcal V)$ is necessarily of the form $j_{!*}\mathcal L$ for some $j\colon X_{\alpha}\hookrightarrow X$ and local system $\mathcal L$ on the stratum $X_{\alpha}$.
\subsubsection{Factorizable perverse sheaves}\label{sec: factorizable perverse sheaves}
Now let $\mathcal V$ be a braided monoidal abelian category over $k$, and we further impose that $\otimes$ is biexact. The braiding data of $\mathcal V$ extends to bounded chain complexes $C^b(\mathcal V)$ by the Koszul sign rule where $R_{V[m],W[n]}=(-1)^{mn}R_{V,W}[m+n]$. Since $\otimes$ is biexact, the braided structure descends to $D^b(\mathcal V)$. This induces a braided monoidal structure on perverse sheaves via the external tensor product as follows. Let $\mathcal F\in \Perv(X,\mathcal V)$ and $\mathcal G\in \Perv(Y,\mathcal V)$ be perverse sheaves with coefficients in $\mathcal V$, then we have the external braiding isomorphism $$R_{\mathcal F,\mathcal G}\colon \mathcal F\boxtimes \mathcal G\rightarrow \pi^*(\mathcal G\boxtimes \mathcal F)$$ where $\pi\colon X\times Y\xrightarrow\sim Y\times X$. 

We can interpret this using the operadic point of view, see \cite[Section 3.2A]{KS20} for more details. Let $E_2$ be the operad of little $2$-disks in $\C$, so an element $(U_1,\ldots, U_m)\in E_2(m)$ is a tuple of disjoint round open disks in the unit open disk $D$. We can interpret the external braiding isomorphism as a datum of a continuous family of external tensor products $\boxtimes_{(U_1,\ldots, U_m)}\mathcal F_i$ indexed by elements in $E_2(m)$ for $m\ge 0$. For example, roughly speaking, swapping $U_i$ and $U_{i+1}$ via a half Dehn twist is the same as applying $R_{\mathcal F_i,\mathcal F_{i+1}}$ to the respective factors $\mathcal F_i$, $\mathcal F_{i+1}$.

This is important for us to define factorizable perverse sheaves. Recall that we have the diagonal stratification $S_n=\{S_\lambda\}_\lambda$ on $\Sym^n(\C)$. Combining this for all $n\ge 0$, we get the stratification $S$ on $\Sym(\C)$. Furthermore, for any open $U\subseteq \C$, the stratifications $S$, $S_n$ restrict to stratifications $S_U$ and $S_{n,U}$ on $\Sym (\C)$ and $\Sym^n(\C)$ respectively. Define $\Perv(\Sym (U), S_U,\mathcal V)$ to be the product of the categories $\Perv(\Sym^n(U), S_{n,U},\mathcal V)$, so an object $\mathcal F\in \Perv(\Sym (U), S_U,\mathcal V)$ is a sequence $(\mathcal F_n)_{n\ge 0}$ of perverse sheaves $\mathcal F_n\in \Perv(\Sym^n (U), S_{n,U},\mathcal V)$. Let 
$a\colon \prod_{i=1}^m \Sym(U_i) \xrightarrow{\ \sim\ }\Sym \left(\bigsqcup_{i=1}^m U_i\right)$ be the isomorphism given by combining the tuples. 
\begin{definition}\label{def: FPS}
A \textit{factorizable perverse sheaf} is an object $\mathcal F=(\mathcal F_n)_n\in \Perv(\Sym (\C), S,\mathcal V)$ with the additional data of an isomorphism
$$\mu_{U_1,\ldots, U_m} \colon \boxtimes_{(U_1,\ldots, U_m)}\mathcal F|_{\Sym (U_i)}\xrightarrow{\ \sim \ }a^*(\mathcal F|_{\Sym(\sqcup U_i)})$$
for each $U_1,\ldots, U_m$ that are compatible with respect to operadic compositions. The category of factorizable perverse sheaves $\text{FPS}(\mathcal V)$ has these as objects with morphisms that respect the data of the isomorphisms above.
\end{definition}
\begin{remark}
In Section \ref{sec: factorizable perverse sheaves over W}, we will extend this definition to the case when $\otimes$ is not biexact by replacing $\boxtimes$ with its derived version $\boxtimes^\mathbf L$.
\end{remark}
We give the main examples of factorizable perverse sheaf that we will use.
\begin{example}\label{eg: compatible local systems}
For a braided vector space $V\in \mathcal V$, we define the \textit{compatible local systems} $\mathcal L(V[1])=(\mathcal L_n(V[1]))_n$ as follows. Let $\mathcal L_n(V[1])$ be the local system on $\Sym^n_{\neq}(\C)$ which is given by the $B_n$-representation $(V[1])^{\otimes n}$. Due to the Koszul sign twist, this is equal to $(V^{\otimes n}\otimes \text{sgn})[n]$. It is clear that $\mathcal L(V[1])$ is a factorizable perverse sheaf. 

Furthermore, let $j\colon \Sym_{\neq}(\C)\hookrightarrow \Sym(\C)$ and $j_n\colon \Sym^n_{\neq}(\C)\hookrightarrow \Sym^n(\C)$. Then, it can be shown that $j_{*}\mathcal L(V[1])$, $j_{!}\mathcal L(V[1])$ and $j_{!*}\mathcal L(V[1])$ are all factorizable perverse sheaves.
\end{example}
\subsubsection{Kapranov-Schechtman equivalence}\label{sec: KS equivalence} We explain the equivalence of Theorem \ref{thm: KS20} in more detail. Let $\mathcal V$ be a braided monoidal abelian category over $k$ with $\otimes$ biexact. Then \cite{KS20} constructs the functors
\[\begin{tikzcd}
	{\text{PB}(\mathcal V)} && {\text{FPS}(\mathcal V)}
	\arrow["L", shift left, from=1-1, to=1-3]
	\arrow["\Phi", shift left, from=1-3, to=1-1]
\end{tikzcd}\]
which give an equivalence between the two categories.

We give a sketch of the construction and refer the reader to \cite[Section 4]{KS20} for more details. Given a primitive bialgebra $A$, the factorizable perverse sheaf $L(A)=(\mathcal E^\bullet_n(A))_n$ consists of the Cousin complexes $\mathcal E_n^\bullet(A)$ which are constructed as follows. For a partition $\alpha$ of $n$, consider the imaginary stratum $X^{\mathfrak I}_\alpha$ given by the preimage of the stratum $\alpha$ on $\Sym^n(\R)$ along the imaginary part map $\text{Im}\colon \Sym^n(\C)\rightarrow \Sym^n(\R)$. Consider the real-parts projection $\rho_\alpha\colon X^{\mathfrak I}_\alpha\rightarrow \prod \Sym^{\alpha_v}(\R)$, and let $\tilde{\mathcal E}^\alpha(A)$ be the pullback along $\rho_\alpha$ of the product of constructible sheaves on $\Sym^{\alpha_v}(\R)$ each given by the $\alpha_v$-th cobar-cube. Then, define the Cousin sheaves $\mathcal E^\alpha(A)$ as the pushforward of $\tilde{\mathcal E}^\alpha(A)$ to $\Sym^n(\C)$, and the $n$-th Cousin complex is given by
$$\mathcal E^\bullet_n(A)=\left\{\mathcal E^{(1^n)}(A)\rightarrow \bigoplus_{l(\alpha)=n-1}\mathcal E^\alpha(A)\rightarrow \cdots \rightarrow \mathcal E^{(n)}(A)\right\}$$
with natural maps between Cousin sheaves along with a Koszul sign twist, and normalized such that $\mathcal E^{(n)}(A)$ is in degree $-1$. On the other hand, given $\mathcal F=(\mathcal F_n)$, we construct $\Phi(\mathcal F)=\oplus_{n\ge 0} \Phi_{\tot}(\mathcal F_n)$ to be the total vanishing cycles at the origin $0^n\in \Sym^n(\C)$. More precisely, $\Phi_{\tot}(\mathcal F_n)$ are the vanishing cycles at the origin in $\mathcal F_n|_{\Sym^n_0(\C)}$ with respect to the coefficient $a_2=\sum_{i\neq j}x_ix_j$ where $\Sym^n_0(\C)\subseteq \Sym^n(\C)$ is the subset where $a_1=\sum_i x_i=0$. The comultiplication is given by generalization maps, and the multiplication is obtained by reconstructing the Cousin complex and looking at the differentials.

Because the construction is rather lengthy, we often rely on the properties of this equivalence rather than go through the construction itself. We list the ones we will use following \cite[Theorem 3.3.1, 3.3.3]{KS20}. Recall the definition of compatible local systems $\mathcal L(V[1])$ and let $j\colon \Sym_{\neq}(\C)\hookrightarrow \Sym(\C)$.
\begin{proposition}\label{prop: equivalence properties}
The equivalence satisfies the following additional properties.
\begin{enumerate}[(\alph*)]
\item On the generic stratum, $L(A)|_{\Sym_{\neq}(\C)}\cong \mathcal L(A_1[1])$. 
\item We have $L(T(V))\cong j_!\mathcal L(V[1])$, $L(\mathfrak A(V))\cong j_*\mathcal L(V[1])$ and $L(\mathfrak B(V))\cong j_{!*}\mathcal L(V[1])$.
\item The total vanishing cycles $\Phi_{\tot}(L_n(A))\cong A_n$.
\item The stalk of $L_n(A)$ at the origin is the $n$-th bar-complex $B_n(A)$.
\item Suppose $A$ is dualizable with dual $A^*$, then $L_n(A^*)\cong \D L_n(A)$.
\end{enumerate} 
\end{proposition}
We discuss how (d) implies the statement on cohomology in Theorem \ref{thm: KS20}, following the proof of \cite[Corollary 3.3.4]{KS20} For a primitive bialgebra $A$, the cohomology of the bar-complex $H^jB_n(A)$ is isomorphic to $\Tor^A_{-j,n}(k,k)$. On the other hand, for any $\mathcal F_n\in \Perv(\Sym^n(\C),S_n,\mathcal V)$, the stalk $(\mathcal F_n)_0$ at the origin is quasi-isomorphic to $R\Gamma \mathcal F_n$ by quasi-homogenity of the diagonal stratification (in other words, by the contracting $\C^*$ action), so (d) gives an isomorphism 
$$R\Gamma^{j}L_n(A)\cong \Tor^A_{-j,n}(k,k).$$
In the special case when $A=\mathfrak A(V)$ so that $L_n(A)=j_{n*}\mathcal L_n(V[1])$ as in (b), we see that LHS is $$R\Gamma^j(\Sym^n(\C),j_{n*}\mathcal L_n(V[1]))=R\Gamma^j(\Sym^n_{\neq}(\C),\mathcal L_n(V[1]))=H^j(B_n,V^{\otimes n}\otimes \text{sgn})[n],$$
so we recover the dual of \cite[Theorem 1.3]{ETW17} twisted by a sign:
$$H^j(B_n,V^{\otimes n}\otimes \text{sgn})\cong \Tor^{\mathfrak A(V)}_{n-j,n}(k,k).$$

Finally, recall that the topological Hurwitz spaces are the finite étale covers $\pi_n \colon \Hur^{G,c}_n \rightarrow \Sym^n_{\neq}(\C)$ given by the $B_n$-representation $c^n$. Recall from Example \ref{eg: braided vector spaces} and \ref{eg: YD G-module} that $c$ is a rack with $g^h=h^{-1}gh$ and there is an associated braided vector space $V_{\epsilon}=V(c,-1)\in \mathcal YD_G$ where we take the cocycle $x_{ab}\equiv -1$. The $B_n$-representation $c^n$ is the same as the $B_n$ representation on $V_{\epsilon}^{\otimes n}\otimes \text{sgn}$ as the sign representation cancels with the negative cocyle. Hence, the cohomology of the Hurwitz space is given by $$H^j(\Hur^{G,c}_n,k)\cong H^j(B_n, c^n)\cong H^j(B_n, V_\epsilon^{\otimes n}\otimes \text{sgn})\cong \Tor^{\mathfrak A(V_{\epsilon})}_{n-j,n}(k,k).$$
\subsection{Weight filtrations}\label{sec: weight filtration}
We introduce the weight filtration in two different settings, first over $\F_p$ using eigenvalues of Frobenius, then over $\C$ using the theory of mixed Hodge modules. In Section \ref{sec: construction weight filtration}, we will prove a comparison theorem that links the two in the case of Hurwitz spaces.
\subsubsection{Weight filtration over $\F_p$}\label{sec: weights over F_p} 
We now briefly explain Deligne's theory of weights in the algebraic setting over $\F_p$ following \cite{BBD,Kiehl_Weissauer_2010}. Let $X$ be a scheme over $\F_p$, let $\mathcal F$ be a constructible $\overline\Q_\ell$-sheaf with $\ell \neq p$ and fix an isomorphism $\iota\colon \overline\Q_\ell\xrightarrow\sim \C$. For a closed point $x\in X$ with residue field $k_x$ and geometric point $\overline x$, the stalk $\mathcal F_{\overline x}$ is a finite-dimensional $\overline\Q_\ell$ vector space by constructibility and has an action of geometric Frobenius $\Frob_x$. We say that $\mathcal F$ is \textit{pure of weight $w$} if for all such closed points $x\in X$, the eigenvalues $\alpha$ of $\Frob_x$ on $\mathcal F_{\overline x}$ satisfies $|\iota(\alpha)|=(\#k_x)^{w/2}$. We say that $\mathcal F$ is \textit{mixed} if there exists a finite filtration of $\mathcal F$ where each graded piece is pure, and we let $w(\mathcal F)$ be the maximum weight of $\mathcal F$.

Recall that in analogy to Section \ref{sec: perverse definitions} we can define the bounded derived category of constructible sheaves $D^b_c(X,\overline\Q_\ell)$ and its subcategory of perverse sheaves $\Perv(X,\overline\Q_\ell)$ in the algebraic setting. We say that a complex $K\in D^b_c(X,\overline\Q_\ell)$ is \textit{pointwise pure of weight $w$} if each cohomology sheaf $\mathcal H^i(K)$ is pure of weight $w+i$, and here we note the shift in weight due to cohomological degree. Unfortunately, the definition of pointwise pure does not behave well with respect to derived functors, so we need to define a different notion of purity.

We say that $K$ is \textit{mixed} if each cohomology sheaf $\mathcal H^i(K)$ is mixed, and let $w(K)$ be the maximum of $w(\mathcal H^i(K))-i$ over all $i$. Let the two subcategories $D^b_{\le w}(X,\overline\Q_\ell)$ and $D^b_{\ge w}(X,\overline\Q_\ell)$ consist of mixed complexes which satisfy $w(K)\le w$ and $w(\D K)\le -w$ respectively. It can be shown that $w(\D K)\le -w$ implies $w(K)\ge w$, but the converse is not true. We call $K$ \textit{pure of weight $w$} if it is in $D^b_{\le w}(X,\overline\Q_\ell)\cap D^b_{\ge w}(X,\overline\Q_\ell)$.  

The notions of pointwise pure and pure do not coincide in general, but they do agree in important special cases. For example, when $X$ is smooth and $\mathcal H^iK$ are local systems, we have $\mathcal H^i \D K\cong (\mathcal H^{2d-i}K)^\vee(d)$ where $d=\dim X$, and we can directly compare the Frobenius eigenvalues and show that the two notions agree. Most importantly, this means that when $X$ is a point the two notions agree for any complex $K$. 

The six derived functors $f_*,f_!,f^*,f^!,\otimes^\mathbf L, \D$ all preserve mixed complexes. It is clear from the definition that $f^*$ preserves $D^b_{\le w}(X,\overline\Q_\ell)$, and it follows from the Weil conjectures that $f_!$ also preserves $D^b_{\le w}(X,\overline\Q_\ell)$, so by duality both $f_*$ and $f^!$ preserve $D^b_{\ge w}(X,\overline\Q_\ell)$. Duality $\D$ exchanges $D^b_{\le -w}(X,\overline\Q_\ell)$ and $D^b_{\ge w}(X,\overline\Q_\ell)$ by definition. Middle extensions $j_{!*}$ for locally closed embeddings preserves weight (i.e. preserves both $D^b_{\le w}(X,\overline\Q_\ell)$ and $D^b_{\ge w}(X,\overline\Q_\ell)$), so any simple mixed perverse sheaf is necessarily pure as it is a middle extension of a local system. Since we are using the geometric Frobenius, the Tate twist $(d)$ decreases the weight by $2d$.

For any mixed perverse sheaf $\mathcal F\in \Perv(X,\overline\Q_\ell)$, one can prove that there is a canonical unique increasing \textit{weight filtration} 
$$\cdots\subseteq W_{-1}\mathcal F\subseteq W_0\mathcal F\subseteq W_1\mathcal F \subseteq \cdots$$
where the graded components $\Gr^W_w\mathcal F$ are pure perverse sheaves of weight $w$.
\subsubsection{Weight filtration over $\C$}\label{sec: weight filtration over C}
There is an analogous theory of weights over $\C$ developed in \cite{Saito_1988,Saito_1990} in the context of mixed Hodge modules, generalizing the theory of mixed Hodge structures in \cite{Deligne_1971,Deligne_1974}. In this setting, the notion of weight does not arise from a Frobenius action, so the formal structure of the theory differs substantially, though its overall behavior is closely parallel to the finite field case. We summarize it as follows. 

For a separated reduced algebraic variety $X$ over $\C$, there is an abelian category $\MHM(X)$ of mixed Hodge modules with the functor
$$\text{rat}\colon D^b(\MHM(X))\rightarrow D^b_c(X^{an},\Q)$$
which restricsts to an exact and faithful functor $\MHM(X)\rightarrow \Perv(X^{an},\Q)$. Furthermore, the derived functors $f_*,f_!,f^*,f^!,\otimes^\mathbf L,\D$ in $D^b_c(X^{an},\Q)$ lift to $D^b(\MHM(X))$. 

Each mixed Hodge module $M\in \MHM(X)$ is by definition equipped with a finite increasing filtration $W_iM$ called the \textit{weight filtration}, with graded pieces denoted by $\Gr^W_w M$. We define $D^b_{\le w}(\MHM(X))$ (resp. $D^b_{\ge w}(\MHM(X))$) to consist of $M\in D^b(\MHM(X))$ where $\Gr^W_iH^jM=0$ for $i>j+w$ (resp. $i<j+w)$. Say that $M$ is \textit{pure of weight $w$} if it is in the intersection of both, i.e. $\Gr^W_iH^jM=0$ for $i\neq j+w$. Many analogous theorems as in the finite field case hold, for example, we have again that $f_!,f^*$ preserves $D^b_{\le w}(\MHM(X))$, $f_*,f^!$ preserves $D^b_{\ge w}(\MHM(X))$, duality $\D$ exchanges $D^b_{\le -w}(\MHM(X))$ with $D^b_{\ge w}(\MHM(X))$, and we also have the analogous properties for Tate-twists and degree shifts.

The point is that if we want a weight filtration on a perverse sheaf $\mathcal F\in \Perv(X^{an},\Q)$, we first lift it to a mixed Hodge module $M\in \MHM(X)$ which has a weight filtration, and then project it down to a weight filtration of $\mathcal F$. This should work as long $\mathcal F$ ``comes from geometry''. In our case, we want to construct the weight filtration of $j_*\mathcal L=(j_{n*}\mathcal L_n)_n$ for the compatible local systems $\mathcal L=\mathcal L(V[1])$ on $\Sym_{\neq}(\C)$ where $V$ is a braided $k$-vector space, which was described in Example \ref{eg: compatible local systems}. The first step is to lift $\mathcal L_n$ to a mixed Hodge module $M$ of type $(0,0)$ with trivial Hodge structure. However, $\mathcal L_n$ may not be defined over $\Q$, and even then not every $\Q$-local system can be lifted to a mixed Hodge module. 

Hence, we impose the condition that $\mathcal L_n$ has \textit{finite monodromy} for each $n$. We consider one such $\mathcal L_n$, because of finite monodromy it is defined over a number field $K/\Q$ and we can write $\mathcal L_n=\mathcal L_K\otimes_K k$. Here, it is worth noting that if $V$ is a braided vector space defined over $K$ then all $\mathcal L_n$ are defined over the same $K$. Define a $K$-mixed Hodge module to be a $K$-module in the category of mixed Hodge modules, i.e. a usual mixed Hodge module $M$ with a $\Q$-algebra homomorphism $K\rightarrow \End_{\text{MHM}}(M)$. The functor $\text{rat}$ extends naturally to $\text{rat}_K$ which sends a $K$-mixed Hodge module to a $K$-module of $\Perv(X,\Q)$ which is simply an object of $\Perv(X,K)$. 

We show that the irreducible $K$-local system $\mathcal L_K$ lifts uniquely to a $K$-mixed Hodge module of type $(0,0)$. Indeed, a mixed Hodge module with underlying perverse sheaf a local system (known as a smooth mixed Hodge module) is just an admissible polarizable variation of Hodge structures in the sense of \cite{Steenbrink_Zucker_1985,Kashiwara_1986}. Furthermore, because we have finite monodromy, the admissibility conditions become trivial, and since we require type $(0,0)$ this is just a local system that can be equipped with a monodromy-invariant symmetric positive-definite bilinear form over $\Q$. Note that we merely require that such a form exists without needing to choose one. With the $K$-module structure, this is a $K$-local system, but since morphisms in the category of mixed Hodge modules need not respect polarizations, we only need the bilinear form to be defined over $\Q$, not $K$. Such a form always exists because the monodromy group $G$ is finite, so averaging any symmetric positive-definite form over $G$ produces a monodromy-invariant one. 

With this, we can define the weight filtration on $j_{n*}\mathcal L_{n}$, by lifting $\mathcal L_n$ to the $K$-mixed Hodge module $M_K$ and taking the pushforward $j_{n*}M_K$ which gives a canonical weight filtration on $\text{rat}(j_{n*}M_K)=j_{n*}\mathcal L_K$. Tensoring this to $k$ gives us the weight filtration on $j_{n*}\mathcal L_{n*}$. 

We discuss a different, geometric way to get the mixed Hodge module structure on $\mathcal L_n$ (and hence on $j_{n*}\mathcal L_n$), which will be useful for Hurwitz spaces later in Section \ref{sec: construction weight filtration}. The finite monodromy condition tells us that there is a finite étale cover $\pi_n\colon X_n\rightarrow \Sym^n_{\neq}(\C)$ such that $\mathcal L_n$ is a direct summand of $\pi_{n*}\underline{k}_{X_n}[n]$. One can construct a canonical lift $\Q^H\in \MHM(\text{pt})$ of $\Q$ with trivial Hodge structure. We pull this back to $X_n$ to get $\underline{\Q}^H_{X_n}(n/2)[n]\in \MHM(X_n)$ where we shifted degrees so that it remains perverse, and added a Tate-twist to normalize it such that it has weight zero. Pushing forward, this gives us the mixed Hodge module $M=\pi_{n*}\underline{\Q}^H_{X_n}(n/2)[n]$ of trivial type $(0,0)$ with underlying perverse sheaf $\pi_{n*}\underline{\Q}_{X_n}[n]$. In the case where we have $\mathcal L_n=\pi_{n*}\underline{k}_{X_n}[n]$, such as the case for Hurwitz spaces, then $\mathcal L_n$ is already defined over $\Q$ and $M$ is a lift of $\mathcal L_{n,\Q}$. In the general case, we base change $M\otimes_\Q K$ to a suitably large number field, after which we recover the $K$-mixed Hodge module lift of $\mathcal L_n$ as a direct summand.
\section{Equivalence of filtrations}\label{sec: equivalence of filtrations}
The goal of this section is to prove Theorem \ref{thm: equivalence of filtrations} with the strategy outlined in Equation \eqref{eqn: diagram for equivalence proof}. Let $\mathcal V$ be a braided monoidal abelian category of $k$-vector spaces, which means that the underlying objects of $\mathcal V$ are $k$-vector spaces and the tensor product is given by the tensor product $\otimes_k$ as $k$-vector spaces. The main application we care about is when $\mathcal V=\mathcal {YD}_G$ is the category of Yetter-Drinfeld modules.

We will first introduce the corresponding braided monoidal abelian category $\mathcal W$ which are finite $\Z$-indexed diagrams in $\mathcal V$. This is analogous to building graded $k[t]$-modules from $k$-vector spaces. The equivalence (a) follows immediately from the definitions, the equivalence (b) adapts the proof in \cite{KS20} to $\mathcal W$ where $\otimes$ is not bi-exact by essentially replacing the tensor products with derived tensor products. The equivalence in (c) is the most involved, and requires proving more technical lemmas about derived sheaves valued in $\mathcal W$, in particular a careful analysis of the derived exterior tensor product. Lastly, we use our construction to show that the spectral sequences computing the cohomologies of the algebra and perverse sheaves are isomorphic.
\subsection{Braided monoidal abelian category $\mathcal W$ of graded $k[t]$-modules}\label{sec: braided monoidal abelian category}
We define the braided monoidal abelian category $\mathcal W$ associated to $\mathcal V$ as follows. It has objects which are $\Z$-indexed finite diagrams $M=\{\cdots \rightarrow M_0\rightarrow M_1\rightarrow \cdots\}$ with $M_s\in \mathcal V$, where finite means that $M_s=0$ for $s\ll 0$ and $M_s\xrightarrow\sim M_{s+1}$ is an isomorphism for $s\gg 0$. We can view these as graded $k[t]$-modules with $V_s$ in grading $s$ with the degree $1$ element $t$ acting by the arrows in the diagram, and from this perspective the finite diagram condition is equivalent to $M$ being finitely generated (equivalently, finitely presented) as a $k[t]$-module. The morphisms are commuting maps of diagrams which preserve the degree. This is a $k$-linear abelian category even with the finitely generated condition because $k[t]$ is Noetherian. For notational convenience, for any $V\in \mathcal V$ and $s\in \Z$, we let $V\{t\}=\{\cdots\rightarrow 0\rightarrow V\xrightarrow\sim V\xrightarrow \sim\cdots\}\in \mathcal W$ to be the diagram of $V$s in degrees $\ge 0$, and for a diagram $M\in \mathcal W$ we denote $t^sM$ to be $M$ shifted $s$ spots to the right. These are analogous to tensoring with $k[t]$ and multiplying by $t^s$ respectively.

We define the tensor product structure on $\mathcal W$ such that on the underlying graded $k[t]$-structure it is simply given by the tensor product over $k[t]$. Given $M,N\in \mathcal W$, the tensor product $M\otimes_{\mathcal W}N$ has degree $s$ term given by
\begin{equation}\label{eqn: degree n term of tensor product over k[t]}
\left(\bigoplus_{u+v=s}M_u\otimes_k N_v \right) \Big/ (t\otimes_k 1-1\otimes_k t) \left(\bigoplus_{u+v=s-1}M_u\otimes_k N_v \right)
\end{equation}
where here we use $t$ to denote the diagram maps $M_u\rightarrow M_{u+1}$ and $N_v\rightarrow N_{v+1}$. The maps between consecutive degrees are given by $t\otimes_k 1=1\otimes_k t$. As both $M$ and $N$ are finitely generated, we can replace the infinite direct sum by a finite direct sum for $u,v$ over some bounded range, so this is indeed an element in the abelian category $\mathcal V$. The unit in $\mathcal W$ is $k\{t\}$ where $k$ is the unit in $\mathcal V$. The braiding $R_{M,N}\colon M\otimes_{\mathcal W}N\rightarrow N\otimes_{\mathcal W}M$ can be defined on each degree by taking the direct sum of the braidings $R_{M_u,N_v}$ over $u+v=s$ for bounded $u,v$ as above, and then descending to the quotient.
\subsubsection{Factorizable perverse sheaves over $\mathcal W$}\label{sec: factorizable perverse sheaves over W} We can define the category of primitive bialgebras $\text{PB}(\mathcal W)$ as in Section \ref{sec: primitive bialgebra}. However, in Section \ref{sec: factorizable perverse sheaves}, we defined the category of factorizable perverse sheaves assuming that $\otimes$ is biexact, however, this is not the case here as we are looking at $k[t]$-modules. We fix this by looking at the derived tensor product $\otimes^{\mathbf L}$ instead. 

Recall that for $K,L\in D^b(\mathcal W)$, the derived tensor product $K\otimes^\mathbf L L$ is given by replacing either $K$ or $L$ by a free resolution then taking the tensor product. Each object $M\in \mathcal W$ has a free resolution of length at most $1$, we construct this by taking a surjection $$\bigoplus_{s\in I} t^sM_s\{t\}\twoheadrightarrow M$$
for some bounded interval $I$ as $M$ is finitely generated. The kernel of this is also free over $k[t]$ because the submodule of any free module over a PID is free. Because of this, we say that the Tor-amplitude of $k[t]$ is $1$.

We now proceed analogously to Section \ref{sec: factorizable perverse sheaves} or \cite[Section 3.2]{KS20}, with $\otimes$ replaced by $\otimes^\mathbf L$. The braiding on $\mathcal W$ descends to a derived braiding on $D^b(\mathcal W)$ $$R_{K,L}\colon K\otimes^\mathbf L L \xrightarrow{\sim} L \otimes^\mathbf L K.$$
for $K,L\in D^b(\mathcal W)$, where we replace either $K$ or $L$ with a free resolution and then define the braiding on the projective resolutions with a Koszul sign twist. Then, in the same way, we have the external braiding isomorphism $$R_{\mathcal F,\mathcal G}\colon \mathcal F\boxtimes^\mathbf L \mathcal G\rightarrow \pi^*(\mathcal G\boxtimes^\mathbf L \mathcal F)$$ where $\pi\colon X\times Y\xrightarrow\sim Y\times X$, after which we can use the operadic point of view to define $\boxtimes_{(U_1,\ldots, U_m)}^\mathbf L\mathcal F_i$, so we can define a factorizable perverse sheaf of $\mathcal W$ in the same way as Definition \ref{def: FPS} except with $\boxtimes$ replaced by $\boxtimes^\mathbf L$. 
\subsection{Kapranov-Schechtman equivalence over $\mathcal W$} In this subsection, we prove (b) in Equation \eqref{eqn: diagram for equivalence proof} which is an analogue of Theorem \ref{thm: KS20} for $\mathcal W$.
\begin{proposition}\label{prop: b equivalence}
We have the equivalence
\[
\left\{
B \in \text{PB}(\mathcal W)
\text{ free over }k[t]
\right\}
\xleftrightarrow{\ \sim\ }
\left\{\mathcal G\in \text{FPS}(\mathcal W) \text{ with } \Phi_{\tot}(\mathcal G_n)\text{ free over }k[t]
\right\}.
\]
Furthermore, we have the analogue of Proposition \ref{prop: equivalence properties}(d), that is, if $B$ corresponds to $\mathcal G$ then the stalk $(\mathcal G_n)_0$ at the origin is isomorphic to the $n$-th bar-complex $B_n(B)$. 
\end{proposition}
Recall that we cannot directly apply Theorem \ref{thm: KS20} in the case of $\mathcal W$ to prove this because $\otimes$ is not biexact. In fact, the correspondence $\text{PB}(\mathcal W)\xleftrightarrow\sim\text{FPS}(\mathcal W)$ is false without the freeness condition. The main point is that $\boxtimes^\mathbf L$ may not preserve perversity. To see this in more detail, suppose instead that Theorem \ref{thm: KS20} and Proposition \ref{prop: equivalence properties} holds for $\mathcal W$. Then, let us take $B$ to be a primitive bialgebra such that $B_1$ is not free. By Proposition \ref{prop: equivalence properties}(a), we have that $L_1(B)$ is a constant sheaf with value $B_1[1]$. But this means that $L_1(B)|_{U_1}\boxtimes^\mathbf L L_1(B)|_{U_2}$ is a constant sheaf with value $(B_1\otimes^\mathbf L B_1)[2]$, but because $B_1$ is not free this is now supported in both degree $-2$ and $-3$ so it is not perverse. At the same time, the factorizability condition tells us that $L_1(B)|_{U_1}\boxtimes^\mathbf L L_1(B)|_{U_2}\cong a^*(L_2(B)|_{U_1\times U_2})$ which is perverse, giving a contradiction. 

If we had imposed freeness of $B_1$ in the above example, then there would have been no problem as $B_1\otimes^\mathbf L B_1$ is purely in degree zero. Indeed, we will see that freeness of $B$ and correspondingly the freeness of $\Phi_{\tot}(L_n(B))$ for all $n$ is exactly the condition that will make this an equivalence. 
\begin{proof}[Proof of Proposition \ref{prop: b equivalence}]
We follow the construction of the functors in both directions given in \cite{KS20} which we sketched in Section \ref{sec: KS equivalence} except that we change all tensor products $\otimes$ and exterior tensor products $\boxtimes$ to their derived versions $\otimes^\mathbf L$ and $\boxtimes^\mathbf L$. Most things go through as \cite{KS20} was mostly working in the derived category, and we only need to worry about when things are not concentrated in the correct degree because of the Tor-amplitude of $k[t]$, and we need to check that the freeness condition fixes this problem. This essentially boils down to checking that expressions that were meant to be non-derived are still non-derived and that sheaves that were perverse remain perverse. 

First, consider the forward functor $L$. Let $B\in \text{PB}(\mathcal W)$ be free over $k[t]$, which means each algebra grading $B_n$ is free as a $k[t]$-module. When constructing the Cousin complexes, \cite{KS20} has tensor products of the form $\bigotimes_iA_{\lambda_i}$ which are meant to be non-derived. In our context, even though we derive the tensor product as $\bigotimes^\mathbf L_iB_{\lambda_i}$, this is still equal to the non-derived version $\bigotimes_iB_{\lambda_i}$ because $B$ is free. Because these tensor products remain concentrated in degree zero, the bar-complex and cobar-complexes remain in degree zero. Then, the claim in \cite[Corollary 4.1.13]{KS20} that $\mathcal E^\alpha (B)$ is in degree zero and the claim in \cite[Proposition 4.2.17]{KS20} that $\mathcal E^\bullet_n(B)$ is perverse both remain true. Everything else works as intended, so we obtain a factorizable perverse sheaf $\mathcal G=(\mathcal E^\bullet_n(B))_n$ with vanishing cycles $\Phi_{\tot}(\mathcal G_n)\cong B_n$ which is free over $k[t]$. 

For the reverse functor $\Phi$, starting from some $\mathcal G\in \text{FPS}(\mathcal W)$ with $\Phi_{\tot}(\mathcal G_n)$ free over $k[t]$, it is clear that the primitive bialgebra $B$ we get is free over $k[t]$ by definition. Similar to the previous paragraph, tensor products which are meant to be non-derived are still in degree zero because of freeness. The rest of the construction already involves derived functors and only invokes results which do not assume the biexactness of $\otimes$, so the proof goes through as it should.  The analogue of Proposition \ref{prop: equivalence properties}(d) follows as well.
\end{proof}
\subsection{Lifting filtrations of primitive bialgebras} We first define carefully the notion of filtrations on objects in $\mathcal V$ and $\text{PB}(\mathcal V)$. By convention, we only look at increasing filtrations, as we can flip any decreasing filtration to an increasing one. 
\begin{definition}\label{def: finite filtration PB}
A finite increasing filtration on $V\in \mathcal V$ is a sequence $\cdots \subseteq F_0V\subseteq F_1V\subseteq \cdots$ such that $F_sV=0$ for $s\ll 0$ and $F_sV=V$ for $s\gg 0$. An increasing filtration on $A\in \text{PB}(\mathcal V)$ is a sequence $\cdots \subseteq F_0A\subseteq F_1A\subseteq \cdots$ such that the associated graded $A^{gr}=\bigoplus_{s\in \Z}\Gr_sA$ is also a primitive bialgebra which respects the grading $s$ and for each algebraic degree $n$ the filtration $F_sA_n$ is finite.
\end{definition}
We unwind the condition that $A^{gr}$ is a bialgebra as follows. Both multiplication and comultiplication needs to descend to $\mu^{gr}\colon A^{gr}\otimes A^{gr}\rightarrow A^{gr}$ and $\Delta^{gr}\colon A^{gr}\rightarrow A^{gr}\otimes A^{gr}$, which is equivalent to having $\mu(F_s(A\otimes A))\subseteq F_{s}A$ and $\Delta(F_sA)\subseteq F_s(A\otimes A)$ where $F_s(A\otimes A)=\bigcup_{u+v=s}F_uA\otimes F_vA\subseteq A\otimes A$. The unit and counit should be in grading $s=0$, i.e. $F_{-1}A_0=\{0\}$ and $F_0A_0=A_0=k$. Once we have these, the compatibility conditions for $A^{gr}$ will automatically descend from that of $A$.

Now, we prove (a) in Equation \eqref{eqn: diagram for equivalence proof} by lifting these filtrations to $B\in \text{PB}(\mathcal W)$. This turns out to be rather tautological as the finite generation corresponds to the finite filtration condition and the condition that $A^{gr}$ is a primitive bialgebra corresponds to the condition that $B$ is a primitive bialgebra. 
\begin{proposition}\label{prop: a equivalence}
We have the equivalence
\[
\left\{
\text{Filtrations of }
A\in \text{PB}(\mathcal V)
\right\}
\xleftrightarrow{\ \sim\ }
\left\{
B \in \text{PB}(\mathcal W)
\text{ free over }k[t]
\right\}.
\]
\end{proposition}
\begin{proof}
Suppose we have a filtration $F_sA$. From this we define $B=\bigoplus_{n=0}^\infty B_n$ where $B_n=\{\cdots \hookrightarrow F_0A_n\hookrightarrow F_1A_n\hookrightarrow \cdots\}$ is the diagram associated to $F_sA_n$. Each $B_n$ is free because the diagram maps are inclusions, and finitely generated because the filtration on each $A_n$ is finite.

By Equation \eqref{eqn: degree n term of tensor product over k[t]}, the degree $s$ part (with respect to the grading on $k[t]$) of $B\otimes_{k[t]} B$ is
$$\left(\bigoplus_{u+v=s}F_uA\otimes_k F_vA \right) \Big/ (t\otimes_k 1-1\otimes_k t) \left(\bigoplus_{u+v=s-1}F_uA\otimes_k F_vA \right),$$
where the maps $t$ here are simply inclusions. It is clear that when viewed inside $A\otimes A$ this is simply
$$F_s(A\otimes A)=\bigcup_{u+v=s} F_uA\otimes_k F_vA$$
as the quotient identifies elements that are the same in $A\otimes A$. 

Recall that since $A^{gr}$ is a primitive bialgebra, we have $\mu_A\colon F_s(A\otimes A)\rightarrow F_sA$ and $\Delta_A\colon F_sA\rightarrow F_s(A\otimes A)$. This this allows us to define the degree $s$ part of $\mu_B$ and $\Delta_B$ respectively, so $B$ is indeed a primitive bialgebra in $\mathcal W$.

Conversely, starting from a primitive bialgebra $B$, we can recover $F_sA$ by looking at the degree $s$ part of $B$. The same argument shows that this is indeed a filtration.
\end{proof}
\subsection{Lifting filtrations of factorizable perverse sheaves} We first define filtrations on $\text{FPS}(\mathcal V)$ in an analogous way to Definition \ref{def: finite filtration PB}.
\begin{definition}\label{def: filtration FPS}
A filtration of $\mathcal F=(\mathcal F_n)_n\in \text{FPS}(\mathcal V)$ is a sequence $\cdots \subseteq F_0\mathcal F \subseteq F_1\mathcal F\subseteq \cdots$ such that the associated graded $\mathcal F^{gr}=\bigoplus_{s\in \Z}\Gr_s\mathcal F$ is a factorizable perverse sheaf, and for each $n$ the filtration $F_s\mathcal F_n$ is finite, that is $F_s\mathcal F_n=0$ for $s\ll 0$ and $F_s\mathcal F_n=\mathcal F_n$ for $s\gg 0$.
\end{definition}
We want to prove (c) in Equation \eqref{eqn: diagram for equivalence proof} given as follows.
\begin{proposition}\label{prop: c equivalence}
We have the equivalence
\[
\left\{
\text{Filtrations of }
\mathcal F\in \text{FPS}(\mathcal V)
\right\}
\xleftrightarrow{\ \sim\ }
\left\{
\mathcal G \in \text{FPS}(\mathcal W)\text{ with }
\Phi_{\tot}(\mathcal G_n)\text{ free over }k[t]
\right\}.
\]
\end{proposition}
This turns out to be the most involved out of the three equivalences, and we need to do some setup before the proof. 
\subsubsection{Sheaves valued in $\mathcal W$} We will show that (constructible, derived, perverse) sheaves in $\mathcal W$ are given by finite diagrams of (constructible, derived, perverse) sheaves in $\mathcal V$.

Let $\Shv(X,S,\mathcal C)$ be the constructible sheaves on $X$ with respect to $S$ valued in $\mathcal C$. Recall that a $\Z$-indexed finite diagram in an abelian category $\mathcal C$ is $\{\cdots \rightarrow C_0\rightarrow C_1\rightarrow \cdots\}$ where $C_s=0$ for $s\ll 0$ and $C_s\xrightarrow \sim C_{s+1}$ for $s\gg 0$. 

We have the following proposition which basically says that taking $\Z$-indexed finite diagrams commutes with both the derived category construction and the derived functors $f^*,f_*,f^!,f_!,\Phi_{\tot}$. This is not surprising as these constructions are ``geometric'' in the sense that they only care about the underlying vector space structure which is just a direct sum over all $t$-degrees, without looking at the tensor product in $\mathcal W$ where the algebraic structure of $k[t]$ will come in. 
\begin{proposition}\label{prop: sheaves and chains}
Let $S$ be a finite stratification for $X$. There are equivalences between the following pairs of categories.
\begin{enumerate}[(\alph*)]
\item $\Shv(X,S,\mathcal W)$ and $\Z$-indexed finite diagrams in $\Shv(X,S,\mathcal V)$.
\item $C^b(X,S,\mathcal W)$ and $\Z$-indexed finite diagrams in $C^b(X,S,\mathcal V)$.
\item $D^b(X,S,\mathcal W)$ and $\Z$-indexed finite diagrams in $D^b(X,S,\mathcal V)$.
\item $\Perv(X,S,\mathcal W)$ and $\Z$-indexed finite diagrams in $\Perv(X,S,\mathcal V)$.
\end{enumerate}
Furthermore, the equivalence in (c) commutes with the derived functors $f^*,f_*,f^!,f_!$ and $\Phi_{\tot}$.
\end{proposition}
\begin{proof}
For (a), given a $\Z$-indexed finite diagram $\{\cdots \rightarrow \mathcal F_0\rightarrow \mathcal F_1\rightarrow \cdots\}$ in $\Shv(X,S,\mathcal V)$, we construct $\mathcal G\in \Shv(X,S,\mathcal W)$ as follows. For any open $U\subseteq X$, let $\mathcal G(U)=\{\cdots \rightarrow \mathcal F_0(U)\rightarrow \mathcal F_1(U)\rightarrow \cdots\}\in \mathcal W$ as it is a finite diagram. We can check that this satisfies the sheaf condition degree by degree. Conversely, given $\mathcal G\in \Shv(X,S,\mathcal W)$, we define $\mathcal F_s(U)$ to be the degree $s$ part of $\mathcal G(U)$, and similarly $\mathcal F_s$ satisfies the sheaf property. For each $U$ it is true that $\mathcal F_s(U)=0$ for $s\ll 0$ and $\mathcal F_s(U)\xrightarrow\sim \mathcal F_{s+1}(U)$ for $s\gg 0$. Because our stratification is finite, this means that $\mathcal F_s=0$ for $s\ll 0$ and $\mathcal F_s\xrightarrow\sim \mathcal F_{s+1}$ for $s\gg 0$ so the corresponding diagram is finite.

(b) and (c) follows formally from (a) by choosing a representing chain complex and checking that quasi-isomorphisms agree on both sides. The finite diagram condition ensures that our chain complexes are still bounded. For (d), we can check perversity using stalks and costalks, and it easily reduces to the statement that taking stalks in $D^b(X,S,\mathcal W)$ is the same as taking stalks pointwise in the diagram, which is a subcase of what we prove next.

It remains to show that the equivalence in (c) commutes with the functors listed above. The equivalence (a) commutes with the non-derived $f_*$ and $f_!$ by our construction, and so the derived functor commutes with (c) as we can replace each object in the diagram by its injective resolution. From this, we conclude that the equivalence in (c) commutes with $f^*$ and $f^!$ via adjunction, and $\Phi_{\tot}$ because it can be expressed in the previous four functors.
\end{proof}
\subsubsection{Derived tensor on sheaves valued in $\mathcal W$} Unlike the functors above, the story is not that simple for the tensor product because it is not completely geometric and involves the algebraic structre of $k[t]$. Our discussion here will be similar to the derived tensor product for graded $k[t]$-modules. Recall that to take $\mathcal F\boxtimes^\mathbf L\mathcal G$ we need to replace either $\mathcal F$ or $\mathcal G$ by a flat resolution then take the regular $\boxtimes$. For $K\in D^b(X,S,\mathcal V)$, define $K\{t\}\in D^b(X,S,\mathcal W)$ and the shift by $t^s$ in $D^b(X,S,\mathcal W)$ just like at the start of Section \ref{sec: braided monoidal abelian category}. 

By unwinding definitions for tensor product of sheaves, it is easy to see that $t^sK\{t\}$ is a flat complex of sheaves. Furthermore, given $M=\{\cdots \rightarrow M_0\rightarrow M_1\rightarrow \cdots \}\in D^b(Y,T,\mathcal W)$, we have $t^sK\{t\}\boxtimes^\mathbf L_{\mathcal W} M = t^sK\{t\}\boxtimes_{\mathcal W} M = t^s K\boxtimes_{\mathcal V} M$ where we write $K\boxtimes_{\mathcal V} M=\{\cdots \rightarrow K\boxtimes_{\mathcal V} M_0\rightarrow K\boxtimes_{\mathcal V} M_1\rightarrow \cdots\}$ where $K\boxtimes_{\mathcal V} M_0$ is in degree $t^s$. We will drop the subscripts $\mathcal V$ and $\mathcal W$ on $\boxtimes$ when the context is sufficiently clear.

Given $\mathcal F =\{\cdots \rightarrow \mathcal F_0\rightarrow \mathcal F_1\rightarrow \cdots \}\in \Perv(X,S,\mathcal W)$, we can construct a flat resolution in $D^b(\Perv(X,S,\mathcal W))=D^b(X,S,\mathcal W)$ given by
$$\left\{\bigoplus_{s\in I} t^{s+1}\mathcal F_s\{t\}\rightarrow \bigoplus_{s\in I}  t^s\mathcal F_s\{t\} \right\}$$
where the complex is in perverse degeree $-1,0$. Here, the maps are given by $t^{s+1}\mathcal F_s\{t\}\rightarrow t^s\mathcal F_s\{t\}\oplus t^{s+1}\mathcal F_{s+1}\{t\}$ where in degree $t^i$ for $i\ge s+1$ it is $\mathcal F_s\rightarrow \mathcal F_s\oplus \mathcal F_{s+1}$ where the first component is identity and the second is given by the diagram map in $\mathcal F$. Again, here we are able to restrict the sum over $s$ to a finite interval $I$. To see that this is indeed quasi-isomorphic to $\mathcal F$, we note that the complex $\{t^{s+1}\mathcal F_s\{t\}\rightarrow t^s\mathcal F_s\{t\}\}$ is quasi-isomorphic to $t^s\mathcal F_s$ in perverse degree $0$, and because of the way the maps are defined, we recover $\mathcal F$ when combining over all $s$.

Hence, if we have $\mathcal G=\{\cdots \rightarrow \mathcal G_0\rightarrow \mathcal G_1\rightarrow \cdots\}\in \Perv(Y,T,\mathcal W)$, we can get $\mathcal F\boxtimes^\mathbf L\mathcal G$ by
$$\left\{\bigoplus_{s\in I} t^{s+1}\mathcal F_s\boxtimes_{\mathcal V}\mathcal G\rightarrow \bigoplus_{s\in I}  t^s\mathcal F_s\boxtimes_{\mathcal V}\mathcal G \right\}$$
in perverse degree $-1,0$. In degree $t^s$ this would be
\begin{equation}\label{eqn: degree t^s derived exterior product as complex}
\left\{\bigoplus_{u+v=s-1}\mathcal F_u\boxtimes G_v\rightarrow \bigoplus_{u+v=s}\mathcal F_u\boxtimes G_v \right\}
\end{equation}
induced by the maps in $\mathcal F$ and $\mathcal G$. Again, the sum over $u$ and $v$ can be made to be finite. Taking cohomology in perverse degree $0$, we recover the analogue of Equation \eqref{eqn: degree n term of tensor product over k[t]} 
\begin{equation}\label{eqn: degree t^s derived exterior product as complex as quotient}
\left(\bigoplus_{u+v=s}\mathcal F_u\boxtimes \mathcal G_v \right) \Big/ (t\boxtimes 1-1\boxtimes t) \left(\bigoplus_{u+v=s-1}\mathcal F_u\boxtimes \mathcal G_v \right),
\end{equation}
where again, abusing notation, we use $t$ here to denote the maps in $\mathcal F$ and $\mathcal G$. This is not surprising as the perverse cohomology of the derived tensor product in degree zero should just be the classical tensor product. 

For the cohomology in perverse degree $-1$, we prove the following lemma. Say that $\mathcal  F= \{\cdots \rightarrow \mathcal F_0\rightarrow \mathcal F_1 \rightarrow \cdots\}\in \Perv(X,S,\mathcal W)$ is free if all the maps are monomorphisms.
\begin{lemma}\label{lem: exterior product perverse}
If either $\mathcal F$ or $\mathcal G$ is free, the exterior product $\mathcal F \boxtimes^\mathbf L\mathcal G$ is perverse, i.e. the cohomology in perverse degree $-1$ vanishes.
\end{lemma}
\begin{proof}
Suppose the maps $\mathcal F_u\hookrightarrow \mathcal F_{u+1}$ in $\mathcal F$ are monomorphisms. As $\otimes$ is bi-exact for $\mathcal V$, the exterior product $\boxtimes$ is also bi-exact, so we have monomorphisms $\mathcal F_u\boxtimes \mathcal G_v\hookrightarrow \mathcal F_{u+1}\boxtimes \mathcal G_v$. We want to show that Equation \eqref{eqn: degree t^s derived exterior product as complex} is a monomorphism, it is the finite direct sum of maps of the form
\[\begin{tikzcd}
	\cdots & {A_0} & {A_1} & {A_2} & \cdots \\
	\cdots & {B_0} & {B_1} & {B_2} & \cdots .
	\arrow[from=1-1, to=2-2]
	\arrow[hook, from=1-2, to=2-2]
	\arrow[from=1-2, to=2-3]
	\arrow[hook, from=1-3, to=2-3]
	\arrow[from=1-3, to=2-4]
	\arrow[hook, from=1-4, to=2-4]
	\arrow[from=1-4, to=2-5]
\end{tikzcd}\]
We show this is a monomorphism in any abelian category. Suppose we are in the category of abelian groups, then this is injective because if $(x_i)\in \bigoplus_{i\in I} A_i$ goes to zero where $I$ is a finite interval, then we consider the lowest index $i$ where $x_i\neq 0$, but this maps to a nonzero element in $B_i$ and $x_{i-1}=0$ maps to zero in $B_i$, giving a contradiction. For general abelian categories, we adapt the proof using the Yoneda lemma by looking at the Homs from an object $Z$ which are abelian groups, and use the fact that a map $A\rightarrow B$ is a monomorphism if and only $\Hom(Z,A)\rightarrow \Hom(Z,B)$ is a monomorphism for all $Z$.
\end{proof}
There is a natural extension of the lemma given as follows.
\begin{lemma}\label{lem: free x free = free}
If both $\mathcal F$ and $\mathcal G$ are free, then the perverse sheaf $\mathcal F\boxtimes^\mathbf L\mathcal G$ is free. Furthermore, we have
$$(\mathcal F\boxtimes^\mathbf L\mathcal G)_s/(\mathcal F\boxtimes^\mathbf L\mathcal G)_{s-1}=\bigoplus_{u+v=s} (\mathcal F_u/\mathcal F_{u-1})\boxtimes (\mathcal G_v/\mathcal G_{v-1}).$$
\end{lemma}
\begin{proof}
The previous lemma tells us that $\mathcal F\boxtimes^\mathbf L\mathcal G$ is perverse, so it is equal to its perverse cohomology in degree $0$ given by Equation \eqref{eqn: degree t^s derived exterior product as complex as quotient}. The map from degree $t^{s-1}$ to $t^s$ is $t\boxtimes 1=1\boxtimes t$ as both maps have been identified in the quotient, we want to check that this is a monomorphism, with quotient as given in the equation above.

As in the previous lemma, we want to pretend that we are working with a category of abelian groups, where we change the perverse sheaves $\mathcal F_u,\mathcal G_v$ to abelian groups $A_u,B_v$ and $\boxtimes$ to $\otimes$ where $A_u\hookrightarrow A_{u+1}$ and $B_v\hookrightarrow B_{v+1}$ for all $u,v$. If this were true, then Equation \eqref{eqn: degree t^s derived exterior product as complex as quotient} is simply the union $\bigcup_{u+v=s} A_u \otimes B_v$ in $(\bigcup_u A_u) \otimes (\bigcup_v B_v)$, the same argument was made in the proof of Proposition \ref{prop: a equivalence}. Then, it is clear that we have the injection $\bigcup_{u+v=s-1} A_u \otimes B_v\hookrightarrow \bigcup_{u+v=s} A_u \otimes B_v$. Furthermore, it is clear that the quotient is $\bigoplus_{u+v=s} (A_u/A_{u-1})\otimes (B_v/B_{v-1})$ as desired.

To reduce to the case of abelian groups, we again look at the Homs from objects in $\Perv(X\times Y,S\times T, \mathcal W)$. By \cite{Lyubashenko_2001}, the category $\Perv(X\times Y,S\times T, \mathcal V)$ is the Deligne tensor product of abelian categories $\Perv(X,S,\mathcal V) \boxtimes \Perv(Y,T,\mathcal V)$ introduced in \cite[Proposition 5.13]{Tannakian}, and the exterior tensor product matches the Deligne tensor product. Thus, we for $\mathcal K\in \Perv(X,S,\mathcal V)$ and $\mathcal L\in \Perv(Y,T,\mathcal V)$ that $\Hom(\mathcal K\boxtimes \mathcal L,\mathcal F_u\boxtimes \mathcal G_v) \cong \Hom(\mathcal K,\mathcal F_u)\otimes \Hom(\mathcal L, \mathcal G_v)$, so we can indeed take $A_u=\Hom(\mathcal K,\mathcal F_u)$ and $B_v=\Hom(\mathcal L, \mathcal G_v)$. However, we can only test Homs from objects of the form $\mathcal K\boxtimes \mathcal L$. To fix this, let $\mathcal K=\bigoplus_{u\in I} \mathcal F_u$ be the direct sum in the finite range $I$ of $u$, and $\langle \mathcal K\rangle $ be the abelian subcategory of $\Perv(X,S,\mathcal V)$ generated by the projective generator $\mathcal K$, see \cite[Section 5.12]{Tannakian}, this contains all $\mathcal F_u$. Likewise, define $\mathcal L=\bigoplus_{v\in J} \mathcal G_v$ and $\langle \mathcal L\rangle$. Then, the abelian subcategory $\langle \mathcal K\boxtimes \mathcal L\rangle=\langle \mathcal K\rangle\boxtimes \langle\mathcal L\rangle$ of $\Perv(X\times Y,S\times T, \mathcal V)$ contains all the relevant objects we are working with. Furthermore, there is a projective generator $P=\mathcal K\boxtimes \mathcal L$, so for any object $Z\in \langle \mathcal K\boxtimes \mathcal L\rangle$ there is a projective resolution $P^{\oplus m}\rightarrow P^{\oplus n}\rightarrow Z\rightarrow 0$ so we can recover the Hom-sets from $0\rightarrow \Hom(Z,-)\rightarrow \Hom(P^{\oplus n},-)\rightarrow \Hom(P^{\oplus m},-)$, and injectivity is preserved by the four lemma.
\end{proof}
\subsubsection{Proof of equivalence} We bring everything together to finish the proof.
\begin{proof}[Proof of Proposition \ref{prop: c equivalence}]
Suppose we have a filtration $F_s\mathcal F_n$ of $\mathcal F\in \text{FPS}(\mathcal V)$. By Proposition \ref{prop: sheaves and chains}, in each algebraic degree $n$ we can view the diagram $\{\cdots \hookrightarrow F_0\mathcal F_n\hookrightarrow F_1\mathcal F_n\hookrightarrow \cdots \}$ as $\mathcal G_n\in \Perv(\Sym^n(\C),S_n,\mathcal W)$. Since this commutes with $\Phi_{\tot}$, we see that $\Phi_{\tot}(\mathcal G_n)=\{\cdots \hookrightarrow \Phi_{\tot}F_0\mathcal F_n\hookrightarrow \Phi_{\tot}F_1\mathcal F_n\hookrightarrow\cdots \}$ is free over $k[t]$ because $\Phi_{\tot}$ is exact.

Next, we check the factorizability of $\mathcal G=(\mathcal G_n)$. By associativity it suffices to check for a product of two sheaves that the map
\begin{equation}\label{eqn: factorizable equation for G}
\mathcal G_n|_{\Sym^n(U)}\boxtimes^\mathbf L_{(U,V)} \mathcal G_m|_{\Sym^m(V)} \xrightarrow \sim a^*(\mathcal G_{n+m}|_{\Sym^n(U)\times \Sym^m(V)}) 
\end{equation}
is an isomorphism. Denote the LHS by $\mathcal K=\{\cdots\rightarrow \mathcal K_0\rightarrow \mathcal K_1\rightarrow \cdots\}$ and the RHS by $\mathcal L=\{\cdots\rightarrow \mathcal L_0\rightarrow \mathcal L_1\rightarrow \cdots\}$. By Lemma \ref{lem: free x free = free}, since both sheaves on the left are free, $\mathcal K$ is perverse and free with quotients given by 
$$\mathcal K_s/\mathcal K_{s-1}=\bigoplus_{u+v=s} \Gr_u \mathcal F_n|_{\Sym^n(U)}\boxtimes \Gr_v \mathcal F_m|_{\Sym^m(V)}.$$
On the other hand, we have 
$$\mathcal L_s/\mathcal L_{s-1}= a^*(\Gr_s \mathcal F_{n+m}|_{\Sym^n(U)\times \Sym^m(V)})$$
and these are isomorphic because we required $\mathcal F^{gr}$ to be a factorizable perverse sheaf in Definition \ref{def: filtration FPS}. Now, there is some $s\ll 0$ where $\mathcal K_s=\mathcal L_s=0$, and from there we can use the five lemma to conclude that $\mathcal K_s\xrightarrow\sim \mathcal L_s$ for all $s$. This is compatible with operadic compositions because $\mathcal F^{gr}$ is, and this concludes the proof of factorizability.

For the other direction, suppose we have $\mathcal G\in \text{FPS}(\mathcal W)$ with $\Phi_{\tot}(\mathcal G_n)$ free over $k[t]$ for all $n$. We first show that $\mathcal G_n$ is free for all $n$. Otherwise, take the smallest $n$ such that some $\mathcal G_n$ is not free. Note that we cannot have $n=0$, because in this case $\Sym^0(\C)$ is trivial so $\mathcal G_0=\Phi_{\tot}(\mathcal G_0)$ is free (in fact, by factorizability, we necessarily have $\mathcal G_0=k\{t\}$ for all $\mathcal G\in \text{FPS}(\mathcal W)$). Thus, we have $n>0$ and we consider the kernel $0\rightarrow K\rightarrow (\mathcal G_n)_{s-1}\rightarrow (\mathcal G_n)_s$ where $K\neq 0$ for some $s$. By factorizability, we have $a^*(\mathcal G_n|_{\Sym^p(U)\times \Sym^p(V)} )\cong \mathcal G_p|_{\Sym^p(U)}\boxtimes^\mathbf L_{(U,V)}\mathcal G_q|_{\Sym^q(V)}$ for $p+q=n$, and by minimality of $n$, the sheaves on the RHS are free and by Lemma \ref{lem: free x free = free} this implies that the LHS is free too. Since restricting to an open is $t$-exact, this means that $K$ has no support on any $\Sym^p(U)\times \Sym^p(V)$, in particular $K$ is concentrated on the diagonal $\{x_1=x_2=\cdots = x_n\}$. On the other hand, applying the t-exact functor $\Phi_{\tot}$ we get that $\Phi_{\tot}(K)=0$ because $\Phi_{\tot}(\mathcal G_n)_{s-1}\hookrightarrow \Phi_{\tot}(\mathcal G_n)_s$ as $\Phi_{\tot}\mathcal G_n$ is free by assumption. Combining these two facts gives us $K=0$ as the nearby cycles and vanishing cycles are both zero so the stalk at the diagonal is also zero.

Now, we can define the filtration on $\mathcal F$ by writing $\mathcal G_n=\{\cdots \hookrightarrow F_0\mathcal F_n\hookrightarrow F_1\mathcal F_n\hookrightarrow \cdots\}$ which is a finite diagram, so it satisfies the finiteness condition in Definition \ref{def: filtration FPS}. The fact that $\mathcal F^{gr}$ is factorizable follows directly from $\mathcal G$ being factorizable by simply taking the quotient of the degree $t^s$ term by the $t^{s-1}$ term in Equation \eqref{eqn: factorizable equation for G}.
\end{proof}
The first part of Theorem \ref{thm: equivalence of filtrations} then follows from Proposition \ref{prop: b equivalence}, \ref{prop: a equivalence} and \ref{prop: c equivalence}.
\subsection{Spectral sequences}\label{sec: spectral sequences} We are left with the claim of spectral sequences in Theorem \ref{thm: equivalence of filtrations}, which we prove here. From Proposition \ref{prop: b equivalence}, we have that if $B\in \text{PB}(\mathcal W)$ corresponds to $\mathcal G\in \text{FPS}(\mathcal W)$ under the equivalence stated there, then we have the isomorphism $(\mathcal G_n)_0\cong B_n(B)$. We note that the stratification is quasi-homogeneous, so $R\Gamma \mathcal G_n\cong (\mathcal G_n)_0$ just like at the end of Section \ref{sec: KS equivalence}.

Now, we transport $B$ and $\mathcal G$ to filtrations of $A\in \text{PB}(\mathcal V)$ and $\mathcal F\in \text{FPS}(\mathcal V)$ respectively using Proposition \ref{prop: a equivalence}, \ref{prop: c equivalence}. In the proofs above we constructed $B_n=\{\cdots \hookrightarrow  F_0A_n\hookrightarrow F_1A_n\hookrightarrow \cdots\}$ and $\mathcal G_n=\{\cdots \hookrightarrow F_0\mathcal F_n\hookrightarrow F_1\mathcal F_n\hookrightarrow \cdots \}$. With this, we see that the bar-complex $B_n(B)$ is simply the filtration of the bar-complex of $B_n(A)$ induced by the filtration on $A$. Likewise, $R\Gamma\mathcal G_n$ is simply the filtration on $R\Gamma \mathcal F_n$.

Hence, the filtrations on $B_n(A)$ and $R\Gamma \mathcal F_n$ are isomorphic, so their spectral sequences 
\begin{equation}\label{eqn: spectral sequences of filtration}
\begin{split}
E_1^{p,q}= R\Gamma^{p+q}\mathcal F^{gr}_n\cong R\Gamma^{p+q}(\Gr_{p}\mathcal F_n)&\Rightarrow R\Gamma^{p+q}\mathcal F_n\\
E_1^{p,q} = \Tor^{A^{gr}}_{-(p+q),n}(k,k)\cong H^{p+q}(\Gr_{p} B_n(A))&\Rightarrow H^{p+q}(B_n(A))\cong \Tor^A_{-(p+q),n}(k,k)
\end{split}
\end{equation}
are isomorphic. This finishes the proof of Theorem \ref{thm: equivalence of filtrations}.

Finally, we prove an additional compatibility result between two different filtrations, which we will need later for the proof of Proposition \ref{prop: equiv weight filtration}. Let $F_s A$ and $G_s A$ be two filtrations on $A\in \text{PB}(\mathcal V)$ which correspond to the filtrations $F_s\mathcal F$ and $G_s\mathcal F$ on $\mathcal F\in \text{FPS}(\mathcal V)$ respectively. Consider the bi-filtrations $H_{u,v}A_n=F_uA_n\cap G_vA_n$ and $H_{u,v}\mathcal F_n=F_u\mathcal F_n\cap G_v\mathcal F_n$, where the second intersection is taken in the abelian category of perverse sheaves. Then, the following proposition generalizes our previous discussion to the case of bi-filtrations.
\begin{proposition}\label{prop: bi-filtrations}
The induced bi-filtration $H_{u,v}B_n(A)$ on the bar-complex is isomorphic to the bi-filtration $R\Gamma H_{u,v}\mathcal F_n$.
\end{proposition}
\begin{proof}
We prove this by adapting the arguments from this section to the bi-filtered setting, replacing $k[t]$-modules with $k[x,y]$-modules and $\Z$-indexed diagrams with $\Z^2$-indexed diagrams. Just like how a filtration on a finite $k$-vector space can be viewed as a graded $k[t]$-module, a bi-filtration on a finite $k$-vector space can be viewed as a graded $k[x,y]$-module. The key fact we need is that this $k[x,y]$-module is free. This is not true for tri-filtrations, for example, consider a two-dimensional vector space and three different length two filtrations on it.

We prove freeness as follows. Let $V$ be a vector space with basis $e_1,\ldots ,e_n$, $G=\GL(V)$, and $H_{u,v}V=F_uV\cap G_vV$ be a bi-filtration on $V$. We assume for simplicity that $0=F_0V\subsetneq F_1V\subsetneq \cdots \subsetneq F_nV=V$ and $0=G_0V\subsetneq G_1V\subsetneq \cdots \subsetneq G_mV=V$, and now we can view these as flags $g_1\in G/P_1$ and $g_2\in G/P_2$ in the respective flag varieties. To be precise, $P_1$ and $P_2$ are standard parabolic subgroups which stabilize the standard flags with subspaces of the correct dimension, where a standard flag is defined such that a $d$-dimensional subspace has basis $e_1,\ldots, e_d$. Then, the filtrations $F_uV$ and $G_vV$ are obtained respectively by applying $g_1$ and $g_2$ to the standard flags. Using the isomorphism $G/P_1\times G/P_2\cong P_1\backslash G /P_2$ which sends $(g_1,g_2)\mapsto g_1^{-1}g_2$, we see that such a bi-filtration corresponds to a double coset. However, by the Bruhat decomposition for parabolic subgroups, we can write $P_1\backslash G/P_2$ as the disjoint union of $P_1wP_2$ over permutation matrices $w$. In particular $g_1^{-1}g_2\in P_1wP_2$ for some $w$, and unwinding this we see that there is a basis $f_i$ (obtained by $g_1$ acting on $e_i$) such that $F_uV$ is a standard flag in this basis while $G_vV$ is a standard flag in the basis $f_{w^{-1}(i)}$. Then, it is clear that the graded $k[x,y]$-module corresponding to the bi-filtration $H_{u,v}V$ is free with generators $f_i$ placed in degree $(u,v)$ where $u,v$ are the smallest integers such that $f_i\in H_{u,v}V$.

Going back to our situation, let $\mathcal W'$ be the braided monoidal abelian category with objects $\Z^2$-indexed finite diagrams which are finitely generated when viewed as $k[x,y]$-modules. From the bi-filtration $H_{u,v}A_n$, we define $B'\in \text{PB}(\mathcal W')$ in basically the same way as the proof of Proposition \ref{prop: a equivalence} by letting $n$-th graded part $B_n'$ be the $\Z^2$-indexed diagram with entries $H_{u,v}A_n$. We check that $B'$ is indeed a primitive bialgebra in the same way, and moreover it is free over $k[x,y]$ by our discussion above. On the other hand, we can define $\mathcal G'\in \text{FPS}(\mathcal W')$ in the same way as the proof of Proposition \ref{prop: c equivalence} by letting $\mathcal G'_n$ be the $\Z^2$-indexed diagram with entries $H_{u,v}\mathcal F_n$. By adapting the proof there, we see that factorizability follows analogously. As $\Phi_{\tot}$ is exact, it commutes with intersections, so $\Phi_{\tot}(\mathcal G_n')$ is free over $k[x,y]$. The analogue of Proposition \ref{prop: b equivalence} for $\mathcal W'$ and $k[x,y]$ holds, and it is easy to see that $B'$ corresponds to $\mathcal G'$, so by the latter part of the statement $R\Gamma \mathcal G_n$ is isomorphic to $B_n(B')$, which tells us that the bi-filtrations are the same as desired.
\end{proof}

\section{Equivalence of word length and codimension filtration}\label{sec: equivalence of word length and codim filtration}
In this section, we show that the shifted word length filtration and codimension filtration correspond under the equivalence in Theorem \ref{thm: equivalence of filtrations}. Let $\mathcal V$ be a braided monoidal abelian category of $k$-vector spaces, and $V\in \mathcal V$ a braided vector space. Let $j\colon \Sym_{\neq}(\C)\hookrightarrow \Sym(\C)$ and $j_n\colon \Sym^n_{\neq}(\C)\hookrightarrow \Sym^n(\C)$. Recall from Proposition \ref{prop: equivalence properties}(b) that $\mathfrak A=\mathfrak A(V)\in \text{PB}(\mathcal V)$ corresponds to $j_*\mathcal L\in \text{FPS}(\mathcal V)$ where $\mathcal L=\mathcal L(V[1])$ was defined in Example \ref{eg: compatible local systems}. Proposition \ref{prop: equivalence properties}(e) tells us that the dual tensor algebra $T=T(V^*)\in \text{PB}(\mathcal V)$ corresponds to the dual sheaf $\D j_*\mathcal L=j_!\mathcal L^\vee\in \text{FPS}(\mathcal V)$. Specializing Theorem \ref{thm: equivalence of filtrations} to this setting, we see that filtrations on $\mathfrak A$ correspond to filtrations on $j_*\mathcal L$ and filtrations on $T$ corresponds to filtrations on $j_!\mathcal L^\vee$. 

Recall from the introduction that we defined the increasing shifted word length filtration to be $$F_c\mathfrak A_n=((\mathfrak A_{>0})^{n-c})_n$$ for $0\le c< n$, with $F_{-1}\mathfrak A_n=\{0\}$ and $F_n\mathfrak A_n=\mathfrak A_n$. In other words, $F_c\mathfrak A_n$ has elements in algebraic grading $n$ that can be expressed as a linear combination of the shuffle product of $n-c$ words. Note that $F_0\mathfrak A$ is simply the subalgebra generated in degree $1$ which is the Nichols algebra $\mathfrak B=\mathfrak B(V)$, which is one reason why we chose to use the shifted filtration. 

We need to check that this is a filtration in the sense of Definition \ref{def: finite filtration PB}. It clearly satisfies the finiteness conditions, and the multiplication respects the grading by definition. Furthermore, \cite[Section 4.2]{ETW17} gives an argument for why comultiplication respects the unshifted filtration, and although they work in the context of Yetter-Drinfeld modules, we note that we can remove this assumption in their proof because the braiding always sends $\mathfrak A_{>0}\otimes \mathfrak A$ to $\mathfrak A\otimes \mathfrak A_{>0}$. From this it is clear that this implies that comultiplication respects the shifted filtration as well.

On the geometric side, we define the codimension filtration as follows. Let $\mathcal F=j_*\mathcal L$ and $\mathcal F_n=j_{n*}\mathcal L_n$. Now, fix some $n$, and let $V_c$ be the union of all codimension $\le c$ (i.e. dimension $\ge n-c$) strata in the stratificaton $S_n$ of $\Sym^n(\C)$. For example, we have $V_0=\Sym^n_{\neq}(\C)$ and $V_n=\Sym^n(\C)$. Let $u_c\colon \Sym^n_{\neq}(\C)\hookrightarrow V_c$ and $v_c\colon V_c\hookrightarrow \Sym^n(\C)$. We want the filtered piece $F_c\mathcal F_n$ to come as middle extensions from codimension $\le c$ strata, so we define
$$F_c\mathcal F_n=v_{c!*}u_{c*}\mathcal L_n,$$
for $0\le c\le n$ and by the property of middle extensions this is exactly the subobject of $j_{n*}\mathcal L_n$ consisting of all composition factors that are supported on $V_c$. We set $F_{-1}\mathcal F_n=0$, and note that $F_0\mathcal F_n=j_{n!*}\mathcal L_n$ and $F_n\mathcal F_n=\mathcal F_n$. Hence, the degree $0$ part of $j_*\mathcal L$ is exactly $j_{!*}\mathcal L$, analogously to the situation for algebras above which we expect because $j_{!*}\mathcal L$ corresponds to $\mathfrak B$.

We check that this is a filtration in the sense of Definition \ref{def: filtration FPS}, the finiteness condition is clear and the sheaves are clearly perverse, so it suffices to check that $\mathcal F^{gr}$ is factorizable. By factorizability of $\mathcal F$, we have 
$$\mathcal F_n|_{\Sym^n(U)}\boxtimes^\mathbf L_{(U,V)} \mathcal F_m|_{\Sym^m(V)} \xrightarrow \sim a^*(\mathcal F_{n+m}|_{\Sym^n(U)\times \Sym^m(V)})$$
and it suffices to check that this induces an isomorphism of subquotients
\begin{equation}\label{eqn: factorizability check}
\bigoplus_{u+v=s}\mathcal \Gr_u\mathcal F_n|_{\Sym^n(U)}\boxtimes_{(U,V)} \Gr_v \mathcal F_m|_{\Sym^m(V)} \xrightarrow \sim a^*(\Gr_{u+v}\mathcal F_{n+m}|_{\Sym^n(U)\times \Sym^m(V)}).
\end{equation}
By Lemma \ref{lem: free x free = free}, the LHS is a graded piece of a filtration on $\mathcal F_n|_{\Sym^n(U)}\boxtimes_{(U,V)} \mathcal F_m|_{\Sym^m(V)}$. We check that the composition factors on both sides are the same. But this is true since the graded pieces $\Gr_u\mathcal F_n$ consist of exactly those composition factors which are middle extensions from codimension $u$ strata, and similarly for $\Gr_v\mathcal F_n$ and $\Gr_{u+v}\mathcal F_{n+m}$. Since the exterior tensor product commutes with middle extensions, the LHS has all composition factors which are middle extensions from codimension $u+v$, which is exactly the RHS. From this, we can conclude that they are isomorphic. 
\subsection{Proof of equivalence} We prove that the two filtrations defined above correspond.
\begin{proposition}\label{prop: word length codim equivalence}
The shifted word length filtration $F_c\mathfrak A_n$ corresponds to the codimension filtration $F_c\mathcal F_n$ under the equivalence of Theorem \ref{thm: equivalence of filtrations}.
\end{proposition}
\begin{proof}
Let $F'_c\mathfrak A_n$ be the filtration that corresponds to the codimension filtration under Theorem \ref{thm: equivalence of filtrations}, and we show that $F'_c\mathfrak A_n=F_c\mathfrak A_n$. We prove this by induction on $n$. The base case of $n=0$ is trivial because for the associated graded to be a primitive algebra as in Definition \ref{def: finite filtration PB}, the unit must be in filtration degree $0$. 

Now suppose that both filtrations are equal in algebraic degree less than $n$ for some $n>0$. Recall from Section \ref{sec: spectral sequences} that the filtration induced on the bar-complex $B_n(\mathfrak A)$ and the filtration on the stalk $i_0^*\mathcal F_n$ are isomorphic, where $i_0\colon \{0\}\hookrightarrow \Sym^n(\C)$ is the inclusion of the origin. We write $F'_cB_n(\mathfrak A)\cong i_0^*F_c\mathcal F_n$ as the complex
$$\left\{\cdots \rightarrow \bigoplus_{\substack{p+q+r=n\\ 0<p,q,r<n}}F_c(\mathfrak A_p\otimes\mathfrak A_q \otimes \mathfrak A_r)\rightarrow \bigoplus_{\substack{p+q=n\\ 0<p,q<n}}F_c(\mathfrak A_p\otimes\mathfrak A_q )\rightarrow F_c'\mathfrak A_n\right\}$$
with $F'_c\mathfrak A_n$ in degree $-1$, where we note that by our induction hypothesis we have $F_c'(\mathfrak A_p\otimes \mathfrak A_q) = F_c(\mathfrak A_p\otimes \mathfrak A_q)$ in the direct sum because $p,q<n$, and likewise for $F_c'(\mathfrak A_p\otimes \mathfrak A_q\otimes \mathfrak A_r)$ and so on.

First consider the case where $c<n$. We claim that cohomology of this complex in degree $-1$ vanishes, i.e. we have $H^{-1}i_0^*F_c\mathcal F_n=0$. Let $i\colon \C\hookrightarrow \Sym^n(\C)$ be the inclusion of the diagonal, then by translation invariance it suffices to show that $\mathcal H^{-1}i^*F_c\mathcal F_n={}^p\mathcal H^0i^*F_c\mathcal F_n=0$. Here, we have a shift in degree by $1$ due to the fact that the perverse t-structure on the diagonal (which has dimension $1$) is the usual t-structure shifted by $1$. But this is true because $F_c\mathcal F_n$ is a middle extension from the open $V_c$ which does not contain the diagonal, and it is a property of middle extensions that the degree $0$ perverse cohomology of the stalk of such a middle extension is zero, see \cite[Lemma III.5.1(b)]{Kiehl_Weissauer_2010}.

Hence, we conclude that the map to $F_c'\mathfrak A_n$ must be surjective. Since the map is induced by the shuffle product by definition of the bar-complex, we conclude that $F_c'(\mathfrak A_n)$ is exactly the image under the shuffle product of $\bigoplus F_c(\mathfrak A_p\otimes \mathfrak A_q)$ with the direct sum being over $p+q=n$ and $p,q>0$. But these are exactly the linear combinations of the product of $p+q-c=n-c$ words, so $F'_c\mathfrak A_n=F_c\mathfrak A_n$ for $c<n$. 

For $c=n$, we have $F'_nB_n(\mathfrak A)\cong i_0^*F_n\mathcal F_n\cong i_0^*\mathcal F_n$, and by Proposition \ref{prop: equivalence properties} this is simply $B_n(\mathfrak A_n)$, so we conclude that $F_n'\mathfrak A_n=F_n\mathfrak A_n=\mathfrak A_n$.
\end{proof}
Let us try to generalize this correspondence as much as possible. Note that the codimension filtration comes directly from a filtration on the stratification of $\Sym(\C)$, where we assign the codimension $c$ stratum to be in the $c$ graded piece. We analyze all possible filtrations on the stratification of $\Sym(\C)$. Any such stratification is determined by a function $f$ that assigns partitions $\lambda = (\lambda_1\ge \cdots \ge \lambda_p)$ to integers $f(\lambda)$. From this, we can construct a geometric filtration on the perverse sheaf $j_*\mathcal L$ corresponding to $f$ analogously to the codimension filtration but now we let $V_c$ be the union of all strata $\lambda$ with $f(\lambda)\le c$. We call this the weighted codimension filtration associated to $f$. 

The function $f$ needs to satisfy some properties. Firstly, by factorizability we must have $f(\lambda \sqcup \lambda')=f(\lambda)+f(\lambda')$. In other words, $f$ is determined by its values on integers $f(n)$. Then, we need that $V_c$ is an open subscheme, which corresponds to having $f(n)+f(m)\le f(n+m)$. To normalize this to non-negative degrees, we can let $f(1)=0$ which forces $f(n)\ge 0$ for $n\ge 0$. The codimension filtration simply corresponds to the function $f(n)=n-1$.

Unsurprisingly, the corresponding algebraic filtration is the weighted (shifted) word length filtration corresponding to $f$, where the elements in $F_c\mathfrak A_n$ are linear combinations of shuffle product of words of length $\lambda_1,\ldots, \lambda_p$ where $\sum_{i=1}^p f(\lambda_i)\le c$. It is clear that the proof above generalizes to the equivalence of the weighted codimension filtration and weighted word length filtration.

The point of this mental exercise to show that if we just restrict ourselves to filtrations that come from the geometry of the stratification $\Sym^n(\A^1)$, we would not get very far. For example, we would not be able to obtain the weight filtration from such an argument.
\subsection{Examples} \label{sec: word filtration examples}
We describe the shifted word length filtration in the context of three quantum shuffle algebras. This will allow us to highlight how the word length filtration differs from the weight filtration later on in Section \ref{sec: weight examples}.

For the first example, consider the braided vector space $V=k$, where the braiding $R\colon V\otimes V\rightarrow V\otimes V$ is given by multiplication by $q\in k^\times$. Then, the quantum shuffle algebra $\mathfrak A(V)$ is isomorphic to the quantum divided power algebra $\Gamma_q[x]$ as discussed in \cite[Section 3.5]{ETW17} and \cite{Callegaro_2006}. The algebra $\Gamma_q[x]$ is defined to be additively generated by elements $x_n$ in degree $n$, with product given by
$$x_n\star x_m=\binom{n+m}{m}_q x_{n+m}$$
where the quantum binomial coefficient is defined as $\binom a b_q = \frac{[a]_q\cdots [a-b+1]_q}{[b]_q\cdots [1]_q}$ with $[r]_q=\frac{1-q^r}{1-q}=1+q+\cdots +q^{r-1}$. In the isomorphism above, $x_n$ corresponds to the word $1^{\otimes n}\in V^{\otimes n}=\mathfrak A_n$. 

An explicit ring structure of $\Gamma_q[x]$ is given in \cite[Proposition 3.12]{ETW17} and \cite[Lemma 3.4]{Callegaro_2006}, which we describe as follows. If $q$ is not a root of unity, we have $\Gamma_q[x]\cong k[x_1]$, and if $q$ is a primitive $m$-th root of unity, then $\Gamma_q[x]=k[x_1]/x_1^m\otimes \Gamma[x_m]$ where $\Gamma[x]=\Gamma_1[x]$ is the usual divided power algebra.

Hence, if $q$ is not a root of unity, then $\mathfrak A$ is generated in degree $1$, so the Nichols algebra $\mathfrak B$ is the whole algebra and the weight filtration is concentrated in degree zero with $F_0\mathfrak A=\mathfrak B\cong \mathfrak A$. If $q$ is a $m$-th root of unity, the element in degree $n=um+v$ for $0\le v<m$ can be expressed as the product of at most $u+v$ non-negative degree words, so it is in filtration degree $c=n-(u+v)=u(m-1)$. We plot this for $m=3$ in Figure \ref{fig1} below, and note that the filtration jumps by $m-1$ every $m$ algebraic degrees.
\begin{figure}[h]
\[
\begin{array}{c|ccccccccc}
 & \mathfrak A_0 & \mathfrak A_1 & \mathfrak A_2 & \mathfrak A_3 & \mathfrak A_4 & \mathfrak A_5 & \mathfrak A_6 & \mathfrak A_7 & \mathfrak A_8 \\
\hline
\Gr_0 & 1 & 1 & 1 &   &   &   &   &   &   \\
\Gr_1 &   &   &   &   &   &   &   &   &   \\
\Gr_2 &   &   &   & 1 & 1 & 1 &   &   &   \\
\Gr_3 &   &   &   &   &   &   &   &   &   \\
\Gr_4 &   &   &   &   &   &   & 1 & 1 & 1 
\end{array}
\]
\caption{Dimensions of shifted word length filtration for $V=k$ and $q$ third root of unity.}\label{fig1}
\end{figure}

For our second example, we consider the usual shuffle algebra where $V$ has basis $v_1,\ldots, v_d$ and $R\colon V\otimes V\rightarrow V\otimes V$ is given simply by swapping the basis $v_i\otimes v_j\mapsto v_j\otimes v_i$. It is clear that the Nichols subalgebra $\mathfrak B(V)$ is $\Sym(V)$, as the shuffle product $x_1\star\cdots \star x_n$ of degree $1$ elements is simply the sum of the $n!$ words which are permutations of $x_1,\ldots, x_n$. This tells us what the first filtered piece $F_0\mathfrak A$ is, but to get a clearer picture of the rest of the filtration, we need to turn to the dual tensor algebra. Recall that the dual filtration on $T=T(V^*)$ is the decreasing filtration defined by $F^{c+1}T_n = (F_c\mathfrak A_n)^\perp$
so that $\dim (\Gr^c T_n)=\dim(\Gr_c\mathfrak A_n)$. We attempt to understand the smallest nonzero filtered piece $F^{n-1}T_n=(F_{n-2}\mathfrak A_n)^\perp=(((\mathfrak A_{>0})^2)_n)^\perp$. In other words, $z\in F^{n-1}T_n$ if and only if for any $x,y\in \mathfrak A_{>0}$ we have $(x\star y,z)=0$. But $(x\star y,z)=(x\otimes y, \Delta_\star(z))$ by duality, so this happens if and only if $z$ is primitive. 

Let $P=P(T(V^*))\subseteq T$ be the subspace of primitive elements so $F^{n-1}T_n=P_n$. As we are dealing with usual non-braided Hopf algebras, $P$ has the structure of a Lie algebra with the commutator $[x,y]=xy-yx$. Furthermore, in our tensor algebra case, \cite[Definition 6.18]{Milnor_Moore_1965} tells us that $P(T(V^*))\cong L(V^*)$ is the free Lie algebra generated by $V^*$. By the Milnor-Moore theorem \cite{Milnor_Moore_1965}, the natural map $U(P)\rightarrow T$ from the universal enveloping algebra of $P$ is an isomorphism. Recall that there is a natural Poincaré–Birkhoff–Witt (PBW) filtration on the universal enveloping algebra $U(P)$ where $F_sU(P)$ consists of products of at most $s$ elements of $P$, this gives an increasing filtration $F_sT=F_sU(P)$ on the tensor algebra. One might suspect that this agrees with the shifted word length filtration, and this is indeed the case.
\begin{proposition}\label{prop: PBW word length agree}
The shifted word length filtration and the PBW filtration are related by $F^cT_n=F_{n-c}T_n$.
\end{proposition}
\begin{proof}
We first prove that $F_{n-c}T_n\subseteq F^cT_n$. Unwinding the definitions, it suffices to prove that the product (concatenation) of $s$ primitives $P_1\cdots P_s$ are orthogonal to all $x_1\star\cdots \star x_{s+1}$ for $x_i\in \mathfrak A_{>0}$. This is equivalent to showing $\Delta_\star^s(P_1\cdots P_s)\in I$, where we let $I=k\otimes T\otimes \cdots \otimes T+\cdots +T\otimes \cdots \otimes T\otimes k$. For the usual shuffle algebra, it is easy to see by induction that $\Delta_\star^{s}(P_1\cdots P_s)$ is a sum of terms corresponding to $s+1$ groupings of $P$. More precisely, the words are of the form $a_1\otimes \cdots \otimes a_{s+1}$ where $a_k=P_{i^k_1}\cdots P_{i^k_{b_k}}$ where $i^k_1<\cdots < i^k_{b_k}$ and the union of indices $i^k_j$ over all $j,k$ is $\{1,\ldots,s\}$. From this it is clear that $\Delta_\star^{s}(P_1\cdots P_s)\in I$ because some group must be empty.

Now, we show that $F_{n-c}T_n\supseteq F^cT_n$, which boils down to showing that an element $z\in T_n$ with $\Delta_\star^s(z)\in I$ is necessarily the sum of product of $\le s$ primitives. By PBW, we can write $z$ as the sum of terms $P_1\cdots P_t$ where $P_1\le \cdots \le P_t$ for some ordering of the basis of $P$. If $t>s$, $\Delta_\star^s(P_1\cdots P_t)$ contains a tensor product like $P_1\otimes \cdots \otimes P_s\otimes P_{s+1}\cdots P_t$ which does not lie in $I$. However, because we imposed the ordering $P_1\le\cdots \le P_t$, there are no other terms that can contribute this tensor product. Hence, $z$ is the sum of product of $\le s$ primitives.
\end{proof}

We plot the dimensions of the filtration for $\dim(V)=3$ in Figure \ref{fig2}. Note that $\dim(\Gr_0\mathfrak A_n)=\binom {n+2}2$ is the dimension of $\Sym^n(V)$ and the diagonal $\dim(\Gr_{n-1}\mathfrak A_n)=\dim(L(V^*)_n)$ is the number of Lyndon words of length $n$ over an alphabet of size $\dim(V)$.
\begin{figure}[h]
\centering
\[
\begin{array}{c|cccccc}
 & \mathfrak A_0 & \mathfrak A_1 & \mathfrak A_2 & \mathfrak A_3 & \mathfrak A_4 & \mathfrak A_5 \\
\hline
\Gr_0 & 1 & 3 & 6 & 10 & 15  & 21    \\
\Gr_1 &   &   & 3  & 9  & 18  & 30    \\
\Gr_2 &   &   &   &  8& 30& 66   \\
\Gr_3 &   &   &   &   & 18  &  78   \\
\Gr_4 &   &   &   &   &   &  48  
\end{array}
\]
\caption{Dimensions of shifted word length filtration for the usual shuffle algebra with $\dim(V)=3$.}\label{fig2}
\end{figure}

For our third and last example, we consider the braided vector space $V(c,x)$ corresponding to the rack $c$ of transpositions in $G=S_3$. If the cocycle $x=-1$, this corresponds to the case of Hurwitz space $\Hur^{G,c}_n$. We computed the dimensions of the filtration for $x=1$ and $x=-1$ up to $n=5$, and plotted them in Figure \ref{fig3}. 
\begin{figure}[h]
\centering
\[
\begin{array}{c|cccccc}
 & \mathfrak A_0 & \mathfrak A_1 & \mathfrak A_2 & \mathfrak A_3 & \mathfrak A_4 & \mathfrak A_5 \\
\hline
\Gr_0 & 1 & 3 & 9 & 27 & 79  & 225    \\
\Gr_1 &   &   &   &   &   &     \\
\Gr_2 &   &   &   &  & &   \\
\Gr_3 &   &   &   &   & 2  & 12   \\
\Gr_4 &   &   &   &   &   &  6 
\end{array}
\qquad
\begin{array}{c|cccccc}
 & \mathfrak A_0 & \mathfrak A_1 & \mathfrak A_2 & \mathfrak A_3 & \mathfrak A_4 & \mathfrak A_5 \\
\hline
\Gr_0 & 1 & 3 & 4 & 3 & 1  &     \\
\Gr_1 &   &   & 5  & 21 & 42  & 51    \\
\Gr_2 &   &   &   &  3& 32& 132   \\
\Gr_3 &   &   &   &   & 6  &  51   \\
\Gr_4 &   &   &   &   &   &  9  
\end{array}
\]
\caption{Dimensions of shifted word length filtration for the $S_3$ transposition rack with cocycle $x=1$ on the left and $x=-1$ on the right.}\label{fig3}
\end{figure}

For $x=1$, we see that there are no primitive elements in $\mathfrak A_2,\mathfrak A_3$, but there is a $2$ dimensional space of primitive elements for $\mathfrak A_4$, and $6$ dimensional space for $\mathfrak A_5$, which leads to the gap in gradings $1$ and $2$. For $x=-1$, we notice that the Nichols algebra $\mathfrak B=F_0\mathfrak A$ is finite-dimensional and ends at $n=4$. This is the 12 dimensional Fomin-Kirillov algebra $\mathcal{FK}_3$ first defined in \cite{Fomin_Kirillov_1999}.
\section{Equivalence of weight 
filtrations}\label{sec: equivalence of weight filtrations}
We present another instance of the equivalence established in Theorem \ref{thm: equivalence of filtrations}, namely the case of the weight filtration. For this section, let $V\in \mathcal V$ be a braided vector space such that $\mathcal L_n(V[1])$ has \textit{finite monodromy}, this is required to define the geometric weight filtration. As usual, let $\mathfrak A=\mathfrak A(V)$ be the quantum shuffle algebra, $T=T(V^*)$ be the dual tensor algebra, and $\mathfrak B=\mathfrak B(V)$ be the Nichols algebra. 

First, we discuss the geometric weight filtration on $\mathcal F_n=j_{n*}\mathcal L_n$ and its dual $\D\mathcal F_n=j_{n!}\mathcal L_n^\vee$. We define the weight filtration $W_w\mathcal F_n$ using the construction in Section \ref{sec: weight filtration over C}, and here we use that the local systems have finite monodromy. This weight filtration $W_w\D\mathcal F_n$ can be defined in the same way using $j_{n!}\mathcal L_n^\vee$ in place of $j_{n*}\mathcal L_n$, and it is easy to see that this is dual to the weight filtration in the sense that $\D \Gr_w\mathcal F_n\cong \Gr_{-w}\D\mathcal F_n$ and $\D W_w\mathcal F_n \cong \D\mathcal F_n/W_{-w-1}\D\mathcal F_n$. The weight filtration on $\mathcal F_n$ and $\D\mathcal F_n$ are concentrated in weights $\ge 0$ and $\le 0$ respectively. Just like in Section \ref{sec: equivalence of word length and codim filtration}, we can check that both are filtrations of factorizable perverse sheaves in the sense of Definition \ref{def: filtration FPS}. Indeed, the LHS of Equation \eqref{eqn: factorizability check} contains exactly the composition factors with weight $u+v$ because $\boxtimes$ adds weight, and the same is true for the RHS.

Recall that we defined the increasing algebraic weight filtrations $W_w \mathfrak A_n$ and $W_wT_n$ in Definition \ref{defn: alg weight filtration}, and that the weights in $\mathfrak A_n$ and $T_n$ are concentrated in degrees $\ge 0$ and $\le 0$ respectively. We were unable to show that this is a bialgebra filtration that works for all quantum shuffle algebras, and we suspect that there are some quantum shuffle algebras for which this is not a bialgebra filtration. Roughly speaking, the issue that comes up when proving comultiplication preserves the filtration is that twisted multiplication on $T\otimes T$ depends on the braiding so we need to check that the braiding $R\colon T\otimes T\rightarrow T\otimes T$ respects the weight filtration, which will only happen if the braiding behaves nicely with respect to primitive elements. Nevertheless, in the finite monodromy case, Proposition \ref{prop: equiv weight filtration} will imply that this is indeed a bialgebra filtration.
\subsection{Proof of equivalence} Recall that we assumed the local systems $\mathcal L(V[1])$ have finite monodromy, and under this assumption we prove the following correspondence.
\begin{proposition}\label{prop: equiv weight filtration}
The algebraic weight filtration $W_w\mathfrak A_n$ corresponds to the geometric weight filtration $W_w\mathcal F_n$ under the equivalence of Theorem \ref{thm: equivalence of filtrations}. Dually, $W_w T_n$ corresponds to $W_w\D\mathcal F_n$.
\end{proposition}
\begin{proof}
Let $W_w'T_n$ be the weight filtration that corresponds to $W_wj_{n!}\mathcal L_n^\vee$, and we will prove that $W'_wT_n=W_wT_n$ by induction on $n$. Suppose that both filtrations are equal in algebraic degree less than $n$ for some $n>0$.

Let $V=\Sym^n(\C)\setminus \C$ be the complement of the diagonal, with $u\colon \Sym^n_{\neq}(\C)\hookrightarrow V$, $v\colon V\hookrightarrow \Sym^n(\C)$ and $i\colon \C\hookrightarrow \Sym^n(\C)$. Recall from the codimension filtration $F_c\mathcal F_n$ that we have $F_{n-1}\mathcal F=v_{!*}u_*\mathcal L_n=v_{!*}v^*\mathcal F_n$, and by the property of the middle extension the quotient $\Gr^F_n\mathcal F_n=\mathcal F_n/F_{n-1}\mathcal F_n$ is supported on the diagonal. Hence, we have an exact sequence
$$0\rightarrow v_{!*}v^*\mathcal F_n \rightarrow \mathcal F_n\rightarrow i_*C_n\rightarrow 0$$
where by translation invariance, $C_n$ is a constant sheaf on the diagonal. Furthermore, just like in the proof of Proposition \ref{prop: word length codim equivalence}, from Section \ref{sec: spectral sequences} we see that $i_0^*\Gr^F_n\mathcal F_n\cong \Gr^F_n B_n(\mathfrak A_n)$ where $i_0\colon \{0\}\hookrightarrow \Sym^n(\C)$. From this, we see that $C_n$ is the constant sheaf in degree $-1$ with values in $\mathfrak A_n/\mu(\bigoplus\mathfrak A_p\otimes \mathfrak A_q)$ where the direct sum is over $p+q=n$ satisfying $0<p,q<n$. Taking the dual, we have
$$0\rightarrow i_*C_n^\vee\rightarrow \D\mathcal F_n\rightarrow v_{!*}v^*\D\mathcal F_n\rightarrow 0$$
where we note that $C_n^\vee$ is now the sheaf in degree $-1$ valued in primitive elements $P_n\subseteq T_n$ which is the dual of $\mathfrak A_n/\mu(\bigoplus\mathfrak A_p\otimes \mathfrak A_q)$.

We now consider the weight filtrations on $i_*C_n^\vee$ and $v_{!*}v^*\D\mathcal F_n$ which we write as
$$0=W_{-a}v_{!*}v^*\D\mathcal F_n\subseteq \cdots \subseteq W_0v_{!*}v^*\D\mathcal F_n=v_{!*}v^*\D\mathcal F_n$$
$$0=W_{-b}i_*C_n^\vee\subseteq \cdots \subseteq W_0i_*C_n^\vee=i_*C_n^\vee$$
where for notational purposes we select some sufficiently large indices $a,b$. Then, we concatenate them naively via the short exact sequence above to a filtration
$$G_{-a-b}\D\mathcal F_n \subseteq \cdots \subseteq G_0\D\mathcal F_n$$
where $\Gr^{G}_{w-a}\D\mathcal F_n=\Gr^W_{w}i_*C_n^\vee$ for $-b< w\le 0$ and $\Gr^{G}_{w}\D\mathcal F_n=\Gr^W_{w}v_{!*}v^*\D\mathcal F_n$  for $-a< w\le 0$. The composition factors can be rearranged to obtain the weight filtration $W_w\D\mathcal F_n$. In fact, by semisimplicity of (polarizable) pure Hodge modules \cite[Corollary 5.2.13] {Saito_1988}, we have
$$\Gr^W_w\D\mathcal F_n=\Gr^W_{w}i_*C_n^\vee\oplus \Gr^W_w v_{!*}v^*\D\mathcal F_n.$$
In principle, we understand the weight filtration on $v^*\D\mathcal F_n$ (and hence $v_{!*}v^*\D\mathcal F_n$) in terms of lower algebraic degrees by factorizability, as we can cover $\Sym^n(\C)\setminus \C$ with opens of the form $\Sym^p(U)\times \Sym^q(V)$. On the other hand, we do not know the weight filtration on $i_*C_n^\vee$, and we will need to deduce it from the following spectral sequence argument.

Consider the spectral sequence
$$E_1^{p,q}=H^{p+q}(i_0^*\Gr^{G}_{p}\D\mathcal F_n) \Rightarrow H^{p+q}(i_0^*\D\mathcal F_n)$$
where we recall that by quasi-homogeneity (contracting $\C^*$ action) we have $R\Gamma K \cong i_0^* K$ for any $K\in D^b_c(\Sym^n(\C),S_n,\mathcal V)$. This spectral sequence has the following two useful properties. Firstly, we have $i_0^*\D\mathcal F_n \cong i_0^*j_{n!}\mathcal L_n^\vee = 0$ by the property of extensions by zero, so the abutment is zero. Secondly, only differentials between terms of the same weight are nontrivial. To illustrate this, note that each $\Gr^{G}_{p}\D\mathcal F_n$ has pure weight, so $i_0^*\Gr^{G}_{p}\D\mathcal F_n\cong R\Gamma\Gr^{G}_{p}\D\mathcal F_n$ must be pure of the same weight because pullbacks can only decrease weight and pushforwards can only increase weight. Then, $H^{p+q}i_0^*\Gr^{G}_{p}\D\mathcal F_n$ is a pure Hodge structure, and maps between pure Hodge structures are only nonzero if they have the same weight.

To write down the spectral sequence, note that by Section \ref{sec: spectral sequences} we know that $i_0^*\Gr_w^W\D\mathcal F_n \cong \Gr_w^W B_n(T)$ is the graded part of the bar-complex on $T$
\begin{equation}\label{eqn: bar-complex tensor algebra}
\left\{\cdots \rightarrow \bigoplus_{\substack{p+q+r=n\\ 0<p,q,r<n}}\Gr^W_w(T_p\otimes T_q\otimes T_r)\rightarrow \bigoplus_{\substack{p+q=n\\ 0<p,q<n}}\Gr^W_w(T_p\otimes T_q )\rightarrow \Gr^{W'}_wT_n\right\}
\end{equation}
where $\Gr^{W'}_wT_n$ is in degree $-1$ and we used the induction hypothesis that $W'_wT_m=W_wT_m$ for $m<n$. This is quasi-isomorphic to the direct sum $i_0^*\Gr^W_{w}i_*C_n^\vee\oplus i_0^*\Gr^W_w v_{!*}v^*\D\mathcal F_n$, but a priori it is not clear that these can be represented as subcomplexes of the complex above. Instead, to see this, we recall that $i_*C_n= \Gr_n^F\mathcal F_n$ is a graded piece of the codimension filtration, so dually $i_*C_n^\vee = \Gr^n_F \D\mathcal F_n$, and by Section \ref{sec: spectral sequences} together with the discussion above we see that $i_0^*i_*C_n^\vee$ is isomorphic to the subcomplex of $B_n(T)$ which is $P_n$ in degree $-1$. By Proposition \ref{prop: bi-filtrations}, the isomorphisms to the bar-complex for the codimension filtration and weight filtration are compatible. Hence, we can indeed identify $i_0^*\Gr^W_wi_*C_n^\vee$ with the subcomplex $P_n\cap \Gr^{W'}_wT_n$ in degree $-1$ and $i_0^*\Gr^W_w v_{!*}v^*\D\mathcal F_n$ with the subcomplex given above by modifying the last term of Equation \eqref{eqn: bar-complex tensor algebra} to $\Gr^{W'}_wT_n/P_n$. Furthermore, note that $\Gr^W_w v_{!*}v^*\D\mathcal F_n$ has composition factors which are middle extensions from non-diagonal strata, so ${}^p\mathcal H^0(i^*\Gr^W_w v_{!*}v^*\D\mathcal F_n)=0$ which implies that $H^{-1}(i_0^*\Gr^W_w v_{!*}v^*\D\mathcal F_n)=0$, thus the map to $\Gr^{W'}_wT_n/P_n$ in the subcomplex of $i_0^*\Gr^W_w v_{!*}v^*\D\mathcal F_n$ must be surjective.

With this, we can draw the $E_0$ page in Figure \ref{fig4}, which is just the graded pieces of the bar-complex corresponding to the filtration $G$ which we obtain by reshuffling the subcomplexes above. We also take the index $b=a+1$ for notational convenience.
\begin{figure}[h]
\[\begin{tikzcd}[column sep=small,row sep=tiny]
	{\text{degree}} && {-2} & {-1} & \begin{array}{c} \text{complex}\\\text{weight} \end{array} \\
	{i_0^*\Gr^{W}_0v_{!*}v^*\D\mathcal F_n=} & \cdots & \begin{array}{c} \bigoplus_{\substack{p+q=n\\ 0<p,q<n}}\Gr^W_0(T_p\otimes T_q ) \end{array} & {\Gr^{W'}_0T_n/P_n} & 0 \\
	\vdots & \vdots & \vdots & \vdots & \vdots \\
	{i_0^*\Gr^{W}_{-a+1}v_{!*}v^*\D\mathcal F_n=} & \cdots & \begin{array}{c} \bigoplus_{\substack{p+q=n\\ 0<p,q<n}}\Gr^W_{-a+1}(T_p\otimes T_q ) \end{array} & {\Gr^{W'}_{-a+1}T_n/P_n} & {-a+1} \\
	{i_0^*\Gr^{W}_0i_*C_n^\vee=} & \cdots & 0 & {P_n\cap \Gr^{W'}_0T_n} & 0 \\
	{i_0^*\Gr^{W}_{-1}i_*C_n^\vee=} & \cdots & 0 & {P_n\cap \Gr^{W'}_{-1}T_n} & {-1} \\
	\vdots & \vdots & \vdots & \vdots & \vdots \\
	{i_0^*\Gr^{W}_{-a}i_*C_n^\vee=} & \cdots & 0 & {P_n\cap \Gr^{W'}_{-a}T_n} & {-a}
	\arrow[from=2-2, to=2-3]
	\arrow[two heads, from=2-3, to=2-4]
	\arrow[dashed, from=2-3, to=6-4]
	\arrow[dashed, from=3-3, to=7-4]
	\arrow[two heads, from=4-3, to=4-4]
	\arrow[dashed, from=4-3, to=8-4]
	\arrow[from=5-2, to=5-3]
	\arrow[from=5-3, to=5-4]
	\arrow[from=6-2, to=6-3]
	\arrow[from=6-3, to=6-4]
	\arrow[from=8-2, to=8-3]
	\arrow[from=8-3, to=8-4]
\end{tikzcd}\]
\caption{$E_0$ page of spectral sequence.}\label{fig4}
\end{figure}

The differentials in the spectral sequence always go from cohomology degree $d$ to $d+1$. On the $E_m$ page, they go one to the right and $m$ rows downwards. All differentials are induced by concatenation, because the spectral sequence comes from the filtration of the bar-complex whose maps are given by concatenation. On the $E_0$ page, we take differentials along the solid arrows to get the cohomology of the complexes. From the $E_1$ page onwards, the nonzero differentials from degree $d$ to $d+1$ must go from $w$ to $w-1$ in complex (derived) weight, so that they are between pure Hodge structures of the same weight $d+w=(d+1)+(w-1)$. Here, we recall from Section \ref{sec: weight filtration over C} that a complex $K\in D^b(\MHM(X))$ has complex (derived) weight $w$ if and only if each $H^j(K)$ has weight $w+j$, and that there are no maps between pure Hodge structures of different weight.

With this, we can draw all possible nonzero differentials from degree $-2$ to degree $-1$ from the $E_1$ page onwards using dotted arrows in the figure. There may be differentials from degree $d$ to degree $d+1$ where $d<-2$, but these do not affect the degree $-1$ entries, and they only replace the degree $-2$ entries with a quotient, so they do not really concern us and we do not consider them. We see that these dotted arrows are on the $E_{a+1}$ page, going one to the right and $a+1$ rows down. There could potentially be nontrivial differentials also on the $E_1$ page going from degree $-2$ to $-1$ as they go from complexes of weight $w$ to $w-1$. However, these differentials all turn out to be trivial, either because the degree $-1$ term is killed by a surjection on the $E_0$ page (first $a$ rows in the figure), or the degree $-2$ term is already zero (last $a+1$ rows in the figure).

As discussed earlier, we know that the spectral sequence abuts to zero, so the differentials indicated by dotted arrows in the diagram must kill the terms of the form $P_n\cap \Gr^{W'}_{-a}T_n$. Let us first look at the lowest term $P_n\cap W'_{-a}T_n$. The differential comes from a quotient of $\ker(\bigoplus \Gr^W_{-a+1}(T_p\otimes T_q )\rightarrow \Gr^{W'}_{-a+1}T_n/P_n)$ and is given by concatenation, so we must have $P_n\cap W'_{-a}T_n=P_n\cap \mu(\bigoplus \Gr^W_w(T_p\otimes T_q ))$. Working our way up the diagram, we obtain 
$$P_n\cap W'_{w}T_n=P_n\cap \mu\left(\bigoplus \Gr^W_{w+1}(T_p\otimes T_q )\right)$$ 
for all $w$ in the same way. Finally, given that the map to $\Gr^{W'}_{w}T_n/P_n$ is a surjection, we see that $$W'_{w}T_n/P_n= \mu\left(\bigoplus \Gr^W_{w}(T_p\otimes T_q)\right)$$
for all $w$. Combining these two equations tells us that $W'_wT_n=W_wT_n$ agrees with the algebraic weight filtration, completing the proof.
\end{proof}
\subsection{Examples}\label{sec: weight examples}
We use the same three examples in Section \ref{sec: word filtration examples} to compare the algebraic weight filtration with the shifted word length filtration. We will see that they are indeed different filtrations, although they agree in some special cases.

For the first example with $V=k$ and braiding given by $q\in k^\times$, note that we have finite monodromy if and only if $q$ is a root of unity. In this case, recall that we have an explicit expression $\mathfrak A\cong k[x_1]/x_1^m\otimes \Gamma[x_m]$. From this, it is easy to see that the only primitive element in $T$ with algebraic grading $n>1$ is $1^{\otimes m}$ corresponding to $x_m^*$. Hence, $1^{\otimes m}$ goes down one grading to $W_{-1}T$, and all terms after that are determined by concatenation that respects the weight grading. In particular, for $n=um+v$ for $0\le v<m$, we can write $1^{\otimes m}=(1^{\otimes m})^{\otimes u}\otimes 1^{\otimes v}$, so $1^{\otimes m}$ is in weight $-u$. When we view this in the quantum shuffle algebra $\mathfrak A$, the sign of the weights flip. We plot the dimensions for $m=3$ in Figure \ref{fig5}, and note that compared to Figure \ref{fig1}, the weights here only go up by $1$ each time. 
\begin{figure}[h]
\centering
\[
\begin{array}{c|ccccccccc}
 & \mathfrak A_0 & \mathfrak A_1 & \mathfrak A_2 & \mathfrak A_3 & \mathfrak A_4 & \mathfrak A_5 & \mathfrak A_6 & \mathfrak A_7 & \mathfrak A_8 \\
\hline
\Gr_0 & 1 & 1 & 1 &   &   &   &   &   &   \\
\Gr_1 &   &   &   & 1 & 1 & 1 &   &   &   \\
\Gr_2 &   &   &   &   &   &   & 1 & 1 & 1 
\end{array}
\]
\caption{Dimensions of weight filtration for $V=k$ and $q$ third root of unity.}\label{fig5}
\end{figure}

Next, we look at the usual shuffle algebra. Recall that from the discussion in Section \ref{sec: word filtration examples}, the space of primitive elements $P\subseteq T$ is the free Lie algebra $L(V^*)$ generated by $V^*$. We prove that the weight filtration here agrees with both the shifted word length filtration and the PBW filtration in Proposition \ref{prop: PBW word length agree} by $W_wT_n=F^{-w}T_n=F_{n+w}T_n$ for $n>0$. Indeed, by PBW, each element in $T$ can be written as a unique combination of terms $P_1\cdots P_t$ for $P_1\le \cdots \le P_t$, and from this the claim follows by induction and the definition of the weight filtration. Thus, the dimensions of the filtration are as given in Figure \ref{fig2}.

Lastly, we compute the dimensions for the braided vector space $V(c,x)$ where $c$ is the conjugacy classes in $G=S_3$. We find that when $x=-1$ the weight filtration agrees with the shifted word length filtration up to $n=5$ on the right of Figure \ref{fig3}, although we are unable to prove that it agrees in general. For $x=1$, the dimensions are different and are given in Figure \ref{fig6}. When $n=4$, the $2$ dimensional space of primitive elements in $T_4$ are in weight $-1$, so there is a corresponding $2$ dimensional space in weight $1$ for $\mathfrak A_4$. For $n=5$, the $6$ dimensional space of primitive elements in $T_5$ combine with the $12$ dimensional space generated from the primitive elements in $T_4$, yielding a $18$ dimensional space in weight $1$ for $\mathfrak A_5$.
\begin{figure}[h]
\centering
\[
\begin{array}{c|cccccc}
 & \mathfrak A_0 & \mathfrak A_1 & \mathfrak A_2 & \mathfrak A_3 & \mathfrak A_4 & \mathfrak A_5 \\
\hline
\Gr_0 & 1 & 3 & 9 & 27 & 79 & 225   \\
\Gr_1 &   &   &   &  & 2 & 18 
\end{array}
\]
\caption{Dimensions of weight filtration for $S_3$ transposition rack with cocycle $x=1$.}\label{fig6}
\end{figure}

The structure of the weight filtration depends solely on the primitive elements in the tensor algebra. However, in general, it is difficult to determine the primitive elements for braided algebra, for example, unlike the case for the usual shuffle algebra, the primitive elements no longer have a simple Lie algebra structure. \cite{Westerland_2025} investigates this in more detail, developing versions of the PBW and Cartier-Milnor-Moore theory in the setting of braided Hopf algebras. In particular, they define the braided primitive operads $\text{BrPrim}(n)$ which construct primitive elements from $n$ different elements. $\text{BrPrim}(2)$ is easy to understand, essentially the theory tells us that from elements $x,y$ we can construct the primitive element $$\frac{1}{2n}\left(\mu(x\otimes y)-\mu(R(x\otimes y))+\cdots -\mu(R^{2n-1}(x\otimes y))\right)$$
if $R^{2n}(x\otimes y)=x\otimes y$. However, there can be operations in $\text{BrPrim}(n)$ which are not generated by this. Indeed, we see this for $V(c,1)$ above where the first primitive element is in degree $4$, generated by an operation in $\text{BrPrim}(4)$, and it is not possible to get these degree $4$ primitive elements by repeatedly applying the operation in $\text{BrPrim}(2)$ starting from degree $1$ elements. However, it seems possible that a large class of shuffle algebras coming from Hurwitz spaces have the property that all primitive elements can be generated from repeatedly applying $\text{BrPrim}(2)$ starting from degree $1$ elements, and perhaps there is a geometric reason for this. If this were true, we might be able to say something more about the structure of the weight filtration.
\section{Comparison theorems for weight filtration in Hurwitz spaces}
\label{sec: construction weight filtration} 

In the previous section, we showed that the geometric weight filtration on $j_*\mathcal L(V[1])$ constructed via mixed Hodge modules over $\C$ corresponds to the algebraic weight filtration on the quantum shuffle algebra $\mathfrak A(V)$. Recall that we want to apply this to point counting over finite fields, so we need a comparison theorem between the weight filtrations of $j_*\mathcal L(V[1])$ over $\F_p$ and $\C$. It is difficult to do this in general, so in this section, we specialize to the case when $V$ comes from a Hurwitz space. This will allow us to prove Theorem \ref{thm: relate weight decompositions}, which allows us to compute the cohomological weights for Hurwitz spaces using quantum shuffle algebras.

Fix an integer $n\ge 1$, a group $G$, a union of conjugacy classes $c\subseteq G$, along with primes $p>|G|$ and $\ell>n$ with $p\neq \ell$. Fix an isomorphism $\iota\colon \overline\Q_\ell\xrightarrow\sim k$ and we will implicitly identify both coefficient fields. From now on, we will work over $S=\Spec(\Z_p)$, and we denote $s,\eta$ to be the special point and generic point respectively.

Note that we can view the Hurwtiz space $\Hur^{G,c}_n$ as a scheme over $\Z_p$ by Section \ref{sec: alg Hurwitz spaces}, as well as the spaces $\Sym^n_{\neq}(\A^1)$ and $\Sym^n(\A)$. Let $\pi\colon \Hur^{G,c}_n \rightarrow \Sym^n_{\neq}(\A^1)$ and recall that $j_n\colon \Sym^n_{\neq}(\A^1)\hookrightarrow \Sym^n(\A^1)$. Recall the discussion at the end of Section \ref{sec: KS equivalence} that the relevant braided vector space for Hurwitz spaces is $V_{\epsilon}=V(c,-1)\in \mathcal{YD}_G$, and that $\mathcal L_n=\mathcal L_n(V_{\epsilon}[1])$ on $\Sym^n_{\neq}(\C)$ is the local system in degree $-n$ given by the $B_n$-representation $c^n$. Also recall from Section \ref{sec: weight filtration over C} that we want $\mathcal L_n$ to be normalized such that it has weight $0$ as a complex, so over $\C$ we have $\mathcal L_n=\pi_* \underline{k}_{\Hur^{G,c}_n}(n/2)[n]$. This allows us to define $\mathcal L_n$ over $\Z_p$ by the same formula.

We defined the weight filtration on $j_{n*}\mathcal L_n$ on both $\F_p$ and $\C$ separately as in Section \ref{sec: weight filtration}, where we note that the finite monodromy assumption for $\mathcal L_n$ over $\C$ is valid because it comes from a finite étale cover. Hence, we have two spectral sequences, one for each setting:
\begin{equation}\label{eqn: both spectral sequence}
\begin{split}
E_1^{p,q}=R\Gamma^{p+q}(\Sym^n(\C),\Gr^W_{p}(j_{n*}\mathcal L_n)_\C)&\Rightarrow R\Gamma^{p+q}(\Sym^n(\C),(j_{n*}\mathcal L_n)_\C),\\
E_1^{p,q}=R\Gamma^{p+q}_{\text{ét}}(\Sym^n(\A^1_{\overline\F_p}),\Gr^W_{p}(j_{n*}\mathcal L_n)_{\overline \F_p})&\Rightarrow R\Gamma^{p+q}_{\text{ét}}(\Sym^n(\A^1_{\overline\F_p}),(j_{n*}\mathcal L_n)_{\overline \F_p}).
\end{split}
\end{equation}
In particular, both spectral sequences induce the respective weight filtrations on the abutment. 

We understand the first spectral sequence over $\C$ very explicitly in the language of quantum shuffle algebras. Indeed, by Proposition \ref{prop: equiv weight filtration} it is isomorphic to the spectral sequence on Tor-homology in Theorem \ref{thm: equivalence of filtrations}. For the second spectral sequence, because
$R\Gamma j_{n*}\mathcal L_n=R\Gamma j_{n*}\pi_* \underline{k}_{\Hur^{G,c}_n}(n/2)[n] = R\Gamma \underline{k}_{\Hur^{G,c}_n}(n/2)[n]$, we have
$$R\Gamma^{p+q}_{\text{ét}}(\Sym^n(\A^1_{\overline\F_p}),(j_{n*}\mathcal L_n)_{\overline \F_p})\cong H^{p+q+n}_{\text{ét}}(\Hur^{G,c}_{n,\overline \F_p},\overline\Q_\ell)(n/2).$$
Hence, we can transport the weight filtration to the cohomology of Hurwitz space via the isomorphism above. The $w$ graded part of $R\Gamma^{p+q}_{\text{ét}}(\Sym^n(\A^1_{\overline\F_p}),(j_{n*}\mathcal L_n)_{\overline \F_p})$ would have eigenvalues of Frobenius acting by $q^{(w+p+q)/2}$ (because weights are defined in the derived sense, and pure is pointwise pure in the case of a point, see Section \ref{sec: weights over F_p}). Thus, the eigenvalues on the $w$ graded part of $H^{i}_{\text{ét}}(\Hur^{G,c}_{n,\overline \F_p},\overline\Q_\ell)$ will have eigenvalues $q^{(w+i)/2}$ after accounting for the Tate twist and degree shift, so the corresponding eigenvalues on $H^{2n-i}_{c}(\Hur^{G,c}_{n,\overline \F_p},\overline\Q_\ell)\cong H^{i}_{\text{ét}}(\Hur^{G,c}_{n,\overline \F_p},\overline\Q_\ell)^\vee (n)$ have eigenvalues have absolute value at most $q^{(2n-i-w)/2}$. 

Hence, from the discussion above, to complete the proof of Theorem \ref{thm: relate weight decompositions} it remains to bridge the gap between the weight spectral sequences over $\F_p$ and $\C$. This is done in the following theorem, which generalizes the comparison theorem in \cite[Proposition 7.7, 7.8]{EVW16}. 
\begin{theorem}\label{thm: comparison theorem}
There is a natural isomorphism between the two spectral sequences in Equation \eqref{eqn: both spectral sequence}.
\end{theorem}

The rest of this section is dedicated to the proof of Theorem \ref{thm: comparison theorem}. We start by recalling the normal crossings compactification of Hurwitz space constructed in \cite[Appendix B]{ellenberg2025homologicalstabilitygeneralizedhurwitz}, which uses the notion of twisted stable maps developed in \cite{Abramovich_Vistoli_2001,Abramovich_Corti_Vistoli_2003} which was in turn motivated by admissible covers \cite{Harris_Mumford_1982}. In this normal crossings setting it is easy to construct the weight filtration of the pushforward of the constant sheaf, and from this we can deduce the weight filtration for $j_{n*}\mathcal L_n$ over $\Z_p$ via taking pushforwards and pullbacks. Lastly, we use a vanishing cycles argument with respect to $\Z_p$ to prove the comparison theorem of weight filtrations over the special fiber $\F_p$ and the generic fiber $\C$. 
\subsection{Compactification of Hurwitz space} \label{sec: compactification of hurwtiz space} We state the normal crossings compactification of $\Hur^{G,c}_n$ together with some of its properties.
\begin{proposition}
There is a normal crossings compactification $u\colon \Hur^{G,c}_n\hookrightarrow \overline{\Hur^{G,c}_n}$ over $\Z_p$ where $\overline{\Hur^{G,c}_n}$ is a smooth proper Deligne-Mumford stack, and there is a cartesian diagram
\[\begin{tikzcd}
	{\Hur^{G,c}_n} & {\overline{\Hur^{G,c}_n}} \\
	{\Sym^n_{\neq}(\A^1)} & {\Sym^n(\mathbb P^1).}
	\arrow[hook, from=1-1, to=1-2]
	\arrow[from=1-1, to=2-1]
	\arrow[from=1-2, to=2-2]
	\arrow[hook, from=2-1, to=2-2]
\end{tikzcd}\]
\end{proposition}
\begin{proof} The first statement follows from
\cite[Corollary B.1.4]{ellenberg2025homologicalstabilitygeneralizedhurwitz} with $B=\Spec \Z[\frac{1}{|G|}]$, $C=\mathbb P^1$ and $Z=\infty$. We construct the cartesian diagram by sending a $G$-cover of $\Hur^{G,c}_n$ to its $n$ distinct branch points, and likewise by sending a balanced twisted stable map in $\overline{\Hur^{G,c}_n}$ to the $n$ marked points on the base curve $C$. It is clear from the construction that if these marked points are distinct and are not at $\infty$, then the balanced twisted stable map has to be an actual $G$-cover.
\end{proof}
We will first construct a weight filtration on $u_*\underline{k}_{\Hur^{G,c}_n}(n/2)[n]$, and later we will deduce the weight filtration of $j_{n*}\mathcal L_n$ from this. The compactification of $\Hur^{G,c}_n$ constructed above looks étale-locally like a standard normal crossing with $\A^n$ and boundary a union of coordinate hyperplanes. In the next subsection, we construct the weight filtration étale-locally in the standard normal crossing case, then show that the construction glues together.
\subsection{Weight filtration for standard normal crossings} \label{sec: standard normal crossings}
We consider the following standard normal crossings situation over $\Z_p$. Let $X=\A^n$ with coordinates $x_1,\ldots, x_n$ and the boundary divisor $D=D_1\cup\cdots\cup D_d$ where $D_i=\{x_i=0\}$ for $d\le n$. Let $U=X\setminus D$ with inclusion $u\colon U\hookrightarrow X$. Furthermore, for any subset $I\subseteq J$, denote $D_I=\bigcap_{i\in I}D_i$ and the inclusions $i_I\colon D_I\hookrightarrow X$ and $w_I\colon X\setminus D_I\hookrightarrow X$. 
For notational simplicity, write $i_j=i_{\{j\}}$ and $w_j=w_{\{j\}}$.

We want to define the weight filtration of $\mathcal F=u_*\underline{k}_U(n/2)[n]$ on the compactification $X$ over $\Z_p$. In this normal crossings case, the weight filtration will just be the usual codimension filtration on $X$ like in Section \ref{sec: equivalence of word length and codim filtration} which is not very surprising. Let $V_c=X\setminus \bigcup_{|I|=c+1}D_I$ be the union of codimension $\le c$ strata, and let $u_c\colon U\hookrightarrow V_c$, $v_c \colon V_c\hookrightarrow X$. We define the weight filtration to be 
\begin{equation}\label{eqn: weight filtration over Zp}
W_c\mathcal F = v_{c!*}u_{c*}\underline{k}_U(n/2)[n]
\end{equation}
so $W_{-1}\mathcal F=0$, $W_0\mathcal F=u_{!*}\underline{k}_U(n/2)[n]$ and $W_n\mathcal F=\mathcal F$.

To make sense of intermediate extensions over $\Z_p$, we need the theory of relative perversity developed by \cite{Hansen_Scholze_2023} which is equivalent to perversity on each geometric fiber. Just like the usual case, the relative perverse sheaves form an abelian category and we can construct intermediate extensions in the same way. 

We note several base change properties which is useful for our situation. By proper base change, extensions by zero $f_!$ commute with taking fibers. Furthermore, by \cite[Expose XIII, Lemma 2.1.10]{SGA7II}, if $f$ is an open immersion with normal crossing boundary, $f_*$ commutes with taking fibers (this also follows from the lemma below). This tells us that each term in the weight filtration above are relatively perverse by base changing to the geometric fibers $\overline s, \overline \eta$ and noting that $u_{c*}$ is an affine open immersion so it is t-exact on each fiber.

The following lemma tells us the graded pieces in the weight filtration. 
\begin{lemma}\label{lem: graded pieces in compactification}
We have
$$\Gr^W_c \mathcal F= \bigoplus_{|I|=c}i_{I*}\underline{k}_{D_I}(n/2-c)[n-c].$$
\end{lemma}
It is clear by base change that these are relatively perverse. Here, we keep track of the Tate twist as it is important for weight considerations. The idea of the proof is to extend across divisors inductively to come up with a composition series for $u_*\underline{k}_U[n]$ and then rearrange them into the weight filtration.
\begin{proof}
We induct on $d$, with $d=0$ being the trivial base case. Now we assume the statement for $d-1$ and attempt to prove it for $d$. We first ignore the hyperplane $D_d$ and look only at the hyperplanes $D_1,\ldots, D_{d-1}$. Let $K=\{1,\ldots, d-1\}$, $D'=\bigcup_{i\in K}D_i$, $U'=X\setminus D'$ and $u'\colon U'\hookrightarrow X$. Consider the perverse sheaf $\mathcal F'=u_{*}'\underline{k}_{U'}(n/2)[n]$, which has a filtration $W'_c\mathcal F'$ defined in the same way as Equation \eqref{eqn: weight filtration over Zp} but with only $d-1$ hyperplanes. By the inductive hypothesis, the graded pieces are given by 
$$\Gr^{W'}_c \mathcal F'= \bigoplus_{|I|=c, I\subseteq K} i_{I*}\underline{k}_{D_I}(n/2-c)[n-c].$$

We consider the distinguished triangle $$i_{d*}i^!_d\mathcal F'\rightarrow \mathcal F'\rightarrow w_{d*}w_d^*\mathcal F'\xrightarrow{+1},$$
and we note that $w_{d*}w_d^*\mathcal F'=\mathcal F$ is what we want, so after shifting we obtain
$$\mathcal F'\rightarrow \mathcal F\rightarrow i_{d*}i^!_d\mathcal F'[1]\xrightarrow{+1}.$$ 

We see that the filtration $W'_c\mathcal F'$ induces a filtration $i_{d*}i^!_dW'_c\mathcal F'[1]$ on $i_{d*}i^!_d\mathcal F'[1]$ by relative perverse sheaves. To see this, recall that the inclusion $i\colon D_{I\cup \{d\}}\hookrightarrow D_I$ is a embedding of smooth varieties, so by purity we have $i^!\underline{k}_{D_I}=\underline{k}_{D_{I\cup\{d\}}}(-1)[-2]$. Thus, applying $i_{d*}i^!_d$ and the shift $[1]$ to $i_{I*}\underline{k}_{D_I}(n/2-c)[n-c]$ which is a factor of $\Gr^{W'}_c\mathcal F'$, we get $i_{I\cup\{d\}*}\underline{k}_{D_{I\cup\{d\}}}(n/2-c-1)[n-c-1]$ which is relatively perverse. Since extensions of perverse sheaves in a derived category are necessarily perverse, $i_{d*}i^!_d\mathcal F'[1]$ is also relatively perverse, as the graded pieces of its filtration are relatively perverse.

Combining the filtrations on both $\mathcal F'$ and $i_{d*}i^!_d\mathcal F'[1]$, we get a filtration on $\mathcal F$ by relative perverse sheaves. The factors in the combined composition series have two forms, firstly $i_{I*}\underline{k}_{D_I}(n/2-c)[n-c]$ from $\mathcal F'$, and secondly $i_{I\cup\{d\}*}\underline{k}_{D_{I\cup\{d\}}}(n/2-c-1)[n-c-1]$ from $i_{d*}i^!_d\mathcal F'[1]$, where for both cases we take $I\subseteq K$. However, note that we could have chosen any different $d-1$ subset $K\subseteq \{1,\ldots, d\}$, so none of the composition factors of the same size appear strictly in front of the other, in other words they split and form a direct sum. 

Thus, we can form a filtration where $\Gr_c \mathcal F=\bigoplus_{|I|=c}i_{I*}\underline{k}_{D_I}(n/2-c)[n-c]$ as desired. To check that this agrees with the intermediate extension description, note that $v_{c!*}u_{c*}\underline{k}_U(n/2)[n]$ is the maximal subobject of $\mathcal F$ which does not contain any simple perverse sheaves supported on $X\setminus V_c$, and this is indeed true given the composition factors described above.
\end{proof}

To justify calling it the weight filtration, we will need to check that the fibers agree with our previous definitions. Recall from Section \ref{sec: weight filtration over C} that the weight filtration over $\C$ is determined by lifting perverse sheaves to mixed Hodge modules. In this case, we lift the constant sheaf $\underline{\Q}_{U_\C}$ to $\underline{\Q}^H_{U_\C}$, get a weight filtration on $u_* \underline{\Q}^H_{U_\C}(n/2)[n]$, apply the functor $\text{rat}$ and tensor to $k$.
\begin{lemma}\label{lem: weight agree with fibers std normal crossings}
The fibers of the weight filtration defined over $\Z_p$ in Equation \eqref{eqn: weight filtration over Zp} taken over $\F_p$ and $\C$ agrees with the weight filtrations previously defined in both contexts in Section \ref{sec: weight filtration}.
\end{lemma}
\begin{proof}
We first show by base change that the fiber of the weight filtration still satisfies the same formulas above, i.e. for $x\in \{s,\eta\}$ we have both $(W_c\mathcal F)_x = v_{c,x!*}u_{c,x*}\underline{k}_{U_x}(n/2)[n]$ and $(\Gr^W_c \mathcal F)_x= \bigoplus_{|I|=c}i_{I,x*}\underline{k}_{(D_{I})_x}(n/2-c)[n-c]$. The second equation is clear from proper base change. For the first, recall that $W_c\mathcal F = \text{Im}({}^p\mathcal H^0(v_{c!}u_{c*}\underline{k}_U[n])\rightarrow {}^p\mathcal H^0(u_{*}\underline{k}_U[n]))(n/2)$. Taking fibers is t-exact with respect to both perverse structures, and $v_{c!},u_{c*},u_*$ all commute with taking fibers, so the first equation is true.

On the special fiber $s=\Spec(\F_p)$, it is clear that the perverse sheaves in $(\Gr^W_c\mathcal F)_x$ has weight $c$ as a complex, so this is indeed the weight filtration. For the generic fiber $\overline\eta = \Spec(\C)$, we can repeat the proof of Lemma \ref{lem: graded pieces in compactification} in the category of mixed Hodge modules, using the analogous properties stated in \cite[Section 4.4, 4.5]{Saito_1990} which includes the localization triangle and purity. This gives a filtration on $u_* \underline{\Q}^H_{U_\C}(n/2)[n]$ that after projecting and tensoring to $k$ agrees with $(W_c\mathcal F)_{\overline \eta}$. It is clear that this is the weight filtration on mixed Hodge modules by the analogue of Lemma \ref{lem: graded pieces in compactification}.
\end{proof}
\subsection{Transporting to $\Sym^n(\A^1)$}Now we return to the setting of Section \ref{sec: compactification of hurwtiz space}, again over $\Z_p$, with $u\colon \Hur^{G,c}_n\hookrightarrow \overline{\Hur^{G,c}_n}$, and we let $\mathcal F=u_*\underline{k}_{\Hur^{G,c}_n}(n/2)[n]$. It is clear by either Equation \eqref{eqn: weight filtration over Zp} or Lemma \ref{lem: graded pieces in compactification} that the construction in the previous subsection glues to a weight filtration on $\overline{\Hur^{G,c}_n}$ which we call $\Gr^W_w\mathcal F$, and Lemma \ref{lem: weight agree with fibers std normal crossings} tells us that it agrees with the weight filtration defined over $\F_p$ and $\C$.

Recall that we want a weight filtration on $j_{n*}\mathcal L_n$. Consider the following commutative diagram
\[\begin{tikzcd}
	{\Hur^{G,c}_n} && {\overline{\Hur^{G,c}_n}} \\
	{\Sym^n_{\neq}(\A^1)} & {\Sym^n(\A^1)} & {\Sym^n(\mathbb P^1).}
	\arrow["u", hook, from=1-1, to=1-3]
	\arrow["\pi"', from=1-1, to=2-1]
	\arrow["\rho"', from=1-3, to=2-3]
	\arrow["j_n", hook, from=2-1, to=2-2]
	\arrow["v", hook, from=2-2, to=2-3]
\end{tikzcd}\]
We have $j_{n*}\mathcal L_n=j_{n*}\pi_* \underline{k}_{\Hur^{G,c}_n}(n/2)[n] = v^*\rho_*u_*\underline{k}_{\Hur^{G,c}_n}(n/2)[n] = v^*\rho_*\mathcal F$, so we need to pushforward then pullback the weight filtration on $\mathcal F$.

The map $\rho_*$ is proper and hence preserves weight, but it is only t-left exact so $\rho_*\Gr^W_w\mathcal F$ may not be perverse. However, we know that $\rho_*\mathcal F= v_*j_{n*}\mathcal L_n$ is perverse as affine open immersions are t-exact. Hence, the spectral sequence associated to the weight filtration
\begin{equation}\label{eqn: spectral sequence rho pushforward}
E_1^{r,s}={}^{p}\mathcal H^{r+s}(\rho_*\Gr^W_{r}\mathcal F)\Rightarrow 
  {}^{p}\mathcal H^{r+s}(\rho_* \mathcal F)
\end{equation}
induces a filtration on the abutment $\rho_*\mathcal F={}^p\!H^0(\rho_*\mathcal F)$. Then, we pull this back to $\Sym^n(\A^1)$, which preserves both weight and perversity, to get the weight filtration on $j_{n*}\mathcal L_n$. It is easy to see, using a similar argument as in Lemma \ref{lem: weight agree with fibers std normal crossings}, that the fibers of this filtration agree with the weight filtration over $\F_p$ and $\C$. We state our results in the following proposition.
\begin{proposition}
There is a weight filtration $\Gr^W_wj_{n*}\mathcal L_n$ defined over $\Z_p$ such that the fibers over $\F_p$ and $\C$ agree with the weight filtrations previously defined in both contexts in Section \ref{sec: weight filtration}.
\end{proposition}
Now that we have constructed the weight filtration over $\Z_p$, we will show that the specialization maps induce the isomorphism of weight spectral sequences in Theorem \ref{thm: comparison theorem}.
\begin{proof}[Proof of Theorem \ref{thm: comparison theorem}]
We start our analysis from the graded pieces $\Gr^W_w\mathcal F$ given etale-locally by Lemma \ref{lem: graded pieces in compactification}. As the divisors $D_I$ are smooth and proper over $\Z_p$, the vanishing cycles $R\Phi \underline k_{D_I}$ relative to $\Z_p$ is zero. Since vanishing cycles commute with proper pushforward and can be checked etale-locally, we have that $R\Phi \rho_*\Gr^W_{-r}\mathcal F=0$. Thus, we have $i_s^*\rho_*\Gr^W_{-r}\mathcal F\xrightarrow\sim R\Psi \rho_*\Gr^W_{-r}$, giving an isomorphism of the first page of the special fiber and nearby cycles of Equation \eqref{eqn: spectral sequence rho pushforward}, which implies that the spectral sequences are isomorphic. Hence, the abutments are also isomorphic and we have $i_s^*\Gr^W_w\rho_*\mathcal F\xrightarrow \sim R\Psi \Gr^W_w\rho_*\mathcal F$. Pulling back by $v^*$, we obtain $i_s^*\Gr^W_wj_{n*}\mathcal L_n\xrightarrow \sim R\Psi \Gr^W_wj_{n*}\mathcal L_n$. Taking $R\Gamma$ on both sides, we get the first page of Equation \eqref{eqn: both spectral sequence} where we use the comparison theorem for etale and analytic cohomology theorem over $\C$ \cite[Section 6.1.2]{BBD}. Hence, we get an isomorphism of spectral sequences as desired.
\end{proof}
\section{Consequences of the weight filtration}
Finally, we give two applications of our weight filtration construction as promised in the introduction. First, we will show concentration of weight as stated in Theorem \ref{thm: concentration of weight in algebra}, and then we will restrict to the case of finite-dimensional Nichols algebra and prove Theorem \ref{thm: finite nichols algebra} with Theorem \ref{thm: S3 transposition} as a special case.
\subsection{Concentration of weight in algebras}
We give the proof of Theorem \ref{thm: concentration of weight in algebra} in this subsection. Let $V\in \mathcal V$ be a braided vector space, and as usual let $\mathfrak A=\mathfrak A(V)$, $T=T(V^*)$, $\mathfrak B=\mathfrak B(V)$. Note that for this subsection we do not require the finite monodromy condition, because we do not need the geometric weight filtration, and we will not use that the algebraic weight filtration is a bialgebra filtration. It suffices to prove the same statement for the tensor algebra $T$ but with a negative sign on the grading, i.e.
\begin{equation}\label{eqn: concentration in tensor algebra}
\frac{\dim W_{-c^+n}T_n}{\dim T_n}\rightarrow 0, \qquad \frac{\dim W_{-c^-n}T_n}{\dim T_n}\rightarrow 1.
\end{equation}

We make some simple observations. It is clear from definition that the algebraic weight filtration on $T_n$ is concentrated in weights $-n+1\le w\le 0$. Furthermore, every primitive element causes the filtration on $T$ to go down by one, so intuitively, the more primitive elements $T$ has, the lower the filtration on $T$. However, we have no control of how many primitive elements there are, so in this proof we basically only use that concatenation preserves the weight filtration.
\begin{proof}[Proof of Theorem \ref{thm: concentration of weight in algebra}]
For each algebraic degree $n$, we define the average weight of $T_n$ to be
$$a(n)=\frac{1}{\dim (T_n)}\sum_w w\cdot \dim(\Gr_wT_n).$$
Note that $a(n)\le 0$ since the weight of the tensor algebra is always $\le 0$.

Take a basis of words $x_1,\ldots, x_d$ of $T_n$ compatible with the weight filtration, in other words, we require that for every $w$ that a subset of these elements are a basis for $W_wT_n$. Let $w_i$ be the minimum integer such that $x_i\in W_{w_i}T_n$. Suppose we concatenated $m$ of these words to form $x=x_{i_1}\cdots x_{i_m}$. Then, because the weight filtration is an algebra filtration and multiplication is given by concatenation, we must have $x\in W_wT_{nm}$ for $w=w_{i_1}+\cdots +w_{i_m}$. By considering all possible length $m$ concatenations, these form a basis of $T_{nm}$, and by the above argument the average weight of this basis satisfies 
\begin{equation}\label{eqn: ineq 1}
a(nm)\le m\cdot a(n).
\end{equation}

Furthermore, we can instead concatenate a basis of $T_m$ with a basis of $T_1$ instead, and again since concatenation respects the weight filtration we must have 
\begin{equation}\label{eqn: ineq 2}
a(n+1)\le a(n),
\end{equation}
i.e. that $a(n)$ is non-increasing. 

It is an exercise in analysis to conclude that the limit $\lim_{n\rightarrow \infty} a(n)/n$ exists from these two inequalities, and is equal to $\inf_n a(n)/n$. We explain this as follows. Let $-c=\inf_n a(n)/n$, this is well defined with $0\le c\le 1$ since the weight filtration in $T_n$ is concentrated in weights $-n+1$ to $0$. For every sufficiently small $\epsilon>0$ there is some $n$ where $a(n)/n\le -c+\epsilon$. Now, for any integer $N>n/\epsilon$ we will have $$\frac{a(N)}{N}\le \frac{a(n\lfloor \frac N n\rfloor)}{N}\le  (1-\epsilon) \frac{a(n\lfloor \frac{n'}{n}\rfloor)}{n\lfloor \frac{n'}{n}\rfloor} \le (1-\epsilon) \frac{a(n)}{n}  \le (1-\epsilon)(-c+\epsilon),$$
where we used Equation \eqref{eqn: ineq 2} in the first inequality and Equation \eqref{eqn: ineq 1} in the third inequality, proving that the limit is indeed $-c$.

Let choose this $0\le c\le 1$ in the statement so that $-c=\inf_n a(n)/n=\lim_{n\rightarrow \infty} a(n)/n$. If $c=0$, then we must have $a(n)=0$ for all $n$, so there are no primitive elements in the tensor algebra. This corresponds to the case when $\mathfrak A(V)=\mathfrak B(V)$ is generated in degree $1$.

We first prove Equation \eqref{eqn: concentration in tensor algebra} for $c^-$. Again, for any sufficiently small $\epsilon >0$, take some $n$ where $a(n)/n\le -c+\epsilon$, and a basis $x_1,\ldots ,x_d$ of $T_n$ compatible with the weight filtration like above. Consider i.i.d. random variables $Y_1,\ldots, Y_m$ each uniformly chosen from $\{1,\ldots, d\}$. The concatenation $x_{Y_1}\cdots x_{Y_m}$ is in $W_wT_{nm}$ where $w=w_{Y_1}+\cdots +w_{Y_m}$ is the sum of random variables. By the weak law of large numbers, $$\mathbb P\left(\left|\frac{w}{m}-a(n)\right|<\epsilon \right)\rightarrow 1$$
as $m\rightarrow \infty$. This means that for all sufficiently large $m$, almost all concatenations lie inside $W_wT_{nm}$ for $w=m(a(n)+\epsilon)\le nm(-c+2\epsilon)$, which finishes the argument for $\dim (W_{-c^-n}T_n)/\dim(T_n)\rightarrow 1$ when we take $\epsilon\rightarrow 0$.

Now we prove the statement for $c^+$. For every sufficiently small $\epsilon>0$, we can take $c^-=c-\epsilon$ and deduce that $\dim (W_{-c^-n}T_n)\ge (1-\epsilon)\dim (T_n)$ for sufficiently large $n$ by the above argument. If we let $\dim (W_{-c^+n}T_n)=d_n\dim (T_n)$, then we have
$$-c\le \frac{a(n)}{n}\le -c^+d_n-c^-(1-\epsilon-d_n)$$
and rearranging we obtain
$$d_n\le \frac{c-(c-\epsilon)(1-\epsilon)}{c^+-c+\epsilon}$$
so we are done after taking $\epsilon \rightarrow 0$.
\end{proof}
We remark that our proof does not give any effective rate of convergence, as without further understanding of the primitive elements, we do not know how fast $a(n)/n$ converges to the limit. Also, as discussed in the introduction, we do not know a way to prove the analogous statement for cohomology.

\subsubsection{Examples} Let us discuss this in the context of our three examples in Section \ref{sec: weight examples}. For the one dimensional vector space $V=k$ with braiding given by a $m$-th root of unity $q\in k^\times$, it is clear from Figure \ref{fig5} that $c=1/m$. 

For the usual shuffle algebra, we will prove that $c=1$. Recall that in this case we showed that the weight filtration is the same as the word length filtration, which is the same as the PBW filtration on $U(L(V^*))$ by Proposition \ref{prop: PBW word length agree}. We proved that $\dim(W_{-n+1}T_n)=\dim(L(V^*))$ is the number of Lyndon words of length $n$ over an alphabet of size $d=\dim(V)$, which is given by the formula $I(n)=\frac{1}{n}\sum_{k\mid n} \mu(k)d^{n/k}\approx d^n/n$. However, this is not yet enough to show $c=1$, and we need something stronger. Note that we have $\dim(W_{-n+m}T_n)=\dim(F_{m}U(L(V^*))_n)$ where $F_m$ is the PBW filtration, and this is simply $\sum_{k\le m} M(k,n)$ where $M(k,n)$ is the number of multisets of $k$ Lyndon words whose lengths sum to $n$. One can construct a generating function using $I(n)$ to count $M(k,n)$. It turns out that if $d=q$ is a prime power, then the same formula for $I(n)$ also counts the number of irreducible polynomials of degree $n$ over $\F_q$ and here $M(k,n)$ is the number of polynomials of degree $n$ with $k$ irreducible factors. The problem is well studied in this case, \cite[Corollary 1]{Flajolet_Soria_1990} tells us that the number of irreducible factors of a random polynomial is Gaussian with mean $\log n$ and standard deviation $\sqrt{\log n}$. The generating function technique used there works when $d$ is not prime, so we conclude that when $m=a\log n$ for $a>1$, then $\dim(W_{-n+m}T_n)=\sum_{k\le m} M(k,n)\rightarrow 1$, so we have $c=1$.

Lastly, for the braided vector space associated to the rack of transpositions in $S_3$, we can compute values of $a(n)/n$ which are upper bounds for $-c$. When the cocycle $x=1$, we see that $a(5)/5=-0.0148$ is very small, as there are very little primitive elements up to $n=5$. Meanwhile, for $x=-1$ we have $a(5)/5=-0.415$. This suggests that the cohomological weights of $\Hur^{G,c}_n$ are on average more than $0.415n$ lower than one might expect from cohomological degree.

\subsection{Finite Nichols algebras} Our goal here is to prove Theorem \ref{thm: finite nichols algebra} and \ref{thm: S3 transposition}, following the proof sketch given in the introduction. Consider a Hurwitz space $\Hur^{G,c}_n$ with the associated braided vector space $V_{\epsilon}=V(c,-1)\in \mathcal {YD}_G$ which local systems necessarily has finite monodromy. Let $\mathfrak A=\mathfrak A(V_\epsilon)$, $\mathfrak B=\mathfrak B(V_\epsilon)$, $\mathfrak A^{gr}$ be the associated graded of $\mathfrak A$ with respect to the weight filtration, and also define $\mathfrak C=\mathfrak A^{gr}\Box_{\mathfrak B}k$ to be the cotensor product. We note that $\mathfrak A^{gr}$ has an algebra grading $n$ as well as a weight grading $w$, and likewise the same is true for the subalgebra $\mathfrak B = W_0\mathfrak A \subseteq \mathfrak A^{gr}$ and the cotensor product $\mathfrak C$.

Here, we will use the grading convention $\Tor^A_{n-i,n}(k,k)$ and $\Ext_A^{n-i,n}(k,k)$ where we $n$ is the algebraic degree and $i$ the cohomological degree, and the reason for this choice is that $i$ will correspond to the cohomology degree of Hurwitz space, as we saw in Theorem \ref{thm: relate weight decompositions}. If the bialgebra $A$ also has a weight grading, then $\Tor$ and $\Ext$ also have a weight grading. In this case, recall that we keep the weight grading the same without flipping the sign when we take the dual from $\Tor^{A}_{n-i,n}(k,k)$ to $\Ext_{A}^{n-i,n}(k,k)$.

The next two lemmas gives us inequalities between different gradings on $\mathfrak B$ and $\mathfrak C$ assuming finiteness conditions on the Nichols algebras.
\begin{lemma}\label{lem: B finite dim}
If $\mathfrak B$ is finite-dimensional, then on $\mathfrak C$ the weight and algebraic gradings satisfy $w\ge \delta n$ for some $\delta>0$.
\end{lemma}
\begin{proof}
Suppose that the maximal algebraic degree of $\mathfrak B$ is $m$, then we will prove the statement for $\delta=1/(m+1)$. Since $W_0\mathfrak A=\mathfrak B$, this implies that $W_0 \mathfrak A_{m+1}=0$, so dually $W_{-1}T_{m+1}=T_{m+1}$. As concatenation respects the weight filtration, this implies that $W_{-k}T_{k(m+1)}=T_{k(m+1)}$ so $W_{k-1}\mathfrak A_{k(m+1)}=0$. By \cite[Theorem 4.1]{ETW17}, we have an isomorphism $\mathfrak C\otimes \mathfrak B\xrightarrow \sim \mathfrak A^{gr}$, and this preserves the weight, cohomological and algebraic grading. Suppose the contrary that some element of $\mathfrak C_n$ has weight $w<\delta n$, so $w(m+1)<n$. Then, we can tensor this with an element in $\mathfrak B_m$ that has weight $0$, obtaining an element in $\mathfrak A^{gr}$ that has weight $w$ and degree greater than $w(m+1)+m$, but this is impossible from the discussion above.
\end{proof}
\begin{lemma}\label{lem: Ext B finite gen}
If $\Ext_{\mathfrak B}(k,k)$ is finitely generated, then its cohomological and algebraic grading satisfy $i\le cn$ for some $c<1$.
\end{lemma}
\begin{proof}
This is almost obvious, let its generators be in cohomological degree $i_j$ and algebraic degree $n_j$. Clearly, $n_j-i_j>0$ by our indexing convention, so there is a maximum value of $i_j/n_j$ which we call $c$. Note that $c<1$, and all elements in $\Ext_{\mathfrak B}(k,k)$ satisfy $i\le cn$ since the inequality is true for the generators.
\end{proof}

Now, we combine these two inequalities to complete the proof of Theorem \ref{thm: finite nichols algebra}.
\begin{proof}[Proof of Theorem \ref{thm: finite nichols algebra}] Let us choose the constants $c$ and $\delta$ as in Lemma \ref{lem: B finite dim} and \ref{lem: Ext B finite gen}, and choose $\epsilon=\delta/(1-c)$. Suppose that an element $x\in \Ext_{\mathfrak A}^{n-i,n}(k,k)$ has weight $w$. As in the statement of the theorem, we suppose that $i>cn$. The triple of spectral sequence in Equation \eqref{eqn: triple spectral sequences} respects the three gradings, so $x$ must come from a linear combination of elements of the form $u\otimes v\in \Ext_{\mathfrak B}(k,k)\otimes \Ext_{\mathfrak C}(k,k)$. Suppose $u$ and $v$ have gradings $w_1,i_1,n_1$ and $w_2,i_2,n_2$ respectively, then these gradings must add up to $w,i,n$. However, Lemma \ref{lem: B finite dim} and \ref{lem: Ext B finite gen} tell us that $i_1\le cn_1$ and $w_2\ge \delta n_2$ since we assumed the finiteness conditions in the statement of the theorem. Furthermore, we have the trivial bound $n_2\ge i_2$. Combining these inequalities, we get
$$i_2 = i-i_1\ge i-cn_1= i-cn+cn_2 \ge i-cn+ci_2$$
so we have
$i_2\ge \frac{i-cn}{1-c}$ and thus $w\ge w_2\ge \delta i_2\ge \epsilon (i-cn)$. Finally, Theorem \ref{thm: relate weight decompositions} tells us that the weight grading $w$ on algebras correspond to a Frobenius weight of $2n-i-w$, and this finishes the proof.
\end{proof}
Lastly, we specialize this to the case where $G=S_3$ and $c$ is the conjugacy class of transpositions.
\begin{proof}[Proof of Theorem \ref{thm: S3 transposition}]
The explicit constants in Theorem \ref{thm: S3 transposition} follow as  now $\mathfrak B$ is the Fomin-Kirillov algebra $\mathcal{FK}_3$ with maximal algebraic degree $4$ which gives $\delta=1/5$ in the proof of Lemma \ref{lem: B finite dim}, and $\Ext_{\mathfrak B}(k,k)$ is generated in degrees $(n-i,n)=(1,1),(1,1),(1,1),(4,6)$ by \cite{FK3}, so we can take $c=1/3$ in Lemma \ref{lem: Ext B finite gen} and this gives $\epsilon= 3/10$, yielding the desired inequality.

The floor in the first part of Theorem \ref{thm: S3 transposition} comes from a more careful analysis of the inequalities, because if the weight were to be as expected then we must have $w_2=n_2=i_2=0$, i.e. the cohomology comes purely from $\Ext_{\mathfrak B}(k,k)$, but from the generators we can tell that $i_1\le 2\lfloor \frac{n_1}6\rfloor$ so we must have $i\le 2\lfloor \frac{n}6\rfloor$.
\end{proof}

\newpage
\printbibliography
\end{document}